\documentclass[11pt, reqno]{amsart}

\usepackage{textcmds} % amsrefs needs this, but it has to be loaded early to avoid re-defs.
\usepackage{amsmath, amssymb, amsfonts, amstext, amsthm, amscd, mathrsfs, comment}
\usepackage[all]{xy}
\usepackage[usenames,dvipsnames]{color}
\usepackage{enumitem}
\usepackage{bm}
\usepackage{subfigure}
\usepackage{array}
\usepackage{graphics,graphicx}
\usepackage{colortbl,mathtools}
\usepackage{pinlabel}

\usepackage[a4paper,inner=3.6cm,outer=3.6cm,top=3.7cm,bottom=3.7cm]{geometry}

\let\oldtocsection=\tocsection

\let\oldtocsubsection=\tocsubsection

\setlist[itemize]{leftmargin=2em}
\setlist[enumerate]{leftmargin=2em}
\renewcommand{\tocsection}[2]{\hspace{0em}\oldtocsection{#1}{#2}}
\renewcommand{\tocsubsection}[2]{\hspace{1em}\oldtocsubsection{#1}{#2}}

\usepackage[bookmarks=false,backref,
colorlinks=true,
linkcolor=black,
citecolor=red,urlcolor=blue,citebordercolor={0 0 1},urlbordercolor={0 0 1},linkbordercolor={0 0 1}]{hyperref} %needs to be loaded after most things

\newcommand{\st}{{\bf\textcolor{red}{*}}}

\newcommand{\C}{\mathbb{C}}

\newcommand{\R}{\mathbb{R}}

\newcommand{\Z}{\mathbb{Z}}

\newcommand{\M}{\mathcal{M}}

\newcommand{\B}{\mathscr{B}}

\newcommand{\CP}{\mathbb{C}P}

\newcommand{\id}{\mathrm{Id}}

\newcommand{\pr}{\mathrm{pr}}
\newcommand{\Aut}{\mathrm{Aut}}

\newcommand{\OP}{\operatorname}

\newcommand{\del}{\partial}
\renewcommand{\st}{\star}

\renewcommand{\P}{\mathcal{P}}

\newcommand{\CO}{\mathcal{CO}}

\renewcommand{\t}{\tilde}

\newcommand{\F}{\mathscr{F}}
\def\co{\colon\thinspace}
\newcommand{\bb}[1]{\mathbf{#1}}

\renewcommand{\SS}{\mathscr{S}}

\newcommand{\kk}{\mathbf{k}}

\newcommand{\bL}{\mathbf{L}}

\newcommand{\uu}{\mathbf{u}}
\newcommand{\oo}{\circ}
\newcommand{\bT}{\mathbf{T}}

\numberwithin{equation}{section}

\newtheorem{thm}{Theorem}[section]
\newtheorem*{thm*}{Theorem}
\newtheorem{prp}[thm]{Proposition}
\newtheorem{lem}[thm]{Lemma}
\newtheorem{cor}[thm]{Corollary}

\newtheorem{que}[thm]{Question}

\theoremstyle{definition}
\newtheorem{dfn}[thm]{Definition}

\theoremstyle{remark}
\newtheorem{rmk}{Remark}[section]
\newtheorem{ex}{Example}[section]

\begingroup 
\makeatletter 
\@for\theoremstyle:=definition,remark,plain,TheoremNum\do{% 
\expandafter\g@addto@macro\csname th@\theoremstyle\endcsname{% 
\addtolength\thm@preskip\parskip 
}% 
} 
\endgroup 

\title[Refined disk potentials]
{Refined disk potentials
	for immersed \\ Lagrangian surfaces}
\author{Georgios Dimitroglou Rizell}
\address[GDR]{Department of Mathematics\\
Uppsala University\\
Box 480\\
SE-751 06 Uppsala\\
Sweden}
\email{georgios.dimitroglou@math.uu.se}
\author{Tobias Ekholm}
\address[TE]{Department of Mathematics\\
Uppsala University\\
Box 480\\
SE-751 06 Uppsala\\
Sweden \and 
 Institut Mittag-Leffler, Aurav\"agen 17, 182 60 Djursholm, Sweden 
}
\email{tobias.ekholm@math.uu.se, ekholm@mittag-leffler.se}
\author{Dmitry Tonkonog}
\address[DT]{University of California, Berkeley\\
	United States}
\email{dtonkonog@berkeley.edu}
\thanks{
GDR is supported by the grant KAW 2016.0198 from the Knut and Alice Wallenberg Foundation.
TE is supported by the Knut and Alice Wallenberg Foundation as a Wallenberg Scholar and by the Swedish Research Council.
	DT is partially supported by the Simons
	Foundation grant \#385573, Simons Collaboration on Homological Mirror Symmetry.}

\begin{document}

\begin{abstract}
We define a refined Gromov-Witten disk potential of self-transverse monotone immersed Lagrangian surfaces in a symplectic 4-manifold as an element in a capped version of the Chekanov--Eliashberg dg-algebra of the singularity links of the double points 
(a collection of Legendrian Hopf links). We give a surgery formula that expresses the potential after smoothing a double point. 

We study refined potentials of monotone immersed Lagrangian spheres in the complex projective plane and find monotone spheres that cannot be displaced from complex lines and conics by symplectomorphisms. We also derive general restrictions on sphere potentials using Legendrian lifts to the contact 5-sphere. 
\end{abstract}

\maketitle

\setcounter{tocdepth}{1} 
\tableofcontents

\section{Introduction}
\label{sec:intro}

Recall that a symplectic manifold $(X,\omega)$ is called monotone if the cohomology classes $c_1(X)$ and $[\omega]$ are positively proportional in $H^2(X;\C)$.
A smooth Lagrangian submanifold $L\subset X$ is called monotone if its Maslov class $\mu$ and $[\omega]$ are positively proportional in $H^2(X,L;\C)$. 

Let $L\subset X$ be an embedded monotone (oriented spin)  $n$-dimensional Lagrangian submanifold and let $x_{1},\dots, x_{k}$ be an integral basis of $H_{1}(L;\C)$, $\dim(H_{1}(L;\C))=k$. The \emph{Landau-Ginzburg (LG) potential}, 
$$
W_{L}\in\C[x_1^{\pm 1},\ldots,x_k^{\pm 1}]\cong \C[H_1(L;\C)],
$$
is defined as follows. Recall that the formal dimension of the space of holomorphic disks $u\colon (D,\del D)\to (X,L)$ is $(n-3)+\mu(u)$, where $\mu(u)$ is the Maslov index of $u$. In particular, the space of holomorphic disks $u$ with $\mu(u)=2$ that pass through a fixed reference point $\zeta\in L$ is an oriented compact 0-manifold for generic data. The coefficient in $W_{L}$ of the monomial $x_{1}^{a_{1}}\dots x_{k}^{a_{k}}$ equals the count of Maslov index 2 disks $u$ such that the homology class $[\del u]\coloneqq[u(\partial D)]$ satisfies $[\del u]=a_{1}x_{1}+\dots+a_{k}x_{k}\in H_1(L;\C)$.  

In this paper we generalize the notion of LG potential to immersed Lagrangian surfaces.  Our generalization gives counts of disks that are more refined than the counterparts in immersed Floer cohomology of the counts in homology classes discussed above. We count in what can be described as a counterpart of the homology of the loop space of the immersion rather than in the homology of the immersion. Alternatively, the way we count the disks on the immersion, they can be viewed as deforming the endomorphism algebra of certain generators of the wrapped Fukaya category of a Weinstein neighborhood of the immersion rather than the corresponding endomorphism algebra of the Lagrangian self. See Section \ref{sec:generalizations} for a brief discussion of this perspective and further generalizations.

If $f\colon Q\to X$ is a self-transverse Lagrangian immersion and $L=f(Q)$ is its image, then there is a Maslov index $\mu(u)$ for disks $u\colon (D,\del D)\to(X,L)$ which are allowed to have boundary corners at the double points of $L$, see Section~\ref{ssec:maslovindex}. The formal dimension of the moduli space of holomorphic disks with corners in the homotopy class $u$  equals $(n-3)+\mu(u)$, just like in the smooth case. In particular, the dimension of the space of Maslov index~2 disks with boundary that passes through a fixed (smooth) reference point $\zeta\in L$ is again zero. 

In the immersed case the count of Maslov index 2 disks is however not invariant under deformations.
The reason is that moduli spaces of holomorphic disks with corners have boundary strata that do not appear in the embedded case. In particular, for generic 1-dimensional moduli spaces of disks, new boundary phenomena appear at instances when the boundary of the disk crosses through the double point, see Figure~\ref{fig:bifurcation}. In a sense, the refined disk potential we define below can be characterized as the curve count that retains maximal information about the disks and still remains invariant under such crossings.  

\begin{figure}[h]
	\includegraphics[]{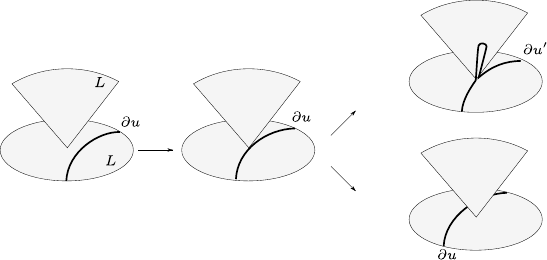}
	\caption{For immersed Lagrangians, a disk with boundary passing through a double point appears  both as a boundary point in a moduli space corresponding to the collision of two boundary corners, and as an interior point of another moduli space of disks with two fewer boundary corners.}
	\label{fig:bifurcation}
\end{figure}

\subsection{Refined disk potentials}\label{sec:genpot}
From now on assume that $n=2$ and let $f\co Q\to X$ be a monotone immersion, see Section~\ref{ssec:monotoneandexact}. To define the refined potential, we puncture $X$ at the double points $\zeta_{1},\dots,\zeta_{m}$ of $L$ (and study holomorphic curves with punctures rather than corners). This gives a monotone symplectic manifold $X^{\circ}$ with negative ends modelled on the negative half of the symplectization of $m$ copies $S_{1}^{3},\dots, S_{m}^{3}$ of the standard contact 3-sphere. Furthermore, $L$ intersects the negative ends along cones $(-\infty,0]\times\Lambda_{j}$ on Legendrian Hopf links $\Lambda_{j}\subset S^{3}_{j}$.

Let $\Lambda=\Lambda_{1}\sqcup\dots\sqcup\Lambda_{m}$ and write $L^{\circ}\subset X^{\circ}$ for the embedded Lagrangian with negative end $\Lambda$ obtained by removing the double points from $L$. Pick basepoints $\st_{j}$, $j=1,\dots,r$, on the connected components $L_{j}^{\circ}$, $j=1,\dots,r$ of $L^{\circ}$, and reference paths from the endpoints of all Reeb chords on $\Lambda$ to these basepoints. The Chekanov--Eliashberg dg-algebra with coefficients in the group algebras $\C[\pi_{1}(L^{\circ}_{j},\star_{j})]$ is the algebra of linear combinations of monomials of homotopy classes and Reeb chords of the form $\gamma_{0}c_{0}\gamma_{1}\dots \gamma_{m-1}c_{m}\gamma_{m}$, where $\gamma_{j}$ is a homotopy class of based loops on the component where the Reeb chord $c_{j}$ begins and $c_{j+1}$ ends. The product is induced by concatenation of loops. The differential counts holomorphic disks in the symplectization $\R\times S^{3}_{j}$ with boundary in $\R\times\Lambda_{j}$. The grading of any homotopy class is zero and the grading of Reeb chords is given by the Conley-Zehnder index relative Maslov potentials on $L_{j}^{\circ}$ that all equal zero. The differential decreases the grading by 1. We denote this Chekanov--Eliashberg dg-algebra with coefficients in the group algebra $CE(\Lambda;\C[\pi_{1}(L^{\circ})])$. It is then a non-commutative algebra and a module over the algebra $\C[e_{1},\dots,e_{r}]$, where $e_{i}e_{i}=e_{i}$ and $e_{i}e_{j}=0$ if $i\ne j$. Here the idempotent $e_{j}$ can be identified with the one-letter word of the constant loop in $L_{j}^{\circ}$. See Section~\ref{sec:dga} for a more detailed definition, also compare \cite{EL}. The refined potential of $L$ is an element in $CE(\Lambda;\C[\pi_{1}(L^{\circ})])$.

Consider almost complex structures $J$ on $X^{\circ}$ that are cylindrical at the negative ends of $X^\circ$ with respect to a generic contact form on the boundary-at-infinity of $X^\circ$ (the union of the contact spheres $S_{j}^{3}$).  Fix an orientation and a spin structure on $L^{\circ}$. Choose one of the basepoints $\zeta=\st_j$ and let $\mathcal{M}(\zeta)$ be the moduli space of Maslov index~2 holomorphic disks with boundary on $L^{\circ}$ that pass through $\zeta$, and with  an arbitrary number of boundary punctures asymptotic to Reeb chords of $\Lambda$ in the negative end of $X^\circ$. If $u\in\mathcal{M}(\zeta)$, then its boundary $\partial u$ defines a monomial $[\del u]$ in $CE(\Lambda;\C[\pi_{1}(L^{\circ})])$ as follows. Starting at the basepoint $\zeta$ and following the boundary $\partial u$, we read off boundary arcs between adjacent punctures (completed to loops using reference paths) and Reeb chords at the punctures. This gives the monomial $[\del u]$; see Section~\ref{subsec:ce_pot} and Figure~\ref{fig:pot_ce_ref}.

\begin{figure}[h]
	\includegraphics[]{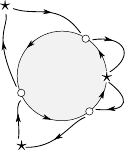}
	\caption{Holomorphic disk with punctures (white dots) passing through a basepoint with boundary arcs completed to loops (all basepoints are denoted $\st$).}
	\label{fig:pot_ce_ref}
\end{figure}

\begin{thm}
	\label{t:general}
	For generic $J$, $\mathcal{M}(\zeta)$ is a compact oriented $0$-manifold. Define
	\[ 
	W_{L}(\zeta)=\sum_{u\in\mathcal{M}(\zeta)}(-1)^{\epsilon (u)}[\del u] \ \in\ CE(\Lambda;\C[\pi_{1}(L^{\circ})]),
	\]
	(recall that $\zeta=\star_{j}$ is one of the base points in the definition of the Chekanov--Eliashberg algebra) where $(-1)^{\epsilon(u)}$ is the orientation sign of the solution $u$ induced by the spin structure on $L^{\circ}$. Then $W_{L}(\zeta)$ is a degree~0 cycle in $CE(\Lambda;\C[\pi_{1}(L^{\circ})])$. Its homology class is independent of the choice of almost complex structure $J$, and of the Lagrangian immersion $f$ up to ambient Hamiltonian isotopy. 
\end{thm}

Theorem \ref{t:general} will be restated and proved in Proposition~\ref{th:potdga}.
The algebra $CE(\Lambda;\C[\pi_{1}(L^{\circ})])$ can be computed explicitly in terms of generators and relations. By a generalization of the surgery result in \cite{EL}, see Section \ref{sec:generalizations} for a discussion, it is isomorphic to 
$$
\mathrm{End}\left(\oplus_{j=1}^r\, T^*_{\st_j}L\right),
$$
the wrapped Floer cohomology endomorphism algebra of the union of the cotangent fibers to the connected components of $L^\circ$ in the Weinstein neighborhood of the immersed Lagrangian $L$ (constructed for example as a self-plumbing of $T^{\ast}Q$).
This algebra can also be expressed in other ways. As shown by Etg\"{u}--Lekili in \cite{EL17}, decomposing $L^{\circ}$ into handles and applying the Legendrian surgery formula to the result provides a quasi-isomorphism with the \emph{derived} higher genus multiplicative pre-projective algebras associated to a certain quiver. 

Multiplicative pre-projective algebras (PPA) of a quiver were introduced by Crawley-Boevey--Shaw \cite{CBS06} and later generalized by Bezrukavnikov--Kapranov \cite{BK16} to a higher genus version. In line with the results mentioned above, we prove in Section~\ref{sec:preprojective} that the algebras $H_0CE(\Lambda,\C[\pi_{1}(L^{\circ})])$ used here indeed are isomorphic to the PPA of a quiver obtained from the plumbing graph of $L^\circ$. Furthermore, in the special case of a sphere with a single positive self-intersection we also establish a quasi-isomorphism of the entire dg-algebra and the PPA, see Theorem \ref{thm:homology}.

\subsection{Immersed spheres and a Lagrangian surgery formula}\label{sec:immsph}
Let $X$ be a monotone symplectic 4-manifold, and consider monotone Lagrangian sphere immersions $f\colon S^{2}\to X$ with a single transverse double point $x\in X$ of positive self-intersection. In this case the capped Chekanov--Eliashberg algebra $CE(\Lambda;\C[\pi_{1}(L^{\circ})])$ discussed above can be described as follows. First, $L^{\circ}\approx S^{1}\times\R$ so $\C[\pi_{1}(L^{\circ})]=\C[t^{\pm 1}]$. Next, $\Lambda$ is the standard Legendrian Hopf link in $S^{3}$ and its dg-algebra can be computed in a Darboux chart, see Section~\ref{sec:monotone}  or \cite[Section 1.4]{EL}, as follows. The Legendrian link $\Lambda$ has four Reeb chords: $a_{1}$, $a_{2}$ of degree $1$ and $p$, $q$ of degree $0$. The differential is (see Figure~\ref{fig:hopf}):
\[ 
d a_{1}= 1-t-pq, \quad da_{2}=1-t-qp,
\]
where we use the bounding spin structure on the components of $\Lambda$ and pick appropriate signs of the generators $p$ and $q$.
It follows, {see Theorem \ref{thm:homology}}, that the homology of the capped Chekanov--Eliashberg algebra is supported in degree~0 and is isomorphic to the following commutative algebra.{
\[
HCE(\Lambda;\C[\pi_{1}(L^{\circ})]) \ = \ H_0CE(\Lambda;\C[\pi_{1}(L^{\circ})]) \ \cong \ 
A \ \coloneqq \ \C[p,q,t^{\pm 1}]/(1-t-pq),
\]}
where we write $1$ for the idempotent associated to the connected component of $L^{\circ}$ ($e_{1}$ in the notation of Section \ref{sec:genpot}).

It follows that the refined potential of a monotone immersed Lagrangian sphere $L\subset X$ can be written in the following form: 
\begin{equation}\label{eq:suppot}
W_{L}(p,q,t)= \sum_{u\in \M(\zeta)} (-1)^{\epsilon(u)}\, p^{a(u)}q^{b(u)}t^{c(u)} \ \in \ A.
\end{equation}
It counts punctured Maslov index~2 disks $u$ that pass through a base-point $\zeta\in L$, have $a(u)\ge 0$ punctures asymptotic to $p$,  $b(u)\ge 0$ punctures asymptotic to $q$, and such that $c(u)\in \Z$ is the total number of times $\del u$ wraps around the non-trivial loop of the cylinder  $L^\circ\cong S^1\times \R$.

Consider the embedded monotone tori $L_{p}, L_{q}\subset X$ obtained by the two zero-area surgeries on $L$, see Section \ref{sec:LagrangeSurgery} or \cite[Section~2, Step~3]{Va17}. The two surgeries correspond to two exact Lagrangian fillings of the Hopf link, and rigid holomorphic disks with boundary on these fillings can be calculated via flow trees as in \cite{EHK}. This leads to the following result relating potentials of immersed spheres and nearby tori.

\begin{thm}
	\label{thm:toric_pot1}
	Fix a spin structure on $ L_p$ that does not bound. 
	There exists a basis $(x,y)$ of $H_1( L_{p})$ such that the LG potential $W_{ L_{p}}(x,y)$ is obtained from the potential $W_L(p,q,t)$ of the immersed sphere by making the following substitutions:
	\begin{equation}
	\label{eq:whitney_to_torus_11}
	\begin{array}{l}
	p\mapsto x\\
	q\mapsto (1+y)x^{-1}\\
	t\mapsto -y.
	\end{array}
	\end{equation}
	Similarly, there exists a basis $(x,y)$ of $H_{1}(L_{q})$ such that the LG potential $W_{ L_{q}}(x,y)$ is obtained from $W_L(p,q,t)$ by making the substitution:
	\begin{equation}
	\begin{array}{l}
	\label{eq:whitney_to_torus_12}
	p\mapsto (1+y)x,\\
	q\mapsto x^{-1},\\
	t\mapsto -y.
	\end{array}
	\end{equation}
\end{thm}

Theorem \ref{thm:toric_pot1} is proved at the end of Section \ref{sec:CurvesOnFilling}.
\begin{rmk}
	When $L$ is a general monotone immersed Lagrangian surface, one can choose any double point of $L$ and perform  surgery on that double point with zero area parameter. This gives two monotone immersions $ L_{p}, L_{q}$ with one less double point. Our proof of the surgery formula is local and applies in this more general context relating the potentials of $L$, $ L_p$ and $ L_q$, see Proposition~\ref{prp:surg}. 
\end{rmk}

\begin{rmk} Theorem~\ref{thm:toric_pot1} implies the wall-crossing formula for the mutation of the Lagrangian torus $ L_{p}$ into $ L_{q}$ which was predicted in the context of symplectic topology by Auroux~\cite{Au07,Au09} and proved in \cite{PT17}. 
\end{rmk}

\begin{rmk} The formulas in Theorem \ref{thm:toric_pot1} are closely related to the surgery result of Fukaya, Oh, Ohta and Ono \cite{FO306}. If the immersion $L$ is $n$-dimensional, their result states that holomorphic curves with $p$-corners smooth bijectively to holomorphic curves with boundary on $ L_p$, whereas curves with $q$-corners give rise to $S^{n-2}$-families of curves with boundary on $ L_p$ (which in particular reflects the change in the Maslov index).
	In the case $n=2$, the sphere $S^0$ is just a pair of points and this corresponds to the two summands in the expression $q\mapsto (1+y)x^{-1}$. Another well known appearance of a related surgery formula is Seidel's exact sequence for Dehn twists \cite{Sei03_LES, SeiBook08}. Roughly speaking, in the language of Theorem~\ref{thm:toric_pot1}, it only concerns curves with $p$-corners, without making any statement about curves with at least one $q$-corner.
\end{rmk}

\subsection{Restrictions and applications}
In Section~\ref{sec:application} we apply refined potentials to the symplectic topology of monotone immersed spheres $L\subset \CP^2$ with one transverse double point. {In Section \ref{ssec:basicmonspheres} we show that the double point of such $L$ must be positive and that it admits an embedded Legendrian lift into the pre-quantization space $S^{5}\to\C P^{2}$.}

Using the relation $pq=1-t$, the potential of an immersed sphere can be written in normal form,
$$
W_L(p,q,t)=\sum_{k\ge 0} p^{1+3k}\cdot P_{1+3k}(t)+\sum_{k\ge 0} q^{2+3k}\cdot Q_{2+3k}(t),
$$
where $P_{1+3k}$, $Q_{2+3k}\in\C[t^{\pm 1}]$ are Laurent polynomials in $t$. Here $P_{1}(t)$ can be interpreted directly as the count of Maslov index 2 disks with only one negative puncture (one corner). 

We first discuss a general restriction on the refined potentials. Interpreting the count for $P_{1}(t)$ in terms of the Chekanov--Eliashberg algebra of the Legendrian lift of $L$ to the standard contact $S^{5}$ gives the following.

\begin{thm}
	\label{th:cp2_p1}
	For any monotone Lagrangian immersion $f\colon S^{2}\to\C P^{2}$ with a single double point (which is then positive) if $L=f(S^{2})$ and we expand $W_{L}$ as above then $P_1(1)=\pm 1$.
\end{thm}

We next  construct infinitely many pairwise  not Hamiltonian isotopic immersed monotone Lagrangian spheres $L[(a,b,c),(a,b,3ab-c)]\subset \CP^2$. They are indexed by pairs of Markov triples $[(a,b,c),(a,b,3ab-c)]$ differing by a single mutation. We distinguish them by computing their refined potentials using wall-crossing and smoothing formulas. For example, the polynomial $P_1(-t)$ for the sphere $L[(1,5,13),(1,13,34)]$ reads:
$$
P_1(-t)=
78t^{-4}+225 t^{-3}+210 t^{-2}+68 t^{-1}+4,
$$
and indeed $P_1(1)$ equals $-1$ as stated in Theorem \ref{th:cp2_p1}. For more examples, see Table~\ref{tab:pot}.

We say that an immersed Lagrangian sphere $L\subset \CP^2$ is non-displaceable from complex lines or conics if, for any symplectomorphism $\phi\colon\CP^2\to \CP^2$, the image $\phi(L)$ intersects every complex  line or conic, respectively. We study this property for the spheres in our family and find the following.
\begin{thm}
	\label{th:wh_line_intro}	
	Any sphere $L[(a,b,c),(a,b,3ab-c)]\subset \CP^2$
	with $c\neq 1$ and $3ab-c\neq 1$ is non-displaceable from complex lines.
\end{thm}

\begin{rmk}
	Immersed Lagrangian spheres in $\CP^2$ are null-homologous and therefore have algebraic intersection number with any complex curve equal to zero. 
\end{rmk}
\begin{rmk}
	The spheres $L[(a,b,c),(a,b,3ab-c)]$ are displaceable from a smooth complex \emph{cubic curve} by construction. Furthermore, any Lagrangian torus in $\CP^2$ is displaceable from a complex line, see \cite{DGI16}. The neighborhood of a monotone sphere is the neighborhood of the nearby resolved Lagrangian torus with a Lagrangian disk attached. Therefore, if the torus is displaced from a given complex line, the attached Lagrangian disk must intersect it. 
\end{rmk}

Theorem \ref{th:wh_line_intro} is proved by analyzing how the disks contributing to various terms in $W_{L}$  can link a line or conic disjoint from $L$, and uses positivity of intersections. In Section \ref{sec:application}, 
we also give examples of monotone immersed spheres that are non-displaceable from conics.

\subsection{Immersions with local systems as objects in the Fukaya category}
The constructions used to define the refined potential can be applied to other holomorphic curve problems as well. For example,
assume that $X$ is a Weinstein 4-manifold and let $f\co Q\to X$ be a strongly exact Lagrangian immersion of an orientable surface. (Strong exactness is a very restrictive condition that precludes any holomorphic disk with corners, see Section \ref{sec:monotone} for notation.)  In this situation, if $L=f(Q)$ and  $\rho$ is a complex representation of the homology  algebra of $ CE(\Lambda,\C[\pi_{1}(L^{\circ})])$, then $\rho$ can be used as a local system on $L$, turning the pair $(L,\rho)$ into a \emph{Lagrangian brane}. That is, $(L,\rho)$ becomes an object of the compact exact $\C$-linear Fukaya category $\F(X)${, see Section \ref{sec:FukayaAnalytic} for the definition of $\F(X)$.

Recall that $H_0CE(\Lambda,\C[\pi_1(L^\circ)])$ is isomorphic to a PPA of a quiver related to the plumbing graph of $L$ (see Section \ref{sec:preprojective}). The representation $\rho$ the homology algebra can thus be interpreted as a representation of the PPA. By definition, such a representation associates a vector space $\C^{n_i}$} to each vertex of the quiver (i.e.~to each connected component of $L^\circ$), and a linear map between the respective spaces for every generator of the homology algebra, with source and target vector spaces corresponding to the source and target of the underlying quiver edge.  Our observation leads to the following result that was predicted by Etg\"u and Lekili \cite{EL17}, see Section~\ref{sec:fukaya} for further explanation.

\begin{thm}\label{mainth:general_brane}	
	Let $L\subset X$ be a strongly exact Lagrangian immersion of a surface into a Weinstein 4-manifold. Complex representations of the degree~0 homology of the capped Chekanov--Eliashberg algebra, $CE_{0}(\Lambda,\C[\pi_{1}(L^{\circ})])$, (or equivalently, of the higher genus multiplicative pre-projective algebra $\B(Q)$) can be used as local systems on $L$ and thereby to define an object of the compact exact $\C$-linear Fukaya category $\F(X)$. Isomorphic representations give quasi-isomorphic objects in $\F(X)$. 
\end{thm}

Floer cohomology of immersed Lagrangians has been studied before. We discuss the relation between Theorem~\ref{mainth:general_brane} and earlier results in Section \ref{sec:fukaya}, see also Remark~\ref{ex:floer_immersed}.

As in the monotone case, there are smoothing formulas in this Floer-theoretic setting as well. In particular, they give natural morphisms between pre-projective algebras and act on their representation varieties. We give an explicit formula for the smoothing maps in Section~\ref{sec: surgery floer}.

\subsection{Generalizations}
\label{sec:generalizations}
Our approach to refined potentials and Floer cohomology of immersed surfaces can be extended in several directions. For a general self-transverse immersion $L$, the above version of the capped Chekanov--Eliashberg algebra should be replaced by $CE(\Lambda;C_{\ast}(\Omega L^{\circ}))$. Here $\Lambda$ is still a collection of Hopf links around double points, but the coefficients are now chains on the based loop space of $L^{\circ}$, compare \cite{EL}. Again, this algebra is the wrapped endomorphism algebra of a union of cotangent fibers in the Weinstein neighborhood of $L$.

It should be noted that self-transverse immersions are never monotone in the sense of Definition~\ref{dfn:monot} when $\dim L > 2$, see Remark~\ref{rmk:monot_2}. For the same reason, the invariants considered in this paper would have quite a different flavor for higher dimensional self-transverse Lagrangians. On the other hand, there is a rather straightforward generalization of the potential to higher-dimensional Lagrangian immersions having a clean self-intersection of codimension \emph{two}; in this case monotonicity makes sense, and the invariants behave accordingly. 

One can in fact go further and replace $L$ by an embedding of a completely general Lagrangian skeleton of a  Weinstein domain $W$; see Section \ref{sec:outlook} for relevant examples. Using a handle presentation of $W$, it is possible to define a suitable Chekanov--Eliashberg algebra associated to the embedded skeleton, again related to the wrapped Fukaya category of $W$. That algebra can then be used to define potentials and objects in the Fukaya category of the ambient manifold. We leave the study of these generalizations to future work.

\subsection*{Acknowledgments}
We are grateful to Yank\i~Lekili for encouraging us to look into the connection with pre-projective algebras, for careful reading of the manuscript, and many very useful comments. We also thank Mohammed Abouzaid and Nicholas Sheridan for helpful conversations.

\section{Holomorphic curves with corners}
\label{sec:corners}

This section studies holomorphic curves with boundary on a monotone immersed Lagrangian submanifold $L\subset X$ that is allowed to have corners at the double points of the Lagrangian. These holomorphic curves can be identified with holomorphic curves in the non-compact symplectic manifold $X^{\circ}$ punctured at the double points of the immersion, where the Legendrian boundary-at-infinity of $L^\circ$ is a Hopf link with Morse-Bott degenerate Reeb chords. 

In later sections we work with generic singularity links with non-degenerate Reeb chords, obtained by attaching a Lagrangian cobordism to the negative ends of $L^\circ$ that interpolates from the Morse-Bott Hopf link to the one with Morse Reeb chords. In this sense, the results in this section can be thought of as establishing transversality (and other properties) for the parts of the holomorphic curves outside these ends.

\subsection{Disks with corners and Maslov index}\label{ssec:maslovindex}
Let $X$ be a $2n$-dimensional symplectic manifold, $Q$ a smooth oriented manifold, and $f\co Q\to X$ a self-transverse Lagrangian immersion. Denote the image $L=f(Q)$. Let $\xi$ be a double point of $L$. Then $f^{-1}(\xi)$ consists of two points $\xi_{1}$ and $\xi_{2}$. Pick an ordering $(\xi_{1},\xi_{2})$ of the preimages. Note that $L$ intersects sufficiently small neighborhoods of $\xi$ in two sheets (Lagrangian disks) $\ell_{1}$ and $\ell_{2}$, and  the ordering of the preimages of $\xi$ induces an ordering $(\ell_1,\ell_2)$ of the sheets.

Let $D\subset \C$ be the unit disk, and consider a finite set $\Sigma \subset \partial D$.
A \emph{topological disk with corners} is a continuous map $u\colon (D,\del D)\to (X,L)$ whose restriction $u|_{\del D\setminus \Sigma}$ admits a continuous lift $\tilde u$,
$$
\xymatrix{
	\del D \ar^-{\tilde u}[r]\ar_{u|_{\del D\setminus \Sigma}}[rd] & Q\ar^f[d]\\
	& L
}\quad
$$
and such that for every $\delta\in\Sigma$, the right and left limits $\lim_{\zeta\to\delta\pm}\tilde u(\zeta)$
both exist and are distinct elements in $f^{-1}(\xi)=\{\xi_{1},\xi_{2}\}$ for some double point $\xi$ of $L$. 
The set $\Sigma$ is the set of \emph{corners (or switches)} of $u$.

We distinguish two types of corners using the ordering of the local sheets near the double points. If the ordering of the sheets is $(\ell_{1},\ell_{2})$ and the oriented boundary of $u$ traverses first the sheet $\ell_1$ and then the sheet $\ell_2$, then we call the corner a \emph{$p$-corner}. If it goes in the opposite direction, from $\ell_{2}$ to $\ell_{1}$, then we call it a \emph{$q$-corner}.

{
Let $u\colon (D,\del D)\to (X,L)$ be a topological disk with corners. Let $\M(u)$ denote the moduli space of topological disks with corners $v\colon (D,\del D)\to (X,L)$ homotopic to $u$ such that $v$ is holomorphic on $D\setminus \Sigma$, for some conformal structure on this domain, as a map into $X$. We next compute the formal dimension of $\M(u)$.

Consider the oriented family $\gamma'$ of Lagrangian tangent spaces $T_{u(\zeta)} L\subset T_{u(\zeta)}X$ parameterized by $\zeta\in \del D$ moving counterclockwise. The family $\gamma'$ is discontinuous at every corner $\delta\in \del D$. For a double point $\xi$, consider the path of Lagrangian subspaces given by $e^{sJ(\xi)}T\ell_{1}\subset T_{\xi}X$, $0\le s\le \pi$. Since $T\ell_{2}$ is transverse to $T\ell_{1}$, there are points $0< s_{1}<\dots< s_{r}<\pi$, where $1\le r\le n$, such that  $\dim(e^{s_{j}J(\xi)}T\ell_{1})\cap T\ell_{2}=d_{j}>0$. The angles $s_{1},\dots,s_{m}$ are called the \emph{K\"ahler angles} of $(T\ell_{1},T\ell_{2})$. Let $V_{j}\subset T_{\xi}X$ be the complex linear subspace containing $(e^{s_{j}J(\xi)}T\ell_{1})\cap T\ell_{2}$ and consider the path $\gamma''(\ell_{1},\ell_{2})$ of Lagrangian subspaces in $T_{\xi} X$ that rotates $T\ell_{1}$ to $T\ell_{2}$ by moving $T\ell_{1}\cap V_{j}$ in $V_{j}$ according to $e^{-sJ}(T\ell_{1}\cap V_{j})$, $0\le s\le (\pi-s_{j})$. (This path is called the \emph{canonical short path} in \cite{Aur14}.) Define $\gamma''(\ell_{2},\ell_{1})$ analogously.
We close $\gamma'$ up with this canonical short path $\gamma''(\ell_{1},\ell_{2})$ or $\gamma''(\ell_{2},\ell_{1})$ at corners according to which sheet is incoming and outgoing at each puncture. This gives a continuous loop $\gamma$ of Lagrangian tangent planes in $TX$. Define the \emph{Maslov index} $\mu(u)$ as 
\begin{equation}\label{eq:maslovimmersed} 
\mu(u)=\mathrm{Maslov}(\gamma) + m,
\end{equation}
where $\mathrm{Maslov}(\gamma)$ denotes the usual Maslov number of the loop $\gamma$ in the trivialization of $TX|_{u(D)}$, and $m$ is the number of corners of $D$. 

The moduli space $\M(u)$ of  unparameterized  holomorphic disks with corners homotopic to $u$ then has expected dimension, compare e.g.~\cite[Section~4.3]{AJ10},
\begin{equation}\label{eq:dimmaslov} 
\dim \M(u)=(n-3)+\mu(u).
\end{equation}
This dimension formula is well-known and used in the description of the differential in Floer cohomology or the description of Chekanov--Eliashberg algebra differentials as in \cite{EESPtimesR} by counts of holomorphic polygons. We outline a proof: the formal dimension of a holomorphic disk with one puncture and boundary condition given by the path $\gamma(\ell_{1},\ell_{2})$ is $-2$, see \cite[Proposition 6.14]{EES}. Applying additivity of index (leading to super-additivity of formal dimension, $d_{\mathrm{glue}}=d_{\mathrm{part1}}+d_{\mathrm{part2}}+1$) under linearized gluing to our disk of interest with these once punctured disks glued on at each corner and using the dimension formula $\dim = (n-3)+\mathrm{Maslov}(\gamma)$ for the resulting closed disk we find that
\[
\dim\M(u) + m(-2 + 1) = (n-3) +\mathrm{Maslov}(\gamma)
\]
and hence
\[ 
\dim\M(u)=(n-3) + \mu(u).
\]
}
We consider how corners and smoothing affect the Maslov index.
\begin{lem}
	\label{lem:const_mu}	
	Let $u\co (D,\del D)\to (X,L)$ be the constant map into a double point of $L$ where $u$ is declared to have $2m$ switches,  $m\in \Z_+$. Then $$\mu(u)=-m(n-2).$$
\end{lem}

\begin{proof}
	Let $r\in L$ be the double point and $T\ell_1$, $T\ell_2$ the tangent spaces to the two sheets of $L$.
	The concatenation of two canonical short paths from $T\ell_1$ to $T\ell_2$ back to $T\ell_1$ results in a $\pi$-rotation in each complex coordinate line, which is a loop of Lagrangian subspaces in $T_rX$ of Maslov index $-n$. Therefore $\mu(u)=-mn+2m=-m(n-2)$.
\end{proof}

\subsection{Lagrange surgery with zero area parameter}
{
\label{sec:LagrangeSurgery}
Lagrange surgery \cite{Pol91} is a smoothing of a Lagrangian near a transverse double point that preserves the Lagrangian condition. More precisely, it replaces an arbitrarily small neighborhood of the transverse double point with an embedded Lagrangian cylinder. By construction, the surgery leaves the complement of an arbitrarily small neighborhood of the double point intact. Up to Hamiltonian isotopy there are precisely two different Lagrangian surgeries at each double point, each of which depends on an `area-parameter'. Here we describe a version of Lagrange surgery with a `zero area-parameter' as in~\cite[Section~2, Step 3]{Va17}. It, in particular, preserves the monotonicity property of Lagrangians.

Recall that a transverse double point of any Lagrangian immersion $L$ has a Darboux $\epsilon$-ball neighborhood $B_\epsilon \subset \C^2$ for some $\epsilon>0$ in which $L$ agrees with the union $(\R^{2} \cup i \R^2) \cap B_\epsilon$. In other words, inside $B_{\epsilon}$, $L$ is identified with 
$$\{xy=0\} \times S^1 \subset \C \times S^1$$
under the (noninjective) map
\begin{gather*}
A \colon \C \times S^1 \to \C^2,\\
(z=x+iy,\theta) \mapsto A_\theta(z) \coloneqq \begin{pmatrix}
\cos \theta & -\sin \theta \\
\sin \theta & \cos \theta
\end{pmatrix}\cdot \begin{pmatrix}
x+iy\\
0
\end{pmatrix}.
\end{gather*}
Choosing the Darboux ball appropriately we may assume that $A(\{x,y\ge 0\}\times \{1\}) \subset \C^2$ is a piece of a holomorphic surface with a single $p$-corner when it is given the complex orientation. 
%(This depends on our previous choice of labelings.)

The matrices $A_\theta$ are elements of $SU(2)$ and hence symplectic. Further, for any $\theta \in S^1$, the image $\frac{d}{d\theta}A_\theta(\C \times \{0\})$ under
$$\frac{d}{d\theta}A_\theta=\begin{pmatrix}
-\sin \theta & -\cos \theta \\
\cos \theta & -\sin \theta
\end{pmatrix},$$
is (symplectically) orthogonal to $A_\theta(\C \times \{0\})$.

The zero area parameter versions of the surgeries can now be described as follows. Consider a proper embedding $\gamma \subset \C$ of $\R$ which satisfies the following:
\begin{itemize}
\item $\gamma$ coincides with the boundary $\{xy=0\} \cup \{x,y \ge 0\}$ of the first quadrant outside of a compact subset, 
\item $-\gamma \cap \gamma = \varnothing$, and
\item the signed area between $\gamma$ and $\{xy=0\} \cup \{x,y \ge 0\}$ equals zero.
\end{itemize}
The two different zero area surgeries are then obtained by replacing the transversely intersecting sheets $A(\{xy=0\} \times S^1)$ with either the Lagrangian cylinder $A(\gamma \times S^1)$, or the Lagrangian cylinder $A(i\gamma \times S^1)$. The first case will be referred to as a \emph{$p$-surgery} on $L$ and the resulting Lagrangian will be denoted by $L_p$, the second as a \emph{$q$-surgery} and the resulting Lagrangian denoted by $L_q$.

It is obvious that the Hamiltonian isotopy class of $L_{p}$ (respectively $L_{q}$) is independent of the choice of curve $\gamma$. We also point out that $\gamma$ can be scaled by an arbitrarily small constant to a new surgery curve for which the surgery leaves the Lagrangian undeformed in the complement of an arbitrarily small neighborhood of the double point.
}

{
Let $\xi\in L$ be a double point on $L$ and $ L_p$, $ L_q$ be the two surgeries of $L$ at that double point using zero area parameter as described above. Both are immersions with one less double point. The next lemma is well-known, we use \cite{ESjams} in the proof below, see also~\cite[Proposition~60.23]{FO306}.  

\begin{lem}
	\label{lem:mu_res}
	Let $u\colon(D,\del D)\to (X,L)$ be a topological disk with corners having a total number $a$ of $p$-switches and a total number $b$ of $q$-switches at $\xi$. Let $L_p$ be the $p$-surgery of $L$ at $\xi$ and $\t u\in H_2(X,L_p)$ a topological disk obtained by smoothing $u$ locally near the double point, $n>1$. Then 
	\begin{equation}
	\label{eq:mu_res}
	\mu(\t u)=\mu(u)+(n-2)b
	\end{equation}
	In particular,  resolution preserves the  Maslov index when $n=2$.
\end{lem}

\begin{proof}
In order to calculate the Maslov index $\mu(\t u)$, we consider the path of tangent planes that starts at one sheet near the double point and ends at the other and then close it up over the respective handles. The contribution to the Maslov index from the close up can be expressed in terms of a phase function that is calculated in \cite[Lemma 2.4]{ESjams}. The result then follows as the Maslov index calculation in \cite[Lemma 2.5]{ESjams}   
\end{proof}

}

\subsection{Monotone and strongly exact immersions}
\label{ssec:monotoneandexact}
Consider a self-transverse immersed Lagrangian surface $L\subset X$.
\begin{dfn}\label{dfn:monot}
	An orientable immersed Lagrangian surface $L\subset X$ is \emph{monotone} if the following two conditions hold.
	\begin{itemize}
		\item For any topological disk $u\colon (D,\del D)\to(X,L)$ with corners, $\mu(u)\ne 1$.
		\item There exists a positive constant $\lambda>0$ such that
		$$
		\mu(u)=\lambda\cdot\omega(u), \quad \lambda>0,
		$$
		for all disks with corners $u$.
	\end{itemize}  
\end{dfn}
\begin{rmk}
Since our goal here is only to construct a well-defined potential function and not to define Floer cohomology (except in the strongly exact case), we leave aside the more involved conditions that would be necessary to impose when the ambient manifold $X$ is not simply connected.
\end{rmk}

\begin{rmk}
	\label{rmk:monot_2}
	In view of Lemma~\ref{lem:const_mu}, $\dim L=2$ is the only case when monotone Lagrangian immersions (with at least one double point) can exist.
\end{rmk}

We next consider the exact case. Let $(X,\omega)$ be an exact symplectic manifold, $\omega=d\lambda$. If $L$ is an exact Lagrangian submanifold of $X$, any holomorphic disk with boundary on $L$ has zero area and must therefore be constant. For immersed Lagrangian submanifolds, exactness does not preclude non-constant disks with corners. Our next notion of strong exactness does. 

\begin{dfn}
	\label{dfn:exact}
	An immersion $f\colon Q\to X$, where $(X,d\lambda)$ is an exact manifold is \emph{strongly exact} if $f^{\ast}\lambda$ admits a primitive function that assumes the same value at the two points $\xi_{1}$ and $\xi_{2}$ in the preimage $f^{-1}(\xi)$ for any double point $\xi$ of $L=f(Q)$.
\end{dfn}

The following is a consequence of Lemma~\ref{lem:mu_res}.

\begin{lem}
	\label{lem:resolv_monot}
	Suppose that $L$ is a monotone (resp.~strongly exact) immersed Lagrangian surface. Then its resolutions $ L_p, L_q$ at any positive double point  are monotone (resp.~strongly exact).
\end{lem}

\begin{proof}
	Orientability of the resolution follows from positivity of the double point.
	
	In the first case, one needs to check monotonicity on each class in $H_2(X, L_p)$. This group is of rank one higher than $H_2(X,L)$, and the new class is generated by the Maslov index 0 Lagrangian disk which arises from the local surgery picture. Since it has zero area, the monotonicity condition holds for this disk.
	
	The topological smoothing of a disk in $H_2(X,L)$ to $H_2(X, L_p)$ has the same Maslov index as the original disk by Lemma~\ref{lem:mu_res}. The disks also have the same areas, by the zero area parameter condition for the surgery. Monotonicity follows.
	
	In the strongly exact case, the primitive on $Q$ can be deformed to a primitive on either resolution.
\end{proof}

\subsection{Number of corners, asymptotic winding, area}
Let $Q$ be a surface and $f\co Q\to X$ a Lagrangian immersion.
We begin our study of holomorphic disks on immersed Lagrangian surfaces $L\subset X$. In order to simplify the analysis at the double points we restrict attention to \emph{admissible} tame almost complex structures $J$ on $X$ that satisfy the following conditions. Every double point $\xi$ of $L$ has a Darboux $\epsilon$-ball neighborhood $B_\epsilon \subset \C^2$ for some $\epsilon>0$ in which $J$ agrees with the standard complex structure on $\C^{2}$ and $L$ agrees with the union $(\R^{2} \cup i \R^2) \cap B_\epsilon$. Note that admissible almost complex structures form a contractible space. 

We find an a priori estimate on the number of corners of a holomorphic disk with boundary on $L$ in terms of its area. An analogous result in the exact case for curves with additional punctures at infinity was established in \cite[Theorem 1.2]{CEL10}. We use a similar proof and therefore recall the setup and notation from \cite{CEL10}, starting with the notion of winding numbers at punctures.

Let $u\colon (D,\partial D)\to (X,L)$ be a holomorphic disk with corners. Recall that we use admissible almost complex structures, i.e.~$J$ and $L$ are standard near the double points. A marked point $\eta\in \del D$ mapping to a double point of $L$ is called a \emph{puncture}, regardless of whether or not it is a switch. For  any puncture of a disk $u$, there are half-strip coordinates $s+it\in[0,\infty)\times[0,1]$ around $\eta$ where $u$ admits Fourier expansion of the form 
\[ 
u(s+it) = \sum_{k>0} c_{k} e^{-k\pi(s+it)},
\]
where $c_{k}\in\R^{2}$ or $c_{k}\in i\R^{2}$, and where $k$ ranges over $\frac12 +\Z_{\ge 0}$ if the puncture is switching and over $\Z_{\ge 1}$ if it is non-switching. Alternatively, one can use a punctured half-disk neighborhood of the origin in $\C$ as coordinates in a neighborhood of the puncture. Then the corresponding expansion is a Taylor expansion
\[ 
u(z) = \sum_{k>0} c_{k} z^{k}.
\]
Following \cite{CELN},  the exponent $k_{0}>0$ of the first non-vanishing Fourier coefficient is called the \emph{asymptotic winding number} of $u$ at $\eta$.

\begin{lem}\label{l:area}
	There exists a constant $C>0$ such that if $u\colon (D,\partial D)\to (X,L)$ is a holomorphic disk with punctures $\eta_{1},\dots,\eta_{m}$ of winding numbers $w_{1},\dots,w_{m}$ then
	\[ 
	\mathrm{Area}(u)=\int_{u}\omega \ > \ C\sum_{j=1}^{m} w_{j}.
	\] 
	In particular, the number of corners of a disk is uniformly bounded in terms of its area.
\end{lem}

\begin{proof}
	The Lagrangian and the almost complex structure are standard on $B_{\epsilon}$, a Darboux ball around any double point of $L$.	
	Consider $\Omega_{r}=u^{-1}(B_{r})$ for a generic $r\in(\frac12\epsilon,\epsilon)$. Let $\Omega^{j}_{r}$ be a component of $\Omega_{r}$ such that $\Omega_{r}^{j}\cap\partial D= S\ne \varnothing$. One applies Schwartz reflection to $u|_{\Omega_{r}^{j}}$, first over $\R^{2}$ and then over $i\R^{2}$. The result is a holomorphic curve $C$ with boundary in $\partial B_{r}$. Holomorphic curves are minimal surfaces, and the monotonicity formula for minimal surfaces \cite[\S 17]{Simon} implies that the function
	\[ 
	\rho(r) = \frac{\mathrm{Area}(C)\cap B_{r}}{\pi r^{2}}
	\]
	is non-decreasing in $r$. Assume that $\Omega_{r}^{j}$ contains punctures $\eta_{1},\dots,\eta_{s}$ of weights $w_{1},\dots, w_{s}$. Then the Fourier (or Taylor) expansion says that
	\[ 
	\lim_{r\to 0}\rho(r)\ge 2\sum_{j=1}^{s} w_{j}. 
	\]      
	Together with the monotonicity formula it implies that
	\[ 
	\mathrm{Area}(u|_{\Omega_{j}^{r}})\ge \frac12 \left(\sum_{j=1}^{s} w_{j}\right) \pi r^{2}.
	\]
	Since every puncture is contained in some $\Omega_{r}^{j}$, the lemma follows.
\end{proof}

\subsection{Gromov compactness for $\M(\zeta)$}\label{s:gromovcompact}
Let $f\co Q\to X$ be a monotone Lagrangian immersion.
Let $\M(\zeta)$ be the moduli space of holomorphic disks with boundary on $L$ (possibly with corners), of Maslov index $2$, and whose boundary passes through a specified point $\zeta\in L$.

Let $L_{k}\subset X$, $k=1,2,\dots$ be a sequence of monotone immersed Lagrangians such that $L_{k}\to L$ as $k\to\infty$, let $\zeta_{k}\in L_{k}$ be points so that $\zeta_{k}\to\zeta$ as $k\to\infty$, let $J_{k}$ be a sequence of almost complex structures admissible with respect to $L_{k}$ such that $J_{k}\to J$ as $k\to\infty$, and let $u_{k}$ be a sequence of $J_{k}$-holomorphic disks with boundary on $L_{k}$ and passing through $\zeta_{k}$.

\begin{lem}\label{l:Gromov1}
	Assume that $\zeta$ is not a double point of $L$.	
	The sequence of holomorphic disks $u_{k}$ has a subsequence that Gromov converges as $k\to\infty$ to a stable holomorphic disk $u$ such that exactly one component $u^{0}$ of $u$ is non-constant, has $\mu(u^{0})=2$ and passes through $\zeta$, and all other components are constant maps to double points of $L$. 
\end{lem}

\begin{proof}
	By Lemma \ref{l:area}, the maps in the sequence have a uniformly bounded number of corners.	
	Gromov convergence is well known in the setting of immersed Lagrangian submanifolds, see e.g.~\cite{FO3Book, EES}. Let $u^{0},\dots, u^{r}$ be the stable disks in the limit. Since the Maslov index $\mu$ is additive we find that
	\[ 
	2=\mu(u^{0})+\dots+\mu(u^{r}).
	\]
	By monotonicity, $\mu(v)\ge 2$ if $v$ is non-constant. It follows that exactly one of the disks in the limit is non-constant and passes through $\zeta$. We choose notation so that it is $u^{0}$. Then $\mu(u^{0})=2$. 
\end{proof}

We discuss the limiting constant components from Lemma~\ref{l:Gromov1} in greater detail. Consider a sequence of holomorphic disks as above that Gromov converges to a broken holomorphic disk. The domain of the disks in the sequence is the unit disk with boundary punctures, and the bubbling of a constant disk  corresponds to two or more of these boundary punctures colliding. It is possible to retain  information about asymptotic winding numbers in the limit. To this end, one  defines a constant disk as follows. 

A \emph{constant disk} is a constant map to a double point of $L$ from a disk $D_{m}$ declared to have any number $m\ge 2$ of punctures. One puncture is distinguished and called \emph{positive}, and all other punctures are called \emph{negative}. Each boundary component of $D_m$ is decorated by a sheet $\ell_{j}$, $j=1,2$ of $L$ at the double point, and each negative puncture is decorated with an asymptotic winding number which is a half-integer or an integer according to whether or not the sheet decoration switches at the puncture. This leads to the following refined version of Lemma~\ref{l:Gromov1}.

\begin{lem}\label{l:Gromov2}
	Consider the situation of Lemma \ref{l:Gromov1}. If $\eta$ is a puncture of the limiting non-constant disk $u^0$ to which a (stable) constant disk is attached, the winding number at $\eta$ equals the sum of the winding numbers at the negative punctures of the attached  constant disk.	
\end{lem}

\begin{proof}
	For sufficiently large $k$ there is a small half-disk in the domain of $u_{k}$ which contains all punctures converging to the constant disk in the limit, and which maps into a neighborhood of the double point. Picking a generic projection to the complexification of a real line through the origin, the statement follows from the argument principle from elementary complex analysis. 
\end{proof}

\subsection{Transversality for $\M_{S}(\zeta)$}\label{s:tv}
In this section we show that any holomorphic disk in $\M(\zeta)$ is somewhere injective. This is used to prove that $\M(\zeta)$ is transversely cut out for a generic admissible almost complex structure, and that the same result holds for families. We use a crude version of an argument of Lazzarini \cite{La00}. %(for a more complete treatment see \cite{Pe18}).

\begin{lem}\label{l:somewhereinj}
	Let $u\colon (D,\partial D)\to (X,L)$ be a Maslov index~2 holomorphic disk (with corners). Then $u$ is somewhere injective.
\end{lem}

\begin{proof}
	By monotonicity, $u$ is absolutely area minimizing among holomorphic integral currents with boundary on $L$. Consider the image $C=u(D)$ as a rectifiable current. By unique continuation for holomorphic maps, $C$ can be written as
	\begin{equation}\label{eq:integralcurrentdecomp} 
	C= C_{1}+ 2C_{2} +\dots + mC_{m},
	\end{equation}  
	where $C_{j}$ is the subset of $C$ where the multiplicity equals $j$. It is clear that each $C_{j}$ is a holomorphic rectifiable current. 
	
	Assume that $u$ is nowhere injective. Then $C_{1}=\varnothing$ in \eqref{eq:integralcurrentdecomp}. We show below that $C_{m}$ is an integral current, i.e.~that $\partial C_{m}$ is rectifiable and contained in $L$. Then $C'=C-C_m$ is an integral current with boundary in $L$ of area smaller than $C$, which contradicts minimality and implies that $C_{1}$ is non-empty.
	
	We first show that the boundary of $C_{m}$ is contained in $L$. Assume that $p\notin L$ lies in the support of $C_{m}$. By Federer's refinement of Sard's theorem \cite[Theorem 3.4.3]{Federer} we may assume that $C_{m}$ is a smooth 2-dimensional submanifold around $p$. By unique continuation it follows that there is a disk $\Delta$ centered at a point in $u^{-1}(p)$ such $u(\Delta)$ lies in the support of $C_{m}$. Hence $p$ does not lie on the boundary of $C_{m}$ and its boundary is contained in $L$. 
	
	Consider a point $\xi\in L$ on the boundary of $C_{m}$. Again by the refinement of Sard's theorem,  one may assume that $\xi=u(z)$ is a regular value of $u|_{\partial D}$. By unique continuation, the boundary of $C_{m}$ consists of an arc in $L$ around $\xi$. It follows by monotonicity that the total length of the boundary arcs in the boundary of $C_m$ is finite. Thus $\partial C_{m}$ is a rectifiable curve in $L$ and hence $C_{m}$ is an integral current. The lemma follows.      
\end{proof}

We next establish the transversality result we need. Consider a smooth $s$-parametric family $(L_{s},\zeta_{s},J_{s})$, $s\in S$, of monotone Lagrangian immersions $L_{s}$, points $\zeta_{s}\in L_{s}$, and admissible almost complex structures $J_{s}$. For each $s\in S$ this defines a moduli space $\M_{s}(\zeta_{s})$ of $J_{s}$-holomorphic Maslov index 2 disks with boundary on $L_{s}$. If $u\in \M_{s}(\zeta_{s})$, then $u$ has a finite number of boundary punctures $\eta_{1},\dots,\eta_{m}$ mapping to double points of $L_{s}$. At each puncture there is an asymptotic winding number $w_{m}$. We write $\M_{s}(\zeta_{s};\mathbf{w})$ for the space of holomorphic disks of Maslov index 2 with boundary punctures with asymptotic winding numbers $\mathbf{w}=(w_{1},\dots,w_{m})$. Denote 
\[ 
|\mathbf{w}|=\sum_{j=1}^{m}2|w_{j}-\tfrac12|.
\] 
\begin{lem}\label{l:tv}
	The formal dimension of $\M(\zeta_{s};\mathbf{w})$ equals
	\[ 
	\dim\M(\zeta_{s};\mathbf{w})=-|\mathbf{w}|.
	\]
	For a generic family $S$,
	\[ 
	\M_{S}(\zeta;\mathbf{w})=\bigcup_{s\in S}\M_{s}(\zeta_{s};\mathbf{w})
	\]
	is a $C^{1}$-smooth and transversely cut out manifold of dimension 
	\[ 
	\dim \M_{S}(\zeta;\mathbf{w})=
	\dim S-|\mathbf{w}|.
	\]
\end{lem} 

\begin{proof}
	Note that the space of disks with asymptotic winding number $w>\frac12$ at $\eta$ can be modeled on a weighted Sobolev space with weight $w-\delta$ for small $\delta>0$. It is well-known that the Fredholm index jumps each time the weight passes through an eigenvalue of the asymptotic operator, see e.g.~ \cite[Proposition 6.5]{EES}. Lemma \ref{l:somewhereinj} implies that any disk in $\M_{s}(\zeta_{s})$ is somewhere injective, and a well-known argument then shows that the moduli space is transversely cut out for generic data. The lemma follows.     
\end{proof}

Note that being transversely cut out in the above lemma in particular means that:
\begin{itemize}
	\item When $|\mathbf{w}|=0$ and $\dim S=0$, the moduli space consists of disks with corners at which the full one-dimensional family of the moduli space evaluates transversely along the boundary to the point $\zeta$.
	\item When $|\mathbf{w}|=1$ and $\dim S=1$, the moduli space consists of disks with corners for which a unique smooth and regular boundary point is mapped to a double point of $L_s$. Furthermore, at a solution in this moduli space, the full two-dimensional family of disks with corners evaluates transversely along the boundary to \emph{both} the double point \emph{and} to $\zeta$. See Figure~\ref{fig:bifurcation}, middle.
\end{itemize}

\subsection{Punctures at double points of an immersion}
\label{subsec:l_o}
In this section we give an alternative interpretation of the curves with corners studied above.
Let 
$f\colon Q\to X$ be a monotone Lagrangian immersion and $L=f(Q)$. Let  $\xi_{1},\dots,\xi_{r}\in X$ be the double points of $L$. Consider an admissible almost complex structure $J$ on $X$, see Section~\ref{sec:corners}. Let $X^{\circ}=X\setminus\{\xi_{1},\dots,\xi_{r}\}$ be the result of puncturing $X$ at all double points of $L$, and $L^{\circ}=L\setminus \{\xi_{1},\dots,\xi_{r}\}$. Then $X^{\circ}$ is a monotone open symplectic manifold with a concave cylindrical end for each $\xi_{j}$. By the admissibility assumption, there are Darboux coordinates near each $\xi_j$ where $L$ corresponds to $\R^{2}\cup i\R^{2}\subset \C^2$, with the double point $\xi_{j}$ at the origin.

These local coordinates give an identification of the concave end of $X^\circ$ at $\xi_{j}$ with the negative half of the symplectization $(-\infty,2\log \epsilon]\times S^{3}$ of the standard contact $S^{3}$ as follows. Choose Darboux coordinates near $\xi_j$ that identifies its neighborhood with an $\epsilon$-ball $B_{\epsilon}\subset \C^{2}$ and consider the symplectomorphism $$\phi\colon ((-\infty,\, 2\log\epsilon]\times S^{3},\, d(e^{t}\alpha))\to (B_{\epsilon} \setminus \{0\},\, \omega_0),$$
given by
\begin{equation}\label{eq:neg_cyl}
\phi(t,y) = e^{t/2}y.
\end{equation}
Here $\alpha=\frac{1}{2}\sum_i(x_idy_i-y_idx_i)$ is the standard $S^1$-invariant contact form on the unit sphere $S^{3}\subset \C^{2}$. Note that $\phi$ takes the punctured Lagrangian $L^{\circ}$ to the cone $(-\infty,2\log\epsilon]\times\Lambda_{\rm st}$, where $\Lambda_{\rm st}$ is the standard Legendrian Hopf link: 
$$
\Lambda_{\rm st}=S^{3}\cap (\R^{2}\cup i\R^{2}).
$$

Let us describe the Reeb chords of $\Lambda_{\rm st}$. They come in $S^{1}$-families that are non-degenerate in the Bott sense. For each positive half-integer $\frac12 k>0$ there are two families of Reeb chords with action $\frac12k\pi$. The chords of integer actions have endpoints on the same connected component of $\Lambda_{\rm st}$, while the chords of non-integer action connect different components. 

Consider the standard cylindrical $S^1$-invariant complex structure on $(-\infty,2\log\epsilon]\times S^3$, which equals the standard integrable complex structure $i$ on $\C^2$ under the identification \eqref{eq:neg_cyl}. Extend it to a tame almost complex structure on all of $X^\circ$.
A \emph{punctured holomorphic disk} in $X^{\circ}$ with boundary on $L^{\circ}$ is a  holomorphic disk with boundary punctures that are asymptotic to Reeb chord strips of $\Lambda_{\rm st}$.
\begin{lem}\label{l:outside=compact}
	There is a natural 1--1 correspondence between punctured  holomorphic disks in $X^{\circ}$ with boundary on $L^{\circ}$, and  holomorphic disks in $X$ with boundary and corners on $L$ in the sense of Section~\ref{ssec:maslovindex}.
\end{lem}

\begin{proof}
	This is immediate from the definition: the compact disks have Taylor expansions near the punctures, the non-compact disks have the corresponding Fourier expansion. 
\end{proof}

\section{Lagrangians with negative cylindrical ends and capped Chekanov--Eliashberg algebras}
\label{sec:monotone}
In this section we first introduce capped Chekanov--Eliashberg algebras used in the definition of the refined potential. Then we give the definition of the refined potential and prove its invariance, establishing Theorem \ref{t:general}.  

\subsection{Rigid disks in standard contact spheres}\label{ssec:geomstcontactsphere}
Chekanov--Eliashberg algebras of Legendrian submanifolds are central tools in symplectic topology, introduced by Chekanov \cite{Che02} and Eliashberg--Givental--Hofer \cite{EGH00}. While a definite general reference is still not available, it is possible to combine the established results to define Chekanov--Eliashberg algebras of Legendrians in the standard contact sphere $S^{2n-1}$, and to show that they satisfy the expected properties. This section explains the necessary geometric and analytical  results leading to this definition, and to concrete calculations.

Let $x+iy=(x_{1}+iy_{1},\dots,x_{n}+iy_{n})$ be coordinates on $\C^{n}$. Consider the standard $S^{1}$-invariant contact form
\[ 
\alpha_{\rm st} = \frac12\sum_{j=1}^{n} (x_{j}dy_{j}-y_{j}dx_{j})
\]
on $S^{2n-1}=\{x+iy\colon |x|^{2}+|y|^{2}=1\}$. The Reeb flow of $\alpha$ is periodic and  Bott-degenerate. Indeed, this is the flow along the fibers of the Hopf fibration $S^{2n-1}\to\C P^{n-1}$, so there is a $\C P^{n-1}$-family of orbits of action $\pi k$ for each positive integer $k$. Furthermore, the Conley-Zehnder index of the family of action $\pi k$ equals $(2n-2)k$. (We use the convention for the Conley-Zehnder index in the Bott setting for which the index coincides with the expected dimension of the moduli space of planes inside the symplectization of $S^{2n-1}$  asymptotic to a fixed orbit in the Bott manifold at its positive puncture.)  It follows that for a small perturbation $\alpha$ of $\alpha_{\rm st}$, no Reeb orbit of $\alpha$ has Conley-Zehnder index less than $2n-2$.

Let $\Lambda\subset S^{2n-1}$ be a Legendrian submanifold with trivial Maslov class. Here we use a contact form $\alpha$ on $S^{2n-1}$ close to $\alpha_{\rm st}$ as  above. We will explain how to define Chekanov--Eliashberg algebras of $\Lambda$ with differential that counts rigid holomorphic disks. (This includes the main case used in this paper, where the coefficients lie in the group algebra of a fundamental group, but does not cover the version with coefficients in chains on the based loop spaces, which requires using higher-dimensional moduli spaces.) 

Consider a contact form $\alpha$ for which the Reeb chords of $\Lambda$ are transverse. Let $\R\times S^{2n-1}$ be the symplectization of $S^{2n-1}$ with the symplectic form $\omega=d(e^{t}\alpha)$, and  $J$ an almost complex structure adapted to $\alpha$. We first consider holomorphic disks $u\colon(D,\partial D)\to (\R\times S^{2n-1},\R\times\Lambda)$ with one positive boundary puncture, several negative boundary punctures, and several negative interior punctures. The disk $u$ must be asymptotic to a Reeb chord at a boundary puncture, and a Reeb orbit at an interior puncture. We call them symplectization disks.

Denote by $\mathcal{M}(u)$ the moduli space of disks with such asymptotics in the homotopy class of $u$. Let $a$ be the positive puncture asymptotic of $u$ and $\mathbf{b}=b_{1}\dots b_{m}$ the word of Reeb chords at the negative boundary punctures. Since the Maslov class of $\Lambda$ vanishes, any Reeb chord $c$ has an integer grading $|c|$. We write $|\mathbf{b}|=\sum_{j=1}^{m}|b_{j}|$. 

\begin{lem}\label{l:tvonepos}
	For generic $J$, $\mathcal{M}(u)$ is a transversely cut out manifold. In particular, if $u$ has at least one interior puncture, $\mathcal{M}(u)\ne \varnothing$ only if
	\[ 
	|a|-|\mathbf{b}|> 2n-2\ge 2.
	\]
\end{lem}    

\begin{proof}
	Since $u$ has only one positive puncture, $u$ has injective points (near the positive puncture), see the proof of \cite[Proposition 3.13]{Riz16}. Transversality then follows by perturbing the almost complex structure near the Reeb chord of the positive puncture and the last statement follows from the index bound on orbits.
\end{proof}

The Chekanov--Eliashberg dg-algebra is generated by Reeb chords of the Legendrian and its differential counts rigid disks with one positive and several negative boundary punctures. We will also use dg-algebra maps associated to exact Lagrangian cobordism $L\subset \R\times S^{2n-1}$. More precisely,  consider an exact Lagrangian cobordism $L$ that agrees with $[T,\infty)\times \Lambda^{+}$ in $([T,\infty)\times S^{2n-1},d(e^{t}\alpha_{+}))$ and with $(-\infty,-T]\times\Lambda_{-}$ in $((-\infty,-T]\times S^{2n-1},d(e^{t}\alpha_{-}))$ for some large $T>0$, where both contact forms are constant multiples of forms close to $\alpha_{\rm st}$. 
We write $\mathcal{M}^{\rm co}(u)$ for moduli spaces of disks  $u\colon (D,\partial D)\to (\R\times S^{2n-1},L)$,
called cobordism disks.
The next lemma is used to prove that the Chekanov--Eliashberg differential squares to zero, and that cobordism maps are chain maps.  

\begin{lem}\label{l:mfdwboundary}
	If $\dim(\M(u))=2$, then for generic $J$, $\M(u)/\R$ is a 1-manifold with a natural compactification consisting of two-level broken disks of Fredholm index $1$. If $\dim(\mathcal{M}^{\rm co}(u))=1$, then $\mathcal{M}^{\rm co}(u)$ is a 1-manifold that has a natural compactification consisting of two-level broken disks:  one cobordism disk of Fredholm index~$0$ attached (above or below) to a symplectization disk of Fredholm index~$1$. 	
\end{lem}   

\begin{proof}
	According to the SFT compactness theorem, the boundary of the above moduli spaces considered consists of several-level broken disks. Lemma~\ref{l:tvonepos} shows that the only possible boundary configurations are two-level disks as stated. To show that these two-level disks are precisely the boundary, one needs a standard gluing result, see e.g.~\cite[Chapter 9]{Seidel} and \cite[Appendix A]{Ersft} for an adaptation to this case.    
\end{proof}

The  general property of Chekanov--Eliashberg algebras is that their homology is invariant under Legendrian isotopy. The standard SFT proof of that property uses chain homotopies of cobordism maps. That argument does, however, use abstract perturbations, since one must consider repeated gluings of disks of negative dimension that in turn lead to non-transverse boundary configurations. There is an alternative to that argument, see \cite[Section 2.4]{EESPtimesR}, that we use here. Consider a generic Legendrian isotopy. Away from a finite set of instances, the moduli spaces involved in the differential undergo a cobordism, and the invariance is obvious. At the finitely many exceptional instances, one can deduce the invariance as an algebraic consequence of the fact that the differential squares to zero for a Legendrian $\Lambda\times \R\subset S^{2n-1}\times T^{\ast}\R$ associated to a small neighborhood of the degenerate instance. We will state the desired result as a lemma.

Let $\Lambda$ be any instance in a generic Legendrian isotopy, and $\Lambda'\subset S^{2n-1}\times T^{\ast}\R$ be the Legendrian submanifold associated to the 1-parameter family around $\Lambda$ as in \cite[Proposition 2.6]{EESPtimesR}. Denote by $\M'(u)$  the moduli space of holomorphic disks in $\R\times S^{2n-1}\times T^{\ast}\R$ with boundary in $\R\times\Lambda'$, one positive and several negative boundary punctures.

\begin{lem}\label{l:invmdli}
	If $\dim(\M'(u))=2$, then for a generic $J$, $\M'(u)/\R$ is a 1-manifold admitting a compactification whose boundary consists of two-level broken disks of Fredholm index~$1$.
\end{lem}

\begin{proof}
	Directly analogous to Lemma \ref{l:mfdwboundary}.	
\end{proof}	

Finally, we state a result that allows  to compute the Chekanov--Eliashberg dg-algebra of $\Lambda\subset S^{2n-1}$ in a Darboux chart.  Since the complement of a point in $S^{2n-1}$ is contactomorphic to the standard contact vector space by \cite[Proposition 2.1.8]{Ge08}, any Legendrian $\Lambda$ obviously lies inside a Darboux chart.  Furthermore, $\Lambda$ has only finitely many Reeb chords inside the Darboux chart. 

\begin{lem}\label{l:smallchords}
	For any $K>0$, there exists a Legendrian isotopy of $\Lambda$ inside the Darboux chart and a contact form $\alpha$ such that the set of Reeb chords inside the Darboux chart does not change, and any Reeb chord not contained in the chart has length as well as index bounded from below by $K$.
\end{lem}

\begin{proof}
	First, shrink $\Lambda$. Second, Morsify $\alpha_{\rm st}$ by a Morse function on $\C P^{n-1}$. Place the chart for $\Lambda$ close to the minimum orbit. The return map of the Reeb flow near the minimum orbit is rotation by a small angle in each complex direction. By making $\Lambda$ sufficiently small, one can  make any outside Reeb chord go arbitrarily many times around the $S^1$-fiber. Each time this  adds $2n-2$ to the grading.	
\end{proof}	

\subsection{Chekanov--Eliashberg algebra with a Lagrangian cap}
\label{sec:dga}
Consider a Legendrian submanifold $\Lambda \subset (S^{2n-1},\alpha)$ in the  contact sphere with a contact form $\alpha$ close to $\alpha_{\rm st}$. Its Chekanov--Eliashberg dg-algebra was described in  Section \ref{ssec:geomstcontactsphere}.

Let $L^\circ$ be an arbitrary smooth oriented spin $n$-manifold with boundary, and fix a diffeomorphism $\del L^\circ\to \Lambda$. One thinks of $L^\circ$ as an abstract cap attached to $\Lambda$ at the positive end of the symplectization of $\R\times\Lambda\subset \R\times S^{2n-1}$.

We introduce a version of the Chekanov--Eliashberg dg-algebra of $\Lambda$ that we call \emph{capped}, with coefficients in $\pi_1(L^\circ)$:
\begin{equation}
\label{eq:ce_pi_1L}
CE(\Lambda;\C[\pi_1(L^\circ)]).
\end{equation}
It is a straightforward generalization of the original definition, which also keeps track of the homotopy classes, in the cap, of the boundaries of the holomorphic disks contributing to the differential.

\begin{rmk}
	Early works on Chekanov--Eliashberg algebras used coefficients in the group ring of $H_1(\Lambda)$. The paper  \cite{ES14} used coefficients in the group ring of the fundamental group of the Legendrian:
	$$
	CE(\Lambda;\C[\pi_1(\Lambda)]).
	$$
	The capped version \eqref{eq:ce_pi_1L} is the extension of the latter by the map $\pi_1(\Lambda)\to \pi_1(L^\circ)$.
	Recently, \cite{CDGG18} worked with Chekanov--Eliashberg algebras in relation to the fundamental group of Lagrangian cobordisms. The capped version used here is closely related. 
	
	More generally, algebras like \eqref{eq:ce_pi_1L} belong to a geometric framework involving chains on loop spaces. The capped algebra here is a specialization of the dg-algebra 
	\begin{equation}
	\label{eq:ce_omega_L}
	CE(\Lambda;C_{*}(\Omega L^\circ)),
	\end{equation}
	where $\Omega L^\circ$ is the disjoint union of based loop spaces of the connected components of $L^\circ$.
	The algebra \eqref{eq:ce_omega_L} is obtained from the dg algebra
	$$
	CE(\Lambda;C_{*}(\Omega\Lambda))
	$$
	introduced in \cite{EL} using the push out diagram
	$$
	\begin{CD}
	C_{\ast}(\Omega\Lambda) @>>> CE(\Lambda,C_{\ast}(\Omega\Lambda))\\
	@VVV  @VVV\\
	C_{\ast}(\Omega L^{\circ}) @>>> CE(\Lambda,C_{\ast}(\Omega L^{\circ}))
	\end{CD}\;.
	$$
\end{rmk}

Fix a base-point $\star_i \in L^{\circ}_i$ on each connected component of $L^{\circ}$. Write
$$\C[\pi_1(L^{\circ})] \coloneqq \bigoplus_i \C[\pi_1(L^{\circ}_i)],$$
and view it as a semi-simple algebra, given as the direct sum of the group rings of the fundamental groups of the components of $L_i^\circ$ based at $\st_i$. 

Write $\mathcal{Q}(\Lambda)$ for the set of Reeb chords of $\Lambda$. Assume that $\Lambda$ is generic. Then the number of Reeb chords below any given  length is finite. In fact, since $\alpha$ is a perturbation of the round contact form $\alpha_{\mathrm{st}}$,  the number of Reeb chords below any given index is finite as well.  Consider the $\C$-vector space  $A(\Lambda)$ spanned by this set. The algebra underlying the Chekanov--Eliashberg dg-algebra is defined as follows. 

Start with the free unital graded tensor ring
$$ \mathcal{A} \coloneqq \bigoplus_{k=0}^\infty A(\Lambda)^{k \cdot \otimes_{\C[\pi_{1}(L^{\circ})]}}$$
over $\C[\pi_{1}(L^{\circ})]$ generated by the Reeb chords $\mathcal{Q}(\Lambda)$. Here
\begin{align*}
& A(\Lambda)^{0 \cdot \otimes_{\C[\pi_{1}(L^{\circ})]}}=\C[\pi_{1}(L^{\circ})],\\
& A(\Lambda)^{k \cdot \otimes_{\C[\pi_{1}(L^{\circ})]}}=\underbrace{A(\Lambda) \otimes_{\C[\pi_{1}(L^{\circ})]} \ldots \otimes_{\C[\pi_{1}(L^{\circ})]} A(\Lambda)}_k, \:\:k >0.
\end{align*}
{Here we use the convention that the product $l_1 \cdot l_2 \in \pi_{1}(L^{\circ})$ in the fundamental group is the loop formed by first traversing $l_1$ and then $l_2$.}

Define
$$ CE(\Lambda;\C[\pi_{1}(L^{\circ})]) \subset \mathcal{A}$$
to be the unital subalgebra generated by the elements in $\C[\pi_{1}(L^{\circ})] \subset \mathcal{A}$ together with elements of the form ${[\star_j]c[\star_i]}$, where $c \in \mathcal{Q}(\Lambda)$ is a Reeb chord starting from the component containing $\star_i$ and ending at the component containing $\star_j$.

We point out that there is an identification
$$ CE(\Lambda;\C[\pi_{1}(L^{\circ})]) = \mathcal{A}/I$$
where $I$ is the two-sided ideal generated by elements of the form $[\star_i]c$ and $c[\star_j]$, where $[\star_i],[\star_j] \in \pi_1(L^\circ)$ are constant loops in components \emph{different} from the starting point and endpoint of the chord $c$, respectively. This way the chords themselves  are still generators of the Chekanov--Eliashberg algebra and elements of $CE(\Lambda;\C[\pi_{1}(L^{\circ})])$ are linear combinations of
words 
\begin{equation}
\label{eq:ce_word}
l_1c_1\ldots l_{i-1}c_il_i\ldots l_k
\end{equation}
where $c_i\in \mathcal Q(\Lambda)$ is a Reeb chord on $\Lambda$
and $l_i\in\pi_1(L^\circ)$, that satisfies the following condition. For each $i\ge 1$,
$l_{i-1}$ belongs to  the connected component of $L^\circ$ corresponding to the starting point of the chord $c_i$ (via the identification $\Lambda=\del L^\circ$), and $l_i$ belongs to the connected component of the endpoint of $c_i$.

\begin{rmk}
	Whenever $\pi_1(L^{\circ})$ is nontrivial, the `coefficients' $\C[\pi_1(L^{\circ})]$ do not commute with the Reeb chord generators. 
	%The Chekanov--Eliashberg algebra  is naturally a $\C[\pi_{1}(L^{\circ})]$-bimodule.
\end{rmk}

To define the differential, one needs to fix reference paths between base-points and the endpoints of all Reeb chords. Specifically, for each Reeb chord $c\in \mathcal Q(\Lambda)$, let $i,j$ be the indices of the connected components of $L^\circ$ which contain the starting point resp.~the endpoint of $c$ (recall that those endpoints are contained in $\Lambda=\del L^\circ$). One fixes a path from $\st_i$ to the starting point of  $c$, and from its endpoint to $\st_j$.
The differential of the capped Chekanov--Eliashberg algebra satisfies the graded Leibniz rule, acts trivially on the coefficients, while its value on a Reeb chord generator $c \in \mathcal{Q}(\Lambda)$ is given by the following formula:
$$ \partial c= \sum_{u \in \mathcal{M}(c;\mathbf{c};\boldsymbol{\gamma})}l_1 c_1 l_2 \otimes \ldots  \otimes l_{m-1} c_m l_m.$$
Here the sum is taken over the set of rigid (up to translation) holomorphic disks
$$u \in (D^2,\partial D^2 \setminus \{p_0,\ldots,p_m\}) \to (\R \times S^{2n-1}, \R \times \Lambda)$$
for a generic cylindrical almost complex structure, where $u$ satisfies the following conditions:
\begin{itemize}
	\item Its boundary is continuous on $\R \times \Lambda$ away from the $m+1 \ge 1$ cyclically ordered boundary punctures $p_0,\ldots,p_m \in \partial D^2$.
	\item There is a single positive puncture $p_0$ near which $u$ is asymptotic to  the Reeb chord $c$ at $+\infty$.
	\item There is a (possibly empty) sequence of negative punctures $p_i$ where $u$ is asymptotic to the Reeb chord $c_i$ at $-\infty$.
	\item The arc of $\del u$ between the punctures $p_i$ and $p_{i+1}$ (where $i \in \Z_{m+1}$) yields the class $l_i \in \pi_1(L^\circ)$ when concatenated with the chosen reference paths to the base-points as in Figure~\ref{fig:dga_ce_ref}. 
	\item $u$ is a rigid up to translation.
\end{itemize}

According to our grading conventions, a disk is rigid up to translation whenever $|c|=|c_1|+\ldots+|c_m|+1$, so the differential has degree $-1$.

\begin{figure}[h]
	\includegraphics[]{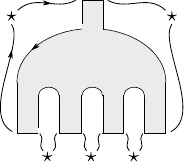}
	\caption{Attaching reference paths for the capped Chekanov--Eliashberg differential.}
	\label{fig:dga_ce_ref}
\end{figure}

\begin{prp}\label{p:dginv}
	The differential on $CE(\Lambda,\C[\pi_{1}(L^{\circ})])$ squares to zero and its homology is invariant under Legendrian isotopy.	
\end{prp}

\begin{proof}
	The first statement follows by identifying the terms contributing to $\del^2$ with the ends of the compact oriented 1-manifold from Lemma \ref{l:mfdwboundary}. The second statement is a consequence of Lemma \ref{l:invmdli} and the argument in \cite[Section 2.4]{EESPtimesR}.
\end{proof}

Let $L\subset \R\times S^{2n-1}$ be an exact Lagrangian cobordism interpolating between Legendrian submanifolds $\Lambda_{+}$ and $\Lambda_{-}$ as in Section \ref{ssec:geomstcontactsphere}. Let $L^{\circ}_+$ be a cap of $\Lambda_{+}$. Then $L^{\circ}_{-}=L\cup L^{\circ}_{+}$ is a cap of $\Lambda_{-}$. For simplicity we now assume that the inclusion $L^\circ_+ \subset L^\circ_-$ induces an isomorphism on the level of $\pi_0$. Define the cobordism map $\Phi\colon CE(\Lambda_{+},\C[\pi_{1}(L^{\circ}_{+})])\to CE(\Lambda_{-},\C[\pi_{1}(L^{\circ}_{-})])$ on generators as follows:
\[ 
\Phi_{L}(c)= \sum_{u \in \mathcal{M}^{\rm co}(c;\mathbf{c};\boldsymbol{\gamma})}l_1 c_1 l_2 \otimes \ldots  \otimes l_{m-1} c_m l_m,
\]
where $\mathcal{M}^{\rm co}(c;\mathbf{c};\boldsymbol{\gamma})$ is defined exactly as $\mathcal{M}(c;\mathbf{c};\boldsymbol{\gamma})$ above, replacing the target by $(\R\times S^{2n-1},L)$. Note that there is no $\R$-invariance here, so the curves contributing to the cobordism map satisfy $|c|=|\mathbf{c}|$.

\begin{prp}\label{p:chmap}
	Assume that the inclusion $L^\circ_+ \subset L^\circ_-$ induces an isomorphism on the level of $\pi_0$. Then the map $\Phi_{L}$ is a chain map of dg-algebras.
\end{prp}

\begin{proof}
	This follows from Lemma \ref{l:mfdwboundary}.
\end{proof}

Finally, we show that $CE(\Lambda;\C[\pi_{1}(L^{\circ})])$ can be computed within a Darboux ball. Placing $\Lambda$ in a Darboux ball, one defines a local version of the algebra by restricting to Reeb chords and holomorphic disks lying inside the Darboux ball. {(This defines a \emph{subcomplex} by a standard monotonicity argument.)} Note that there are only finitely many Reeb chord generators in this algebra.

\begin{prp}\label{p:local}
	The local dg-algebra is quasi-isomorphic to $CE(\Lambda;\C[\pi_{1}(L^{\circ})])$.
\end{prp} 

\begin{proof}
{
The Legendrian $\Lambda$ can be shrunk to $\epsilon\Lambda$ inside the Darboux ball by linear rescaling. This is a Legendrian isotopy that may be assumed to induce a canonical identification of local dg-algebras. Recall that the invariance result of Proposition \ref{p:dginv} is obtained by a dg-morphism which is associated to an exact Lagrangian cobordism in the symplectisation. The cobordisms is obtained from the trace of the isotopy. The associated dg-morphism $\phi$ is a quasi-isomorphism that preserves the local dg-algebra by a monotonicity argument, and which induces a quasi-isomorphism of both the local and full dg-algebras. In fact, there is an homology inverse $\phi^{-1}$ that is a dg-algebra morphism which satisfies the analogous properties.

Since the exact Lagrangian cobordism that defines $\phi$ goes from the larger Legendrian to the smaller, the induced dg-algebra morphism $\phi$ can be assumed to preserve the filtration by action induced by Reeb chord length.

We now consider the canonical inclusion of the local dg-algebra into the full. We need to show that this inclusion induces an isomorphism on the level of homology.

\emph{Injectivity:} Assume that $\partial a=b$ for some element $b$ in the local dg-algebra of $\Lambda$. In view of Lemma \ref{l:smallchords}, when $\epsilon>0$ is sufficiently small we may assume that all generators of $\epsilon\Lambda$ that leave the Darboux ball have length strictly greater than the action of $a$. Since the elements that are not in the local dg-algebra also have action greater than the action of $a$, it follows from the action-preserving properties that $\phi(a)$ must be contained in the local algebra. The chain-map property implies that $\partial \phi(a)=\phi(b)$ and, since $\phi$ is a quasi-isomorphism when restricted to the local algebra, we obtain $[b]=0$ and the inclusion is injective on homology.

\emph{Surjectivity:} Consider a cycle $a$ in the full dg-algebra. Using Lemma \ref{l:smallchords}, $\phi(a)$ may be assumed to be an element in the local dg-algebra if $\epsilon>0$ is chosen sufficiently small. It follows that $[a]=\phi^{-1}[a]$ is represented by the element $\phi^{-1}(a)$ in the local algebra and the inclusion is surjective on homology.}
\end{proof}

\subsection{Chekanov--Eliashberg algebra of the Hopf link}
\label{subsec:Hopf}
We now restrict to $n=2$. Let $\Lambda_{\rm Ho}\subset \R^3$ be the Legendrian Hopf link in a small contact Darboux ball, see Figures~\ref{fig:hopffront} and~\ref{fig:hopf}. (We use notation as there.) Inside $S^3$, it is Legendrian isotopic to the link $\Lambda_{\mathrm{st}}$ of the double point singularity discussed in Section~\ref{subsec:l_o}.
The Reeb chords of $\Lambda_{\rm Ho} \subset (\R^3,\alpha)$ inside the Darboux ball are  as follows:
\begin{itemize}
	\item $a_1$ and $a_2$ have degree $|a_i|=1$, and have both endpoints on a the same component of the link,
	\item $p$ and $q$ have endpoints on different components of the link. For a suitable choice of Maslov potentials on the two Legendrian components they satisfy $|p|=|q|=0$.
\end{itemize}
We compute the differential of the  Chekanov--Eliashberg dg-algebra of $\Lambda_{\mathrm{Ho}}$ with coefficients in the group algebra of its own fundamental group (see the references in Section~\ref{sec:dga}), $CE(\Lambda_{\rm Ho};\C[\pi_1(\Lambda_{\rm Ho})])$.
Here 
$$
\C[\pi_1(\Lambda_{\rm Ho})]=\C[t_1,t_{1}^{-1}]\oplus \C[t_2,t_{2}^{-1}]
$$
is a semi-simple algebra where $t_{1}$ and $t_{2}$ are the generators of the fundamental groups of the Hopf link components and the two corresponding idempotents are $e_1=(t_1)^0$, $e_2=(t_2)^0$. Since the Maslov class of the Hopf link vanishes, the group ring elements have degree zero. The disks contributing to the differential are easy to find in the Lagrangian projection of Figure \ref{fig:hopf} and gives the differential:
\begin{equation}
\label{eq:hopf_dif}
\begin{array}{l}
\partial a_1=e_1-t_1+qp,\\
\partial a_2=e_2-t_2+pq.
\end{array}
\end{equation}

Here we orient the link as the Hopf link with the orientation inducing \emph{positive} linking number. The signs also depend on the choice of a spin structure on each circle, and we choose the bounding spin structure. If we choose the other spin structure on either component, then the result is a change of sign in the coefficient of the variable $t_i$ of the corresponding component.

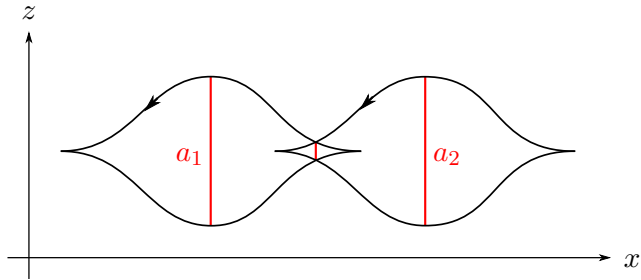
\begin{figure}[htp]
	\vspace{3mm}
	\labellist
	\pinlabel $z$ at 8 100
	\pinlabel $\color{red}a_1$ at 68 46
	\pinlabel $\color{red}a_2$ at 164 46
	\pinlabel $x$ at 233 7
	\endlabellist
	\includegraphics{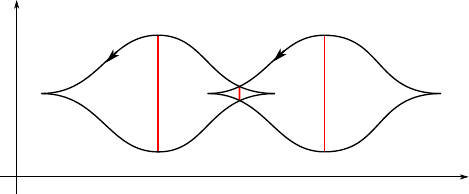}
	\caption{The front projection of   the Hopf link $\Lambda_{\rm{Ho}}$ in a Darboux chart.}
	\label{fig:hopffront}
\end{figure}

\begin{figure}[htp]
	\labellist
	\pinlabel $a_1$ at 65 42
	\pinlabel $t_1$ at 22 14
	\pinlabel $t_2$ at 218 57
	\pinlabel $p$ at 116 60
	\pinlabel $q$ at 116 24
	\pinlabel $a_2$ at 167 42
	\pinlabel $x$ at 233 0
	\endlabellist
	\includegraphics{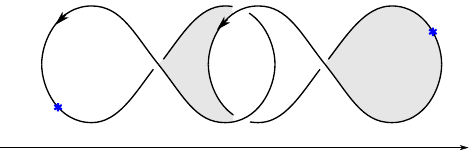}
	\caption{The Lagrangian projection (to the $xy$-plane) of the Hopf link in a Darboux chart. The shaded disc on the right contributes to $\partial a_1 =qp$, and the shaded disk on the right gives $\partial a_2=t_2$.}
	\label{fig:hopf}
\end{figure}

\subsection{Morsifying the Reeb chords}
\label{subsec:l_o_rho}
Recall the standard Hopf link $\Lambda_{\rm st}\subset S^{3}$ which models the negative ends of a punctured Lagrangian $L^\circ$ at the double points of $L$. Its Reeb chords come in $S^{1}$-families and are Bott-degenerate.
We Morsify by replacing $L^\circ$ with $L^\circ_\rho$ with ends 
modeled on the Morse Hopf link $\Lambda_{\rm Ho}$ (shown in Figures~\ref{fig:hopffront},~\ref{fig:hopf}) by attaching a cobordism to $L^\circ$.

\begin{lem}
	\label{lem:Gamma}
	The two Legendrian links $\Lambda_{\rm st},\Lambda_{\rm Ho} \subset S^3$ are Legendrian isotopic. Furthermore, there exists an isotopy whose  induced exact Lagrangian cobordism $$\Gamma\subset (\R\times S^{3},d(e^t\alpha)):$$
	\begin{itemize}
		\item agrees with $(-\infty,-T]\times \Lambda_{\rm Ho}$ in $(-\infty,-T]\times S^{3}$;
		\item agrees with $[T,+\infty)\times \Lambda_{\rm st}$ in $[T,+\infty)\times S^{3}$;
		\item intersects each slice $\{t\} \times S^3$ transversely; and
		\item admits a primitive of $e^t\alpha|_{\Gamma}$ which vanishes outside of a compact subset.
	\end{itemize}
\end{lem}
\begin{proof} First, take a generic Legendrian perturbation of $\Lambda_{\rm st}$ whose projection to $\CP^1$ under the Hopf fibration consists of a pair of figure eight curves. Then consider an appropriate smooth isotopy of such figure eight curves that lifts to a Legendrian isotopy.
	
	The exact Lagrangian cylinder can be constructed from the trace of the isotopy as in \cite{ElGro98}. In order to ensure the last property, we argue as follows. First, we can translate the non-cylindrical part of $\Gamma$ arbitrarily in the negative $t$-direction. This ensures that the difference of the potential between $t=\pm\infty$ at least is arbitrarily small. Finally, it is possible to cancel this potential difference completely by performing a small compactly supported perturbation of the cobordism inside $\{ t \in [-\epsilon,\epsilon]\}$, i.e.~the cylinder $[-\epsilon,\epsilon] \times \Lambda_{\rm Ho}$. Namely, one can use suitable Hamiltonian isotopies applied separately to the two components in that subset, e.g.~generated by Hamiltonians of the form $\rho(t)e^t$.
\end{proof}

Let $\Gamma_{T} \coloneqq \Gamma\cap (-\infty,T]\times S^{3}$ and denote by $\Gamma_{T,\rho}$ a copy of $\Gamma_{T}$ translated so that it lies in $(-\infty,\rho]\times S^{3}$. 
Define $$L^{\circ}_{\rho}\subset X^{\circ}$$ to be the Lagrangian submanifold obtained by removing $(-\infty,-\rho]\times \Lambda_{\rm st}$ in each negative end of $L^{\circ}$, using the model (\ref{eq:neg_cyl}), and inserting $\Gamma_{T,\rho}$ in its place.
This construction applies for any $\rho<2\log\epsilon$ where $\epsilon$ is the size of the Darboux ball from Section~\ref{subsec:Hopf}.

The Maslov index of a disk inside $(X^\circ,L^\circ)$ with strip-like ends asymptotic to Reeb chords on $\Lambda_{\rm Ho}$ can be defined as in Section~\ref{sec:corners}. More precisely, at a \emph{mixed puncture} (i.e.~for which the Reeb chord goes between different components of the Legendrian) we again close up the tangent planes using the shortest rotation along the negative K\"{a}hler angle. At a \emph{pure puncture} (i.e.~for which the Reeb chord has both endpoints on the same component) we instead close up with the path of Lagrangian tangent planes along any choice of \emph{capping path} on the Lagrangian $\R \times \Lambda_{\rm st}$ connecting the endpoint of the Reeb chord with its starting point.

We next introduce continuous paths $G_{p}$ and $G_{q}$ of Lagrangian tangent planes along $\Gamma \subset \C^2 \setminus \{0\}$. We start with $G_{p}$. The Reeb chord $p$ at the negative end of $\Gamma$ connects one component $\Gamma_{0}$ of $\Gamma$ to the other component $\Gamma_{1}$. Let $\theta_{0}$ and $\theta_{1}$ be points in the components of $\Lambda_{\rm st}$ at the positive ends of $\Gamma_{0}$ and $\Gamma_{1}$, respectively, such that there is a length $\frac{\pi}{2}$ Reeb chord connecting $\theta_{0}$ to $\theta_{1}$. Consider path on $\Gamma$ that starts at $\theta_{1}$ and goes down to the endpoint of the chord $p$ in the negative end, and take the path of Lagrangian tangent planes to $\Gamma$ along this path. Transport the Lagrangian tangent plane at the endpoint of $p$ first with the linearized Reeb flow backwards along the Reeb chord $p$ and then with a rotation along the negative K\"{a}hler angle at $p$ to the tangent plane at the start-point of $p$. Concatenate finally with the tangent planes along a path from the start-point of $p$ to $\theta_{0}$. This gives a path of Lagrangian tangent planes connecting the tangent plane to $\Gamma$ at $\theta_{1}$ to that at $\theta_{0}$. The path $G_{q}$ is defined similarly replacing $p$ by $q$ and using a length $\frac{\pi}{2}$ Reeb chord from $\Gamma_{1}$ to $\Gamma_{0}$ at the positive end.

\begin{lem}
	\label{lem:MaslovCobordism}
	The punctured Lagrangian $L^{\circ}_{\rho}$ is monotone. Moreover, the above paths $G_p$ and $G_q$ of Lagrangian tangent planes are homotopic relative endpoints to a rotation along the negative K\"{a}hler angle.
\end{lem}
\begin{proof}
	The last property of	
	Lemma~\ref{lem:Gamma} implies that  the symplectic area of a disk with boundary on $L^{\circ}_{\rho}$ equals the symplectic area of the corresponding disk with boundary and corners on $L$. 
	The assertion about the paths $G_p$, $G_q$ implies that
	the Maslov index of a disk with boundary on $L^{\circ}_{\rho}$ is equal to that of the corresponding disk with boundary and corners on $L$, provided the cobordism $\Gamma$ has a Maslov class satisfying the following property: the value of the Maslov class on any path from the positive to the negative end of $\Gamma$ is the \emph{same} for the two components of the cobordism. So monotonicity follows from the statement about $G_p$ and $G_q$. 
	
	To see that this statement is true, note that $\Gamma$ is a cobordism constructed from an isotopy and consider the cobordism $\Gamma'$ of the inverse isotopy with positive end at $\Lambda_{\rm Ho}$ and negative at $\Lambda_{\rm st}$. Gluing the two gives a cobordism $\Delta$ that can be deformed by shortening the isotopy to the trivial cobordism $\R\times\Lambda_{\rm st}$. We denote the 1-parameter family of cobordism $\Delta_{s}$, $s\in[0,\infty)$, where $\Delta_{0}$ is the trivial cobordism and $\Delta_{\infty}$ is the two level cobordism $\Gamma\cup\Gamma'$. 
	
	Consider the moduli space of holomorphic strips with boundary on $\Delta_{s}$ asymptotic in the positive end to a fixed Reeb chord of length $\frac{\pi}{2}$, connecting the component of $\Lambda_{\rm st}$ in $\Gamma_{0}$ to that in $\Gamma_{1}$, to an unconstrained Reeb chord of the same type in the negative end. It is straightforward to check that this is a transversely cut out moduli space of dimension $0$ for the trivial cobordism at $s=0$. Since the Reeb chords at the ends have minimal action no breaking is possible for $s\in[0,\infty)$. Therefore by SFT compactness the strip breaks into a several level disk at $s=\infty$. As above it is straightforward to show that moduli spaces of disks with one positive puncture are transversely cut out. Noting that $p$ is the only chord of $\Lambda_{\rm Ho}$ that connect $\Gamma_{0}$ to $\Gamma_{1}$, this then implies that there is a unique dimension zero strip connecting a fixed chord in the positive end to $p$ (and a dimension zero strip from $p$ to an unconstrained chord at the negative end). The statement about the Maslov class of $\Gamma$ is then a consequence of the dimension formula for punctured holomorphic curves.
\end{proof}

We study holomorphic disks $u\co (D,\partial D)\to (X^{\circ},L^{\circ}_{\rho})$ that are asymptotic to Reeb chords of $\Lambda_{\rm Ho}$ at boundary punctures. We continue to call them punctured holomorphic disks. The lemma below is an instance of the SFT compactness theorem for neck-stretching \cite{CompSFT03}.

\begin{lem}\label{l:sftcompact}
	Let $u_{\rho}$, $\rho\to +\infty$, be a family of punctured holomorphic disks in $(X^\circ, L_{\rho}^{\circ})$. As $\rho\to+\infty$, the sequence $u_{\rho}$ converges in the SFT sense to a holomorphic building whose top level is a punctured holomorphic disk in $(X^{\circ},L^\circ)$, and whose lower levels are holomorphic buildings in $\R\times S^{3}$ with boundary on $\Gamma$.\qed
\end{lem} 

From now on,  denote by $\Lambda$ the collection of  Hopf links $\Lambda_{\rm Ho}$ corresponding to the punctures of $L^\circ_\rho$ where each copy of the Hopf link sits in a separate copy of $S^3$. For example, a Reeb chord of $\Lambda$ means a Reeb chord on one of the copies of the actual Hopf link in the respective copy of $S^3$. For our choice of Maslov potential, the Reeb chords on $\Lambda_{\rm Ho}$ contained in the Darboux chart have degrees $|a_i|=1$ and $|p|=|q|=0$. 

Let  $c_{1},\dots, c_{m}$ be a collection of Reeb chords on the Hopf links $\Lambda$ at the negative ends of $L^{\circ}_{\rho}$. Recall that each of these Reeb chords is a copy of $a_1$, $a_2$, $p$ or $q$ from Section~\ref{subsec:Hopf}, or a long Reeb chord of index $\ge 2$ which leaves the Darboux chart.

The expected dimension of the moduli space of holomorphic disks $u$ with boundary on $L^{\circ}_{\rho}$ and boundary punctures asymptotic to $c_{1},\dots,c_{m}$ equals
\[ 
\mu(u)-1-\sum_{j=1}^{m}|c_{j}|.
\]
Here we have used the computation of the Maslov class of the cobordism $\Gamma$ in Lemma \ref{lem:MaslovCobordism}.

\begin{rmk}Comparing this dimension formula with Lemma~\ref{l:tv}, the chords $p,q$ resp.~$a_1,a_2$ correspond to winding number $1/2$ resp.~$1$.
\end{rmk}

Fix a point $\zeta$ and write  \begin{equation}
\label{eq:m_zeta_punct}
\M(\zeta;c_{1},\dots,c_{m})
\end{equation} for the moduli space of holomorphic disks $u$ passing through $\zeta$ at a boundary point,  and with boundary punctures asymptotic to $c_{1},\dots,c_{m}$. 
Its expected dimension equals
\[ 
\mu(u)-2-\sum_{j=1}^{m}|c_{j}|. 
\]
Below is an analogue of Lemma~\ref{l:tv} for punctured curves.

\begin{lem}\label{l:tvpunctured}
	For any fixed  $\rho>0$ there is a uniform upper bound on the number of punctures of holomorphic disk with boundary on $L_{\rho}^\circ$ in terms of the $\omega$-area of its topological compactification in $(X,L)$. 
	
	Suppose $\rho>0$ is sufficiently large and $L\subset X$ is a monotone immersion.
	Then for a generic compact $s$-dimensional family $S$ of data defining the moduli spaces (\ref{eq:m_zeta_punct}), the union of the components over $S$ consisting of Maslov index 2 disks is a transversely cut out manifold of dimension
	\[ 
	s-\sum_{j=1}^{m}|c_{j}|.
	\] 
	In particular, all punctures of a Maslov index 2 disk in $\M(\zeta;c_{1},\dots,c_{m})$  are asymptotic to degree 0 Reeb chords on $\Lambda$, that is, one has $|c_i|=0$. Equivalently, each $c_i$ is either the $p$-chord or the $q$-chord from Section~\ref{subsec:Hopf} on a copy of the Hopf link.
\end{lem}

\begin{proof}
	Every punctured holomorphic disk $u$ with boundary on $L_\rho^\circ$ and  asymptotic chords  $\{c_i\}$ satisfies 
	$$\omega(u)\ge \sum_i\mathfrak{a}(c_i),$$
	where the action of $\mathfrak{a}(c_i)$ is computed in the sphere $S^3$ of radius $\epsilon$. (These actions are scaled compared to the actions appearing in Section~\ref{subsec:Hopf}.) Since the set of all possible actions is separated from zero, the first claim follows.
	
	To see the second claim, it again suffices to prove that all holomorphic disks considered are somewhere injective, and hence transversely cut out for generic data. Suppose that for any $\rho$ there exists a disk $u_{\rho}$ that is not somewhere injective. Then, taking the SFT limit as in Lemma~\ref{l:sftcompact}, the top level of the resulting holomorphic building corresponds by Lemma~\ref{l:outside=compact} to a holomorphic Maslov index~2 disk in $(X,L)$ with corners which would not be somewhere injective. That however contradicts Lemma~\ref{l:somewhereinj}. 
\end{proof}

\subsection{The refined potential}
\label{subsec:ce_pot}
Let $Q$ be an oriented spin surface, and  $f\colon Q\to X$ a monotone Lagrangian immersion. Write $L=f(Q)$ and let $L^{\circ}_{\rho}\subset X^{\circ}$ be as above, for a sufficiently large $\rho>0$.
Choose base-points and reference paths as in Section~\ref{sec:dga}.

Fix a connected component of $Q$ with index $i$ and write $\zeta=\st_i$.
Let $\M(\zeta)$ be the moduli space of punctured holomorphic disks of Maslov index~2  with  boundary on $L^{\circ}_{\rho}$, and passing through $\zeta$. This is the union of spaces (\ref{eq:m_zeta_punct}) over all index~0 Reeb chords.

Lemma \ref{l:tvpunctured} implies that $\M(\zeta)$ is a transversely cut out oriented $0$-manifold. 
Let $u\in \M(\zeta)$. Following $u(\del D)$ starting from the marked point mapping to $\zeta=\st_i$, we read off a word of loops in $L_\rho^\circ$ and Reeb chords as in \eqref{eq:ce_word}.  Here the chords  $c_i$ are the asymptotics of $u$ and the loops are the boundary arcs between them completed by reference paths to base-points, see Figure~\ref{fig:pot_ce_ref}.

This way we associate to each such $u$ an element of $CE(\Lambda;\C[\pi_{1}(L_{\rho}^{\circ})])$, where $\Lambda$ is the formal union of the Hopf links at the negative ends of $L^{\circ}$. We denote it $[\del u]$. In fact, since
$[\del u]$ begins and ends at the base-point $\zeta=\st_i$, 
\begin{equation}
\label{eq:ce_submodule}
[\del u]\in
CE(\Lambda;\C[\pi_{1}(L_{\rho}^{\circ})],\st_i) \coloneqq
e_i CE(\Lambda;\C[\pi_{1}(L_{\rho}^{\circ})]) e_i
\end{equation}
where $e_i$ is the idempotent corresponding to the $i$th connected component of $L_\rho^\circ$. Note that
$$ CE(\Lambda;\C[\pi_{1}(L_{\rho}^{\circ})],\st_i) \subset CE(\Lambda;\C[\pi_{1}(L_{\rho}^{\circ})])$$ 
is a subalgebra of the Chekanov--Eliashberg algebra.

Define the refined potential of $L$ as follows:
\[ 
W_{L}(\st_i)=\sum_{u\in\M(\zeta)}(-1)^{\sigma(u)} [\del u]\ \in\  CE(\Lambda;\C[\pi_{1}(L_{\rho}^{\circ})],\st_i),
\]  
$(-1)^{\sigma(u)}$ is the  orientation sign of $u$. (Recall that the algebra $CE(\Lambda;\C[\pi_{1}(L_{\rho}^{\circ})])$ has a $\Z$-grading where generators in $\C[\pi_{1}(L_{\rho}^{\circ})]$ have degree zero.)

\begin{prp}\label{th:potdga}
	The potential $W_{L}(\st_i)$ is a degree 0 cycle in $CE(\Lambda;\C[\pi_{1}(L_{\rho}^{\circ})])$. Its homology class is independent of the choices made 
	that is invariant under symplectomorphisms of $L$.
\end{prp}

\begin{rmk}
	Potentials with respect to two different base-points $\st_i$, $\st_j$ (on different connected components of $L^\circ_\rho$) are related by the Lagrangian surgery formula, see Proposition~\ref{prp:surg} and 
	Lemma~\ref{lem:onedimrepr}.
\end{rmk}

\begin{proof}
	The fact that $W_L(\st_i)$ has degree zero follows from the dimension formula in Lemma~\ref{l:tvpunctured}. It is automatically  a cycle because the Chekanov--Eliashberg algebra of the Hopf link is concentrated in degrees~0 and~1.
	
	For the invariance, consider the moduli space $\M_{I}(\zeta)$ where $I$ is an interval that parameterizes monotone Lagrangians $L^{s}$, $s\in I$ (or a variation of other choices). By SFT compactness, this moduli space has a compactification consisting of broken disks. An elementary topological consideration of the symplectic area, using the monotonicity assumption, shows that  the upper level of a broken disk again must consist of a single (punctured) Maslov index~2 disk. 
	
	This implies that the lower part of the building contains exactly one disk of Fredholm index~one, with one positive puncture at a Reeb chord of degree $1$ which hence must be an $a_i$-chord contained inside the Darboux ball. We use Lemma \ref{l:tvonepos} to control the lower level of the disk.  Counting the ends of the compact 1-dimensional moduli space $\M_{I}(\zeta)$ then gives 
	\[ 
	W_{L^{1}}(\st_i)-W_{L^{0}}(\st_i)= \partial(K_{I}), 
	\]
	where $K_{I}$ is the count of curves in $\M_{I}(\zeta)$ with one puncture at a degree $1$ chord, and $\partial$ is the differential of the Chekanov--Eliashberg algebra.
\end{proof}

{
\subsection{Computing the homology for the cylinder cap}

In this section we compute the full homology of the DGA
$$CE(\Lambda_{\rm Ho};\C[\pi_1(L^\circ)])$$
where $L^\circ$ is a cap of the Hopf link $\Lambda_{\rm Ho}$ diffeomorphic to a cylinder. Similar results were announced in \cite{KaplanSchedler}. 
%The computation will be carried out with coefficients in $\Z$. 

Recall that the underlying unital graded algebra in this case is
$$ \mathcal{A}=\C\langle a_1,a_2,p,q,t,t^{-1} \rangle $$
where $|a_i|=1$ and the remaining generators are in degrees zero. More precisely, it is the free tensor ring over $\C[t,t^{-1}]$ generated by $a_1$, $a_2$, $p$, and $q$. (In particular, it is not a semi-free DGA over $\C$.) The nonzero differentials are
$$ da_1=1-t+qp \:\:\:\text{and}\:\:\: da_2=1-t+pq.$$
Since there are no generators in negative degrees, the homology of the DGA in nonpositive degrees is easily computed; as pointed out before, it is the affine algebra
$$H_0(CE(\Lambda_{\rm Ho};\C[\pi_1(L^\circ)]))= \C[p,q,t]/\langle 1-t+pq\rangle.$$
We show that the homology in degrees different than zero is trivial.
\begin{thm}\label{t:Hopf=0highdeg}
\label{thm:homology}
The homology of the Chekanov--Eliashberg algebra of the Hopf link with the cylinder cap is
$$H(CE(\Lambda_{\rm Ho};\C[\pi_1(L^\circ)]))=\C[p,q][(1+pq)^{-1}],$$
where $\C[p,q][(1+pq)^{-1}]$ is the localization of $\C[p,q]$ at the multiplicative system $\{(1+pq)^{n}\}_{n=1,2,\dots}$, which is a smooth affine commutative algebra concentrated in degree zero.
\end{thm}
\begin{rmk}
The calculation of the degree zero homology in Theorem \ref{thm:homology} is straightforward. It is more involved to show that homology vanishes in all positive degrees.
\end{rmk}

The proof of Theorem \ref{t:Hopf=0highdeg} is a combination of Lemma \ref{lem:homology} on the homology of a sub-DGA of $\mathcal{A}$ and general algebraic results from \cite{BCL18}. Consider the \emph{semi-free} sub-DGA
$$ \mathcal{A}_0=\C\langle a_1,a_2,p,q,t \rangle \subset \mathcal{A},$$
in which $t$ is no longer invertible. Then $\mathcal{A}_{0}$ is freely generated over $\C$ by $a_1$, $a_2$, $p$, $q$, and $t$. We also consider the semifree sub-DGA
$$ \mathcal{B} = \C \langle a \coloneqq a_1-a_2,p,q\rangle \subset \mathcal{A}_0,$$
generated by $a=a_1-a_2$ in degree 1, and $p,q$ in degree $0$. The algebra $\mathcal{B}$ then has the nonzero differential
$$ da=pq-qp, $$
i.e., $da$ equals the commutator of the variables $p$ and $q$. 

\begin{lem}
\label{lem:homology}
The homology $H_i(\mathcal{A}_0)$ is trivial in all nonzero degrees $i  \neq 0$.
\end{lem}
\begin{proof}
	It is straightforward to see that the inclusion $\mathcal{B} \subset \mathcal{A}_0$ is a stable-tame equivalence of semi-free DGAs, and thus a quasi-isomorphism. Therefore, it suffices to show that
    the homology $H_i(\mathcal{B})$ for all $i \neq 0$.

Equip $\mathcal{B}$ with the bidegree $\Z \times \Z$ where a generator $c$ lives in bidegree $\|c\|=(-1,|c|+1)$. Then the total degree agrees with the usual degree of a generator of $\mathcal{B}$.

We have $\|a\|=(-1,2)$ and the differential preserves the second bigrading, $\|da\|=(-2,2)$. It follows that the graded dual complex $\mathcal{B}^\#$ is a complex which is isomorphic to $\mathcal{B}$ with field coefficients. On the other hand, the graded dual complex is the graded dual of the bar resolution of the exterior tensor algebra $A$ of a two-dimensional vector space (with each generator in bidegree $(0,1)$). By a standard computation, the homology of $\mathcal{B}^\#$ is the classical Koszul dual of $A$. In other words, comparing with the small resolution, $H\mathcal{B}^\#$ is a commutative polynomial algebra of two variables concentrated in total degree zero; see \cite[Section 2.3]{KoszulEquivalences}.
\end{proof}

\begin{proof}[Proof of Theorem \ref{thm:homology}]
There is a canonical unital (non-central) sub-DGA $\C[t]\subset\mathcal{A}_0$ that allows us to consider $\mathcal{A}_0$ as a semi-free DGA over $\C[t]$. The homology $H(\mathcal{A}_0)=\C[p,q]$ is commutative by Lemma \ref{lem:homology} and hence the induced unital algebra inclusion $\C[t] \hookrightarrow H(\mathcal{A}_0)$, $t \mapsto 1+pq$ is central. Then \cite[Theorem 5.3]{BCL18} applies and shows that the derived localization $\mathcal{A}_0 \star^{\mathbb{L}}_{\C[t]} \C[t,t^{-1}]$ has homology
$$ H(\mathcal{A}_0 \star^{\mathbb{L}}_{\C[t]} \C[t,t^{-1}])=\C[p,q][(1+pq)^{-1}].$$

The first bullet of \cite[Remark 3.11]{BCL18} shows that the derived localization
$\mathcal{A}_0 \star^{\mathbb{L}}_{\C[t]} \C[t,t^{-1}]$ is quasi-isomorphic to the free product
$\mathcal{A}_0' \star_{\C[t]} \C[t,t^{-1}]$,
where $\mathcal{A}_{0}'$ is a cofibrant replacement of $\mathcal{A}_{0}$ over $\C[t]$. Since $\mathcal{A}_0$ is free and hence cofibrant over $\C[t]$ and 
$$\mathcal{A}=\mathcal{A}_0 \star_{\C[t]} \C[t,t^{-1}]$$ the theorem follows.
\end{proof}

\begin{rmk}
Building on forthcoming work \cite{AE} one can give a geometric proof of Theorem \ref{t:Hopf=0highdeg} that we sketch here. In \cite{AE}, Weinstein sectors are studied from the Legendrian surgery perspective and one of the main results there is the following Legendrian surgery version of the cosheaf property of wrapped Floer cohomolgy in \cite{GPS17b}: consider Weinstein $2n$-manifolds $W_0$ and $W_1$ with a Weinstein $(2n-2)$-submanifold $V\subset \partial W_j$, $j=0,1$ and Legendrians $\Lambda_{j}\subset \partial W_{j} \setminus V$. Then there is a push-out diagram of dg-algebras
\begin{equation}\label{eq:CEpushout}
\begin{CD}
CE(\Lambda_{V},V^0) @>>> CE(\Lambda_{1}, V,W_{1})\\
@VVV  @VVV\\
CE(\Lambda_{0}, V,W_0) @>>> CE(\Lambda_{0}\cup \Lambda_{1}\cup \Lambda_{W_{01}},W_{01}^0).
\end{CD}
\end{equation}
Here $CE(\Lambda_{V},V^0)$ is the Chekanov-Eliashberg algebra of the attaching spheres of the top-dimensional cells in $V$ in the boundary $\partial V^{0}$ of its subcritical part $V^{0}$, $CE(\Lambda_{j},V_j \cup V,W_{j})$ is the Chekanov-Eliashberg algebra of the union of $\Lambda_{j}$ and the Legendrian embedding of the skeleton of $V$ in $\partial W_{j}$, and $CE(\Lambda_{0}\cup \Lambda_{1}\cup \Lambda_{W_{01}},W_{01}^0)$ is the Chekanov-Eliashberg algebra of the union of $\Lambda_{0}\cup \Lambda_{1}$ and the attaching spheres $\Lambda_{W_{01}}$ of the top dimensional disks in the boundary $\partial W_{01}^{0}$ of the complement if the top-dimensional handles corresponding to $V$ in the union $W_{01}=W_{0}\cup_{V} W_{1}$. Furthermore, $CE(\Lambda_{j},V,W_{j})$ is (via Legendrian surgery) quasi-isomorphic to the dg-algebra of $\Lambda_{j}$ in $W_{j}$ stopped at $V$. 

For Theorem \ref{t:Hopf=0highdeg}, we apply this result with $W_{0}$ the sub-critical Weinstein manifold $W_{0}=S^1\times \R^{3}$ obtained by adding a $1$-handle to the ball with its two ends on the Hopf link. Joining the two components over the handle gives a Legendrian knot $\Lambda_{0}\subset \partial W_{0}$. We add a stop along the standard Legendrian knot $\Lambda$ that goes once through the handle, $\Lambda\cap \Lambda_{0}=\varnothing$. It is easy to see that $\partial W_{0}$ stopped at $\Lambda$ (or in the notation above stopped at $V\approx T^{\ast}\Lambda\subset J^{1}\Lambda$, where a neighborhood of the zero-section in $J^{1}\Lambda$ is identified with a neighborhood of $\Lambda\subset\partial W_{0}$) is contactomorphic to $J^{1}S^{1}$ with $\Lambda$ in it. Furthermore, the holomorphic polygons of the knot diagram of $\Lambda_{0}\subset T^{\ast}S^{1}$ shows that $CE(\Lambda,J^{1}S^{1})$ is isomorphic to $\mathcal{A}_{0}$. 

We next apply the push-out diagram \eqref{eq:CEpushout} with $W_{1}$ as in stop removal. That is, 
$W_{1}=T^{\ast}S^{1}\times[0,1]$ with $V$ a neighborhood of $S^{1}\times\{0\}$ which is a loose Legendrian. Then $W_{01}=W_{0}$. Since, $V=T^{\ast}\Lambda=T^{\ast}S^{1}$, the top left corner is the dg-algebra of two points in the circle with homology $\C[t,t^{-1}]$, corresponding to chains on based loops in $S^{1}$ by the Legendrian surgery formula and the fact that the wrapped Floer cohomology of the cotangent fiber is isomorphic to chains on the based loop space of the base. We then get the diagram 
\[
\begin{CD}
CE(2 \text{ points}, \C)\quad @>>> \quad CE(V=(T^{\ast} S^{1})\times\{0\}, T^{\ast}S^{1}\times[0,1])\\
@VVV  @VVV\\
CE(\Lambda,V=T^{\ast}\Lambda_{0},W_{0})\quad @>>> \quad CE(\Lambda,W_{0}),
\end{CD}
\] 
where $CE(\Lambda,T^{\ast}\Lambda_{0},W_{0})$ is quasi-isomorphic to $\mathcal{A}_{0}$ and $CE((T^{\ast} S^{1})\times\{0\}, T^{\ast}S^{1}\times[0,1])$ is quasi-isomorphic to the zero-algebra, since $S^{1}\times\{0\}$ is loose.
Note that $\mathcal{A}$ and $CE(\Lambda,W_0)$ are quasi-isomorphic since (by Legendrian surgery) both computes the wrapped Floer cohomology of the co-core in a Weinstein neighborhood of the Whitney sphere. It follows that the homology is zero in non-zero degrees.   
\end{rmk}

\subsection{Preprojective algebras as the homology of Chekanov--Eliashberg algebras}
\label{sec:preprojective}

In this subsection we represent the degree zero homology of a Chekanov--Eliashberg algebra of a collection of Legendrian Hopf links whose coefficients are induced by a cap $Q^\circ$ by a multiplicative preprojective algebra. 
%see Proposition \ref{prp:ce_and_ppa} which is the main result of this subsection. 
Roughly speaking, $Q^\circ$ can be seen as the smooth locus of a Lagrangian immersion with transverse positive double points, and the multiplicative preprojective algebra will be associated to a quiver $\Gamma$ that arises from the plumbing graph that describes the configuration of double points of this immersion. That this is the case is in fact no big surprise; it should be expected from the calculations in \cite{EL17} of the wrapped Floer homology of the Weinstein manifold given by the plumbing. Indeed, the article shows that the wrapped Floer homology also is the corresponding (derived) PPA.

We first recall the definition of higher genus multiplicative preprojective algebras (PPA), which is an algebra naturally associated to a quiver $\Gamma$. These were originally introduced in \cite{CBS06} in the genus zero case, the higher genus generalizations then appeared in \cite{BK16}. Consider a quiver $\Gamma$  with vertex set $V$ and edge set $E$. To each vertex $v\in V$ we associate an auxiliary integer $g_v\ge 0$ called the genus. The edges of a quiver are oriented, and for $a\in E$ we let $s(a)$ and $t(a)$ be its source and target vertices, respectively.

We next define the PPA $\B(\Gamma)$ associated to $\Gamma$ with genera $g_{v}$ at vertices. Consider the semi-simple ring 
$$
\kk=\bigoplus_{v\in V}\C e_v
$$
where $e_v^2=e_v$ and $e_ve_w=0$ for $v\neq w$. The algebra $\B(\Gamma)$ will be a quotient of the tensor ring of the $\kk$-bimodule $B(\Gamma)$ with generators
$$
\begin{array}{lll}
p_a,\ q_a,&a\in E,\\
t_a^{\pm 1},& a\in E,\\
x_{i,v}^{\pm 1},\ y_{i,v}^{\pm 1}& v\in V,&i=1,\ldots,g_v,
\end{array}
$$
and bimodule structure over $\kk$ determined by
$$
\begin{array}{l}
e_v x_{i,v}=x_{i,v}e_v=x_{i,v},\\ e_vy_{i,v}=y_{i,v}e_v=y_{i,v}, \\
e_{t(a)}p_a=p_ae_{s(a)}=p_a,
\\
e_{s(a)}q_a=q_ae_{t(a)}=q_a,\\
e_{s(a)}t_a=t_ae_{s(a)}=t_a,
\end{array}
$$
and all other bimodule actions trivial, $e_w x_{i,v}=0=x_{i,v}e_w$ for $v\neq w$, etc. (In other words, the generators $p_a$ should be interpreted a a path from $s(a)$ to $t(a)$, $q_a$ should be interpreted as a path from $t(a)$ to $s(a)$, while $t_a$ is a path from $s(a)$ to itself.)

The PPA $\B(\Gamma)$ is then
\begin{equation}\label{eq:defppa}
\B(\Gamma) \coloneqq \kk \oplus \bigoplus_{i=1}^\infty \underbrace{B(\Gamma) \otimes_\kk \ldots \otimes_\kk B(\Gamma)}_i / \sim
\end{equation}
where the relation $\sim$ is as follows. 

Write
$$R_{g_v}=\prod_{i=1}^{g_v}[x_{i,v},y_{i,v}].$$
(I.e., the group ring of a closed surface of genus $g_{v}$ is equal to the quotient of the free algebra generated by $x_{i,v},y_{i,v}$ by the two-sided ideal generated by $R_{g_v}$.) The relation $\sim$ is then generated by
\begin{align*}
& t_at_a^{-1}=t_a^{-1}t_a=e_{s(a)},& a \in E,\\
& x_{i,v}x_{i,v}^{-1}=y_{i,v}y_{i,v}^{-1}=e_v,& v\in V,\\
& e_{s(a)}+q_ap_a=t_a, & a \in E,\\
& \prod_{t(a)=v}(e_v+p_aq_a)=R_{g_v}^{-1}\cdot \prod_{s(a)=v}(e_v+q_ap_a),& v\in V,
\end{align*}
where the factors in the above products are with respect to some \emph{fixed} order (the choice of which is irrelevant up to isomorphism \cite{CBS06}).
\begin{rmk}
\label{rmk:invertible}
The first and third relations imply that every element $e_{s(a)}+q_ap_a=t_a$ is invertible. A simple induction argument based on this fact together with the last relation above then implies that each factor $e_{t(a)}+p_aq_a$ is invertible as well.
\end{rmk}

As a vector space, $\B(\Gamma)$ splits as
$$
\B(\Gamma)=\bigoplus_{v,w\in V}e_v\B(\Gamma) e_w.
$$
We denote
$$
\B(\Gamma,v,w)\coloneqq e_v\B(\Gamma) e_w.
$$
Note that $\B(\Gamma,v,v)$ is a $\C$-algebra for each vertex $v$. We think of it as supported over loops in $\Gamma$ from the vertex $v$ to itself.

We next explain how to realize $\B(\Gamma)$ as the degree zero homology of the Chekanov--Eliashberg algebra of a suitable collection of Legendrian Hopf links. From the underlying graph of $\Gamma$ together with the genera $\{g_v\}$ one constructs a self-plumbing $X_\Gamma$ of the cotangent bundle of the disjoint union of the surfaces that correspond to the nodes. More precisely, we use the underlying graph of $\Gamma$ (i.e., forgetting the orientation of the edges) as a plumbing graph with the convention that all intersections are positive for a suitable orientation of the surfaces. Denote by $Q_\Gamma \subset X_\Gamma$ the image of the zero sections in this plumbing, which then is a strongly exact Lagrangian immersion with transverse double points in bijective correspondence with the edges of $\Gamma$.

Consider a disjoint union of small standard contact spheres $S_{b}$ inside this plumbing which are in bijection with the edges $b$ of $\Gamma$ and such that each sphere surrounds a unique double point of $Q_\Gamma$. Then for each $b$, $S_{b}\cap Q_\Gamma$ is a standard Legendrian Hopf link. We write $\Lambda_\Gamma$ for the union of all these Legendrian Hopf links. (If one prefers to work with a connected contact manifold at the negative end, one can simply perform a contact connected sum of all the contact spheres giving a single standard contact sphere.)

Recall the definition of the capped Chekanov--Eliashberg algebra
$$CE(\Lambda_\Gamma;\C[\pi_1(Q_\Gamma^\circ])$$
from Sections \ref{sec:dga} and \ref{subsec:Hopf}, where $\Lambda_\Gamma$ is the collection of Hopf links in a disjoint union of spheres indexed by the set of edges $E$, and where $Q_\Gamma^\circ$ is the capping surface. We represent each Hopf link by a standard representative $\Lambda_{\OP{Ho}} \subset S^3$ contained in a small contact Darboux ball, see Section \ref{subsec:Hopf}.

We finally use the orientations of the edges of $\Gamma$. Consider an oriented edge $b$ connecting $v_{1}$ to $v_{2}$ and the corresponding Hopf link $\Lambda_{\mathrm{Ho}}\approx S_{b}\cap Q_{\Gamma}$. After a rotation of $\Lambda_{\mathrm{Ho}}$ shown in Figure \ref{fig:hopf}, we may assume that the Reeb chord $p$ that starts on the leftmost unknot component and ends on the rightmost one has the property that it starts and ends on the components of $Q_\Gamma^\circ$ corresponding to $v_{1}$ and $v_{2}$, respectively.

\begin{prp}\label{prp:ce_and_ppa}
The PPA $\B(\Gamma)$ associated to the graph $\Gamma$ with genera $\{g_v\}$ at vertices is isomorphic to the degree~0 homology of the corresponding capped Chekanov--Eliashberg algebra described above, i.e.
	$$
	\B(\Gamma) \cong H_0(CE(\Lambda_\Gamma;\C[\pi_1(Q_\Gamma^\oo)])).
	$$ 
\end{prp}
\begin{proof}
We define a map
$$H_0(CE(\Lambda_\Gamma;\C[\pi_1(Q_\Gamma^\oo)])) \ \to \ \B(\Gamma)$$
on the generators. Consider the DGA of the collection of Hopf links indexed by the oriented edges $b \in E$ in the (oriented) plumbing graph $\Gamma$ for $Q$, where the link at each edge is the standard Hopf link from Section \ref{subsec:Hopf}. In degree zero there are then the two generators $p_b$ and $q_b$ for the edge $b$,  which are the mixed Reeb chords of the corresponding Hopf link.

We rewrite Formulas \eqref{eq:hopf_dif} for the differential as
\begin{equation}
\label{eq:hopf_dif}
\begin{array}{l}
\partial a_1=e_{s(b)}-t_b+q_bp_b,\\
\partial a_2=e_{t(b)}-s_b+p_bq_b.
\end{array}
\end{equation}
for paths $t_b,s_b \in \pi_1(Q^\circ)$ that depend on choices of capping paths. In particular the equations $e_{s(b)}+q_bp_b=t_b$ and $e_{t(b)}+p_bq_b=s_b$ holds in homology. Now chose capping paths so that
$$ \{s_b\}_{t(b)=v} \cup \{t_b\}_{s(b)=v} \cup \{x_1,\ldots,x_{g_v},y_1,\ldots,y_{g_v}\}$$
is a set of generators for the fundamental group $\pi_1(Q^\circ_v)$ of the component of $Q^{\circ}$ that associated to $v \in V$, and such that the relation
$$ \prod_{t(b)=v} s_b=\left(\prod_{i=1}^{g_v}x_{i,v}y_{i,v}x_{i,v}^{-1}y_{i,v}^{-1}\right)^{-1}\prod_{s(b)=v} t_b$$
holds.

Then a well-defined algebra isomorphism from the degree zero part of the Chekanov--Eliashberg algebra to $\B(\Gamma)$ is obtained by
\begin{itemize}
\item sending the generators $t_b$, $p_b$, $q_b$, $x_{i,v}$, $y_{i,v}$ to the corresponding generators of $\B(\Gamma)$,
\item sending the generators $s_b$ to $e_{t(b)}+p_bq_b\in\B(\Gamma)$ (which are invertible by Remark \ref{rmk:invertible}).
\end{itemize}
\end{proof}

\begin{rmk}
It is clear how to invert the process in the proof of Proposition \ref{prp:ce_and_ppa}. Assume that we are given a union $\Lambda$ of standard Hopf links in a disjoint union of standard contact spheres, together with a cap $Q^\circ$ which can be realized as the smooth locus of an immersed surface $Q$ with only positive transverse double points (in some abstract ambient space). We then obtain a quiver $\Gamma_Q$ with vertices the connected components of the cap, labeled by the corresponding genus and  edges corresponding to the canonical orientation of the mixed Reeb chord `$p$' on the standard Hopf links. We denote the corresponding PPA by $\mathscr{B}(Q) \coloneqq \mathscr{B}(\Gamma_Q)$.
\end{rmk}
}

\section{Lagrangian surgery and Lagrangian fillings of the Hopf link}
\label{sec:lagsurgery}
Recall the procedure of resolving a double point of a Lagrangian immersion in Section~\ref{sec:LagrangeSurgery}. Below, we use the resolutions have zero area parameter, so that monotonicity or strong exactness are preserved under this operation. Up to Hamiltonian isotopy, there are then two ways of resolving a double point. From our perspective, these Lagrangian surgeries are obtained by attaching two Lagrangian fillings of the Hopf link to the punctured Lagrangian $L^{\circ}_\rho$. We compute rigid disks for these Lagrangian fillings in the 4-ball. This then leads to the surgery formula relating refined potentials.

\subsection{Lagrangian surgery}\label{ssec:lagsurg}
In Section \ref{subsec:l_o} we started from an immersed Lagrangian $f\colon Q\to X$, $L=f(Q)\subset X$ and produced a punctured Lagrangian $L^{\circ}\subset X^{\circ}$ with negative ends along Legendrian Hopf links $\Lambda_{\rm st}$ at each double point. In Section \ref{subsec:l_o_rho} we added a Lagrangian cobordism and produced a punctured Lagrangian $L_{\rho}^{\circ}$ with negative end along generic Legendrian Hopf links $\Lambda_{\rm Ho}$. 

To reverse that procedure,
we first obtain $X$ from $X^{\circ}$ by filling the negative ends with 4-balls. In order to recover $L$ from $L^{\circ}_{\rho}$, we fill the Hopf link $\Lambda_{\rm Ho}$ with two intersecting Lagrangian disks. We denote the union of these two disks by $F_{\rm imm}$.  Then $F_{\rm imm}$ is an immersed Lagrangian filling of $\Lambda_{\rm Ho}$ in the 4-ball. 

Note that any other exact Lagrangian filling of $\Lambda_{\rm Ho}$ can be used  in the previous construction. 
There are two basic embedded Lagrangian cylinders filling $\Lambda_{\rm Ho}$. They correspond to the two ways of resolving the self-intersection point of $F_{\rm imm}$. These two cylinders can be described, as in \cite{EHK}, in terms of pinching Reeb chords of $\Lambda_{\rm Ho}$. The difference between the two fillings is the order in which the chords are pinched. We denote these two Lagrangian fillings $F_{p}$ and $F_{q}$, where the subscript indicates which Reeb chord is pinched first. Figure~\ref{fig:slices} shows the slices of the fronts of the above three fillings. Here we use a symplectomorphism from a half-space in $\C^{2}$ to the symplectization of the Darboux ball containing $\Lambda_{\rm Ho}$, see \cite{EHK} for details.

By attaching $F_{p}$ or $F_{q}$ to $L^\circ_\rho$, we create immersions $(L_{p})^\circ_\rho$ and $(L_{q})^\circ_\rho$ with one less double point, which are the two resolutions of $L$ at the chosen double point (punctured at all remaining double points). To make sure that these resolutions can be done with zero area parameter, one argues as in the proof of the fourth point of Lemma~\ref{lem:Gamma}.

\begin{figure}[htp]
	\centering
	\includegraphics{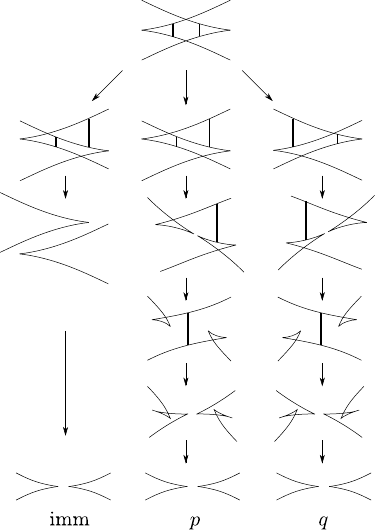}
	\caption{The cobordism from the Hopf link to unlinked knots in slices}
	\label{fig:slices}
\end{figure}

\subsection{Holomorphic curves on the fillings of the Hopf link}
\label{sec:CurvesOnFilling}

We next consider punctured rigid holomorphic disks with boundary on the Lagrangian fillings of $\Lambda_{\rm Ho}$. We use the notation for homotopy classes of paths on the resolved Lagrangians near the double points shown in Figure \ref{fig:ce_glue_basis}.

\begin{figure}[h]
	\includegraphics[]{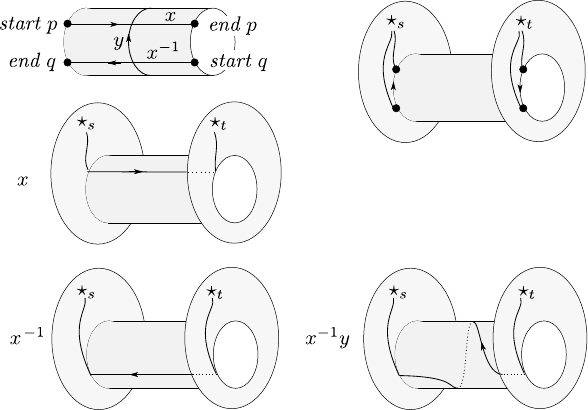}
	\caption{Boundaries of rigid disks in the filling of the Hopf link.}
	\label{fig:ce_glue_basis}
\end{figure}

\begin{lem}\label{l:surgereddisks}
	Consider the three fillings $F_{\rm imm}$, $F_{p}$ and $F_{q}$ of the Hopf link in the 4-ball. The rigid disks in the 4-ball with positive punctures at the Reeb chords $p$ and $q$ and boundary on these fillings are as follows:
	\begin{itemize}
		\item[$(a)$] For $F_{\rm imm}$, there is a unique rigid disk with a positive puncture asymptotic to the Reeb chord $p$ (resp.~$q$) and precisely one corner at the double point. 
		\color{black}
		\item[$(b)$] For $F_{p}$, there is a unique disk with positive puncture at $p$ and its boundary homotopy class equals $x$. There are two disks with positive puncture at $q$ with homotopy classes $x^{-1}$ and $x^{-1}y$, respectively. See Figure~\ref{fig:ce_glue_basis} (top left).
		\item[$(c)$] For $F_{q}$, there is a unique disk with positive puncture at $q$ and its boundary homotopy class equals $x^{-1}$. There are two disks with positive puncture at $p$ with boundary homotopy classes $x$ and $xy$, respectively.
	\end{itemize}	
\end{lem}

\begin{proof}
	For simpler notation, let the symbol $\sharp$ denote one of $\mathrm{imm}$, $p$ or $q$.
	It follows from \cite[Sections 5.5 and 5.6]{EHK} that there is a natural one-to-one correspondence between rigid holomorphic disks with boundary on $F_{\sharp}$ and positive punctures, and rigid flow trees determined by $F_{\sharp}$. A brief description of this result is as follows. As the Lagrangian filling $F_{\sharp}$ is scaled to lie in a small neighborhood of a Lagrangian plane, the holomorphic disks with boundary on $F_{\sharp}$ limit to flow trees and for all sufficiently small scaling parameters, there is a unique rigid disk near each rigid flow tree. 
	
	The flow trees determined by $F_{\sharp}$ are drawn in Figure \ref{fig:trees}. In the terminology of \cite{Eflowtree} they can be described as follows. 
	
	The dashed flow tree in the leftmost picture is a flow line starting at the unstable manifold of the Reeb chord and flowing into the Reeb chord corresponding to the double point. There are two rigid flow trees in this picture. The first is simply the unstable flow line of the second Reeb chord flowing to an end vertex. The second is more involved: it starts along the unstable manifold of the second Reeb chord that breaks at a $Y_{0}$-vertex, the flow line going to the left from this vertex has a switch vertex and then an end, the flow line going to the right goes straight to an end vertex. 
	
	Finally, the flow trees in the right picture are exactly as in the middle picture with the roles of the two Reeb chords interchanged.
\end{proof}

\begin{figure}[htp]
	\centering
	\includegraphics{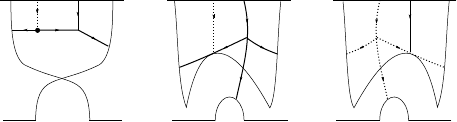}
	\caption{Rigid flow trees for the three fillings of the Hopf link.}
	\label{fig:trees}
\end{figure}

Together with SFT stretching, Lemma \ref{l:surgereddisks} determines rigid holomorphic disks with boundary on $L$, $L_{p}$, and $L_{q}$ from the punctured holomorphic disks on $L^{\circ}$. Introduce a stretching variable $R$ and write $(L_{\sharp})^\circ_\rho$ for the result of attaching $F_{\sharp}$ to $L^\circ_\rho$ at the level $t=-R$ in the negative end of the chosen puncture (the choice of the puncture and $R$ is omitted from the notation). Fix a point $\zeta\in L^{\circ}_\rho$.

\begin{lem}
	\label{lem:sm_prelim}
	For all sufficiently large $R$ there is a 1-1 correspondence between Maslov index~2 holomorphic disks with boundary on $(L_{\sharp})^\circ_\rho$ passing through $\zeta$, and two level buildings of holomorphic disks with first level a Maslov 2 disk on $ L^{\circ}_\rho$ with negative punctures at (degree 0) Reeb chords of $\Lambda_{\rm Ho}$, and second level a rigid disk with one positive puncture and boundary on $F_{\sharp}$ studied in Lemma~\ref{l:surgereddisks}.   
\end{lem}

\begin{proof}
	As $R\to+\infty$, the Maslov index 2 holomorphic disks on $(L_\sharp)^\circ_\rho$ converge to buildings whose top component is an analogous Maslov index~2 disk on $L_\rho^\circ$, and the bottom component is a disk on the filling $F_\sharp$. Conversely, any limiting configuration can be glued uniquely to a Maslov index~2 disk on $(L_\sharp)^\circ_\rho$ for all sufficiently large $R$.
\end{proof}

{
The above results on holomorphic curves near a double point can be used to express how potentials change under surgery and gives a proof of Theorem \ref{thm:toric_pot1}.	
\begin{proof}[Proof of Theorem \ref{thm:toric_pot1}]
Lemma \ref{lem:sm_prelim} gives a 1-1 correspondence between disks on the immersed or surgered Lagrangian and two level disks of the following form. The two level disk has an outside disk with Reeb chord punctures in the negative end near the double point that is independent of the inside Lagrangian and an inside part which is a disk with punctures and boundary on the inside Lagrangian which is either immersed or surgered. The inside disks are described in Lemma \ref{l:surgereddisks}. Replacing the disks on the immersed inside fillings by those on the surgered inside fillings then proves Theorem \ref{thm:toric_pot1}.  	
\end{proof}
}
\subsection{Lagrangian surgery formula}
\label{sec:SmoothingFormula}
Fix a positive double point $\xi$ of the Lagrangian immersion $L$. By resolving $\xi$ as in Section~\ref{ssec:lagsurg}, we obtain monotone Lagrangians
$(L_p)^\circ_\rho$ and $(L_q)^\circ_\rho$ with one puncture less than $L^\circ_\rho$. Write $\Lambda$ for the singularity link of $L$ and $\Lambda'$ for that of the resolved Lagrangians. Lemma~\ref{lem:sm_prelim} explains how the refined potential changes under Lagrangian surgery. {In this section we give a formula expressing the potentials of the resolutions from that of the immersed Lagrangian. We adapt the same sign conventions as in Section \ref{sec:immsph} (in particular, the bounding spin structure is used on the components of the Legendrian Hopf link), which implies that the differential of the Hopf link takes the form described there.

Recall that the algebra $CE(\Lambda;\C[\pi_1(L^\circ_\rho)]$ has a canonical set of idempotents corresponding to the trivial loops at the base-points of the connected components of $L^\circ_\rho$. Choose one such connected component, and let $\st$ be its base-point.
There is a subalgebra $CE(\Lambda;\C[\pi_1(L_\rho^\circ)],\star)$ as in~\eqref{eq:ce_submodule}, corresponding to words  \eqref{eq:ce_word} beginning and ending with an element of $\pi_1(L^\circ_\rho,\st)$.

Let
$$
p,q\quad\text{ and }\quad a_{1},a_{2} \in CE(\Lambda;\C[\pi_1(L^\circ_\rho)])
$$
denote the two Reeb chords connecting distinct components and the same components, respectively, of the Hopf link in the negative end corresponding to $\xi$, see Section~\ref{subsec:Hopf}. We consider first the case of the $p$-smoothing. Let $\st_s$ and $\st_t$ be the base-points on the connected components of the endpoints of the chord $p$. Note that, by construction, $\st_s$ and $\st_t$ coincide if the double point is a self-intersection, and that $\st$ coincides with $\st_s$ or $\st_t$ except in the case when it is situated on a different component.

By definition of the Chekanov--Eliashberg algebra, there are fixed choices of reference paths from $\st_s$ and $\st_t$ to the endpoints of the chords $p$ (so-called capping paths). Consider the paths $x$ and $x^{-1}$, and the loop $y$ in the top row of Figure~\ref{fig:ce_glue_basis}. Concatenation of these paths with the appropriate capping paths gives the following in the smoothed Lagrangian $(L_p)^\circ_\rho$: 
\begin{itemize}
	\item A path from $\st_s$ to $\st_t$. We still denote it by $x$, see Figure~\ref{fig:ce_glue_basis}, middle row.
	\item Two paths from $\st_t$ to $\st_s$. We still denote them by $x^{-1}$ and $x^{-1}y$, respectively, see Figure~\ref{fig:ce_glue_basis}, bottom row.
\end{itemize}
In the case $s=t$, all these paths are loops and we let $\st_{s}=\st_{t}$ be the base point of its connected component of $(L_p)^\circ_\rho$.

In the case $s\ne t$ we choose notation (possibly switching the chords $p$ and $q$) so that $\st\neq\st_t$. The points $\st_s$ and $\st_t$ lie in the same connected component of the surgered Lagrangian $(L_p)^\circ_\rho$, and we choose base points and reference paths as follows:
\begin{itemize}
	\item $\st_s$ is the base-point in its connected component (of $(L_p)^\circ_\rho$), $\st_t$ is removed, and all other base-points remain unchanged.
	\item Any reference path in $L^\circ_\rho$ with an endpoint at $\st_{t}$ is concatenated with $x$ or $x^{-1}$ to give a reference path in $(L_p)^\circ_\rho$ with endpoint at $\st_s$.
\end{itemize}

We next define the surgery map as a unital DGA morphism
\begin{align}
\label{eq:s_ce}
\mathscr{S}^p \colon  CE(\Lambda;\C[\pi_1(L_\rho^\circ)],\st) \to CE(\Lambda';\C[\pi_1((L_p)_\rho^\circ)],\st).
\end{align}
An element of the Chekanov--Eliashberg algebra in the left hand side of \eqref{eq:s_ce} is a sum of monomials as in \eqref{eq:ce_word}.
The map $\SS^p$ modifies each monomial in the following way. First, replace each occurrence of the chords $p$, $q$, $a_{1}$, and $a_{2}$ according to the rule: 
\begin{align}\label{eq:Ssubst1}
& p\mapsto x,\\\notag
&q\mapsto x^{-1}-x^{-1}y,\\\notag
&a_{1}\mapsto 0,\quad a_{2}\mapsto 0,
\end{align}
and concatenate the loops resulting after these substitutions.
Second, replace each occurrence of a generator from $\pi_1(L_\rho^\circ,\star_i)$ in a word with its image in $\pi_1((L_p)_\rho^\circ,\star_i)$ under the map induced by the inclusion $L^\circ \subset (L_p)^\circ$. In the case  $s=t$ this gives monomials in $CE(\Lambda';\C[\pi_1((L_p)_\rho^\circ)],\st)$ and describes the map.

In the case  $s\ne t$, there is a second step where each loop in a monomial resulting from the substitution that is based at $\st_{t}$ is conjugated with $x$ to give a loop based at the base point $\st_{s}$. This again gives monomials in $CE(\Lambda';\C[\pi_1((L_p)_\rho^\circ)],\st)$ and defines the map in this case. In other words, in the case $s\ne t$, we apply the change of base point map $$\pi_1(L^\circ_\rho,\st_t)\hookrightarrow \pi_1((L_p)^\circ_\rho,\st_t)\xrightarrow{l\, \mapsto\, x\cdot l\cdot x^{-1}} \pi_1((L_p)^\circ_\rho,\st_s),$$
whenever it is needed.  In both cases described above, the proof of the chain map property is straightforward.

We then similarly define the map 
\begin{align}
\label{eq:s_ceq}
\mathscr{S}^q \colon  CE(\Lambda;\C[\pi_1(L_\rho^\circ)],\st) \to CE(\Lambda';\C[\pi_1((L_q)_\rho^\circ)],\st)
\end{align}
corresponding to the other surgery in a directly analogous way using instead of \eqref{eq:Ssubst1} the following map on generators 
\begin{align}\label{eq:Ssubst2}
& p\mapsto x-xy,\\\notag
&q\mapsto x^{-1}.\\\notag
&a_{1}\mapsto 0,\quad a_{2}\mapsto 0,
\end{align}
}
%We continue in the case $\sharp=p$, the case $\sharp=q$ is directly analogous using Lemma 
%\ref{l:surgereddisks} (c) instead of (b).

\begin{rmk}
	Suppose $\st_s\neq\st_t$.
	Denote by $y\in \pi_1((L_p)_\rho^\circ,\st_s)$ the loop going around the core circle of the surgery cylinder, as in Figure~\ref{fig:ce_glue_basis} (top), and by $t$ the analogous loop in $L^\circ_\rho$. 
	We find that $\SS^p$ acts as follows:
	\begin{align*}
	& p\mapsto e_{s},\\
	&q\mapsto e_s-y=e_s(1-y),\\
	& t\mapsto y.
	\end{align*}
\end{rmk}

\begin{rmk}
	\label{rmk:loc}
	When $\st=\st_s=\st_t$, it is easy to check that there is a natural  isomorphism
	$$ H_0CE(L_\rho^\circ;\C[\pi_1(L_\rho^\circ)],\st)[p^{-1}] \cong H_0CE((L_p)^\circ_\rho;\C[\pi_1((L_p)_\rho^\circ)],\st)$$
	where $[p^{-1}]$ stands for localization at the element $p$. Then $\SS^p$ becomes the canonical map to the localization. 
	%When some of the base-points are distinct, the surgery map is again essentially the localization, but this is slightly more involved to formalize because $p$ itself is not an element of $CE(L_\rho^\circ;\C[\pi_1(L_\rho^\circ)],\st)$.
\end{rmk}

\begin{prp}
	\label{prp:surg}
	For any choice of a positive double point of $L$, any base-point $\st\neq \st_t$ and any $\sharp=p,q$, the  surgery map \eqref{eq:s_ce} is a well-defined map of algebras that relates the refined potentials of $L$ and $L_p$,
	$$
	W_{L_\sharp}(\st)=\SS^\sharp(W_L(\st)).
	$$
	Here we use the spin structure on $L_\sharp$ that bounds when restricted to central circle in the surgery cylinder. (For the spin structure that does not bound, substitute $y\mapsto -y$.)
\end{prp}

\begin{proof}
	To check that the surgery map is well-defined, it suffices to check that the relations coming from the differential in the Chekanov--Eliashberg algebra: $pq=e_s-t_1$, $qp=e_t- t_2$, are preserved by $\SS^\sharp$. Here $t_1=t$, $t_2$ is the analogous loop at the other boundary component, and both loops $t_1,t_2$ are sent to $y$. The formulas for $\SS^\sharp$ make it clear that these relations are preserved. The surgery map is a map of algebras because it obviously preserves concatenations.

	{The claim about refined potentials is a direct consequence of Lemma~\ref{lem:sm_prelim}: for a sufficiently stretched almost complex structure the disks contributing to the potential on the Lagrangian with double point correspond to two level curves with top level a disk with negative punctures at chords $p$ and $q$ and lower level the disks with positive puncture at $p$ and $q$ and a corner at the double point. Similarly, the disks on surgered Lagrangian are again two level disks with the same top level and with lower level the disks with boundary on the smoothing. It follows from Lemma \ref{l:surgereddisks} that the effect of this relation on the words representing the boundaries of rigid holomorphic disks in the Chekanov--Eliashberg algebras of $\Lambda$ and $\Lambda'$ are as in \eqref{eq:Ssubst1} and \eqref{eq:Ssubst2} for the two smoothings respectively.}
\end{proof}

\section{Whitney spheres in $\CP^2$}
\label{sec:application}
By a \emph{Whitney sphere} we mean an immersed Lagrangian two-sphere having a single transverse self-intersection with positive intersection number, so that its two resolutions give embedded Lagrangian tori and not Klein bottles.
This section contains two applications of our theory to the symplectic topology of monotone Whitney spheres in $\CP^2$. We first show that there are restrictions on the potentials of a monotone Whitney spheres by considering their Legendrian lift to $S^5$. Second, we construct monotone Whitney spheres which are not displaceable from complex lines and conics by Hamiltonian isotopies (unlike the standard Whitney sphere in $\C^{2}\subset \C P^{2}$ that is disjoint from the line at infinity).

{
\subsection{Basic results on monotone spheres}\label{ssec:basicmonspheres}
Let $L\subset \C P^{2}$ be a Lagrangian sphere with one transverse double point that is monotone.	
	\begin{lem}\label{l:basicmonspheres}
		The double point of $L$ is positive, and there exists a punctured $\mu=2$ disk with boundary on $L$. Monotonicity then implies that the disk has area $\frac{1}{3}\int_{\CP^1}\omega_{\OP{FS}}$. In particular, $L$ has an embedded Legendrian lift into $S^5$, where $S^{5}\to\CP^2$ is the pre-quantization space. 
	\end{lem}
	\begin{proof}
		If the double point is negative then its resolution gives an embedded Lagrangian Klein bottle, which contradicts \cite{Shevchishin}. Hence the double point is positive.
		
		To see that there is a $\mu=2$ disk, consider the tori obtained by resolving the double point. By \cite{DamianAudin}, these monotone tori have a non-zero number of $\mu=2$ disks and  \eqref{eq:mu_res} then implies that there also exists a $\mu=2$ disk with boundary on $L$. 
		
		To see that $L$ has a Legendrian lift, note first that any disk without punctures and boundary on $L$ has symplectic area an integer multiple of $\int_{\CP^1}\omega_{\OP{FS}}$. This implies that $L$ lifts to a Legendrian immersion in $S^5$. To see that the Legendrian is embedded we must check that the double point corresponds to a Reeb chord of non-zero length. By the monotonicity assumption and the fact that the first Chern class satisfies $c_{1}(\C P^{2})=3[\C P^{1}]$, the $\mu=2$ disk has area $\frac{1}{3}\int_{\CP^1}\omega_{\OP{FS}}$ and the Reeb chords on the lift that correspond to the double point thus have non-zero lengths:
		$$\frac{1}{3}\int_{\CP^1}\omega_{\OP{FS}} +k\int_{\CP^1}\omega_{\OP{FS}}, \:\:k\in\Z.$$
\end{proof}}

\subsection{Constructing Whitney spheres}
\label{sec:wh_construct}
In \cite{Vi13,Vi14,Vi16} infinitely many different (i.e., pairwise not Hamiltonian isotopic) monotone Lagrangian tori in $\CP^2$ were constructed. We will use this construction to produce infinitely many different Whitney spheres in $\CP^2$. We review the torus-construction, see \cite{PT17}. 

Start with the monotone Clifford torus $T_{\rm Cl}\subset\C P^{2}$. There are three smooth \emph{Lagrangian} disks $D_j$, $j=1,2,3$, such that for all $j$, $T_{\rm Cl}\cap D_{j}=\partial D_{j}$ and such that the boundary homology classes $[\partial D_{j}]$ are, for some identification $H_1(T,\Z)\cong \Z^2$, the orthogonals of:
\begin{equation}
\label{eq:cliff_seed}
v_1=(1,1),\quad 
v_2=(-2,1),\quad\text{and}\quad
v_3=(1,-2),\quad\text{respectively}.
\end{equation}
These Lagrangian disks form a \emph{Lagrangian seed} $(T_{\rm Cl},\{D_1,D_2,D_3\})$ in the sense of \cite{PT17}, see also \cite{STW15}. This means that the interior of $D_j$ is disjoint from $T_{\rm Cl}$ and from $D_{i}$, $i\ne j$, that $D_{j}$ is attached cleanly to $T_{\rm Cl}$ along its boundary (i.e.,~the normal to $\del D_{j}$ in $D_{j}$ is nowhere tangent to $T_{\rm Cl}$), and finally that the curves $\del D_j$ have pairwise minimal intersections in $T_{\rm Cl}$, $|\del D_j\cap\del D_k|=3$, $j\neq k\in\{1,2,3\}$. We refer to \cite{PT17} for a construction of such disks.

Given a Lagrangian seed $(T,\{D_k\})$ and an index $j$ of one of its Lagrangian disks, e.g.~$T=T_{\rm Cl}$ and $j\in\{1,2,3\}$, one defines \cite{STW15,PT17} its \emph{mutation along $D_{j}$} which is a new Lagrangian seed
$$
\left(\mu_{D_j}T,\ \{\mu_{D_j}D_k\}_{k=1}^3\right).
$$
Here $\mu_{D_j}T$ is a new monotone Lagrangian torus. The key point is that the mutated torus bounds
the same number of (mutated) Lagrangian disks $\mu_{D_j}D_k$, $k\in\{1,2,3\}$, as observed in \cite{STW15}.

Because mutation is defined for any Lagrangian seed, the construction can be iterated indefinitely. For each consecutive mutation, one has to choose an index $j$, i.e.~the Lagrangian disk used for mutation. 	
This gives an infinite trivalent graph of mutations of tori in $\CP^2$. More precisely, there exist monotone Lagrangian tori in $\CP^2$ indexed by the vertices of the infinite trivalent graph, with a fixed initial vertex corresponding to $T_{\rm Cl}$. We call them Vianna tori.

Recall that a Markov triple is a triple of positive integers $(a,b,c)$ satisfying
$$
a^2+b^2+c^2=3abc.
$$
A mutation of an ordered Markov triple in the first direction is the transformation
$$
(a,b,c)\mapsto (3bc-a,b,c).
$$
Mutations in the second and the third directions are defined in an analogous manner. Starting with the Markov triple $(1,1,1)$, one can index the vertices of the infinite trivalent graph (and thus Vianna tori obtained from the above construction) by ordered Markov triples, where two Markov triples connected with an edge differ by mutation.

It is well-known that the unordered Markov triples form an infinite tree with edges corresponding to the mutations, and root given by the solution $(1,1,1)$. Except the root and its adjacent vertex, all other vertices are 3-valent.\color{black}

\begin{table}[h]
	$
	{
		\xymatrixcolsep{-1em}
		\xymatrixrowsep{1em}
		\small
		\xymatrix{
			&&&(1,1,1)\ar@{-}[d]&&&\\
			&&&(1,1,2)\ar@{-}[d]&&&\\
			&&&(1,2,5)\ar@{-}[dll]\ar@{-}[drr]&&&\\
			&(2,5,29)\ar@{-}[dl]\ar@{-}[dr]&&&&(1,5,13)\ar@{-}[dl]\ar@{-}[dr]&\\
			(5,29,433)&&(2,29,169)&&(5,13,194)&&(1,13,34)\\
		}
	}
	$
\end{table}

In \cite{Vi13,Vi14} Vianna constructed a monotone Lagrangian torus $T(a,b,c)$ for every Markov triple $(a,b,c)$ and showed that two such tori are not Hamiltonian isotopic when the unordered Markov triples are different. By construction, any two such Lagrangian tori that correspond to the same \emph{unordered} Markov triple differ by an explicit linear symplectic automorphism of $\CP^2$; they are hence Hamiltoninan isotopic.\color{black}

\begin{thm}\cite{Vi14}
	\label{thm:vianna_tori}
	For each (unordered) Markov triple $(a,b,c)$, there is a monotone Lagrangian torus $T(a,b,c)\subset \CP^2$, and all such tori are pairwise Hamiltonian non-isotopic.\qed
\end{thm}

There is a version of the Weinstein neighborhood theorem for the union of a Lagrangian torus and a cleanly attached Lagrangian disk, see e.g.~\cite{PT17}. Such a neighborhood contains a family of Whitney spheres \color{black} obtained by cutting out a neighborhood of the boundary of the disk in the torus and filling the two resulting boundary circles with two intersecting Lagrangian disks close to the original one, see e.g.~\cite{Ha15}. The Hamiltonian isotopy class of the immersed sphere is uniquely determined by the requirement that the two Lagrange surgeries at its double point with \emph{zero area parameter}, see Section \ref{sec:LagrangeSurgery}, are Hamiltonian isotopic to the original torus and its mutation. Observe that such an immersed Lagrangian sphere automatically is monotone since its two zero-area resolutions are by assumption. \color{black}
Thus a Lagrangian seed gives one Whitney sphere for each disk in that seed. In summary we have the following.

\begin{prp}
	\label{prp:whitney_cp2}
	For each unordered pair of unordered Markov triples $(a,b,c)$, $(a,b,3ab-c)$ which differ by a single mutation, there is a monotone Whitney sphere 
	$$L[(a,b,c)(a,b,3ab-c)]\subset \CP^2.$$ 
	Any two Whitney spheres coming from distinct such pairs are not Hamiltonian isotopic. The two resolutions of $L[(a,b,c)(a,b,3ab-c)]$ are $T(a,b,c)$ and $T(a,b,3ab-c)$.
\end{prp}

Note that these Whitney spheres are indexed by the edges of the mutation graph discussed above.

\begin{proof}
	The Whitney spheres are not Hamiltonian isotopic since their resolutions are non-isotopic.
\end{proof}

\subsection{Potentials}
One can compute the LG potentials of the Vianna tori by iteratively applying the wall-crossing formula  starting from the potential $x+1+1/xy$ of the Clifford torus. We refer to \cite{PT17} for a detailed discussion. We point out that  in order to carry out these computations algorithmically, one also needs an explicit formula for how the boundary homology classes of the Lagrangian disks change under mutation. Such a formula is found in \cite{PT17,GU12}. 

Next, knowing the potential of a torus $T$, one computes the potential of the Whitney sphere obtained by collapsing the torus along a Lagrangian disk using Theorem~\ref{thm:toric_pot1}.
Namely, one writes the torus potential $W_T$ using a first homology basis $(x,y)$ for which the smoothing formula reads
$$
\begin{array}{l}
p\mapsto x,\\
q\mapsto (1+y)x^{-1},\\
t\mapsto -y.
\end{array}
$$
(This is a basis in which $(0,1)$ is the boundary of the Lagrangian disk.) Recall from Remark~\ref{rmk:loc} that this map is the map from $\C[p,q,t^\pm]/( 1-t-pq)$ to its localization at $p$, in particular, it is injective. 

Expand the potential $W_T$ in powers of $x$:
$$
W_T(x,y)=\sum_{k\ge 0}x^kF_k(y)+\sum_{k>0}x^{-k}G_k(y)
$$
where $F_k$, $G_k\in\C[y^{\pm 1}]$.
The smoothing formula implies that $G_k(y)$ is divisible by $(1+y)^k$ as a Laurent polynomial. Let $H_k(y)=G_k(y)/(1+y)^k$. Then the potential of the Whitney sphere equals
$$
W_L(p,q,t)=\sum_{k\ge 0}p^kF_k(-t)+\sum_{k>0}q^{k}H_k(-t).
$$

Applying this algorithm, the potentials of Whitney spheres quickly become very complicated. We list a few of them in Table~\ref{tab:pot}. 
%shows the refined potentials of some monotone Whitney spheres in $\CP^2$. \qed

\begin{table}[h]
	\small
	\renewcommand*{\arraystretch}{1.4}
	\begin{tabular}{|c|}
		\arrayrulecolor{Gray}
		\hline
		$L[(1,1,1)(1,1,2)]$:
		\\
		\(\ \, p+{q^2}t^{-1}\)\\
		\hline
		$L[(1,1,2)(1,2,5)]$:
		\\
		\(\ \, {p}t^{-1}+2 q^2+q^5 t\)
		\\
		\hline
		$L[(1,2,5)(2,5,29)]$:
		\\
		$
		p
		+
		2 q^2
		+
		5 q^5
		+
		{4 q^8}t^{-1}
		+
		{10 q^{11}}t^{-1}
		+
		2 q^{14}t^{-2}
		+
		{10 q^{17}}t^{-2}
		+
		5 q^{23}t^{-3}
		+
		q^{29}{t^{-4}}
		$
		\\
		\hline
		$L[(1,2,5)(1,5,13)]$:
		\\
		$
		{p (5+4 t)}t^{-2}
		+
		{2 p^4 (5+t)}t^{-3}
		+
		{10 p^7}t^{-4}
		+
		{5 p^{10}}t^{-5}
		+
		{p^{13}}t^{-6}
		+
		{q^2}t^{-1}
		$	
		\\
		\hline
		$L[(1,5,13)(1,13,34)]$:
		\\
		\renewcommand*{\arraystretch}{1}
		$
		\begin{array}[t]{l}
		{p \left(78+225 t+210 t^2+68 t^3+4 t^4\right)}t^{-4}
		+
		{2 p^4 \left(143+300 t+175 t^2+26 t^3\right)}t^{-3}
		\\
		+
		{p^7
			\left(715+1050 t+350 t^2+18 t^3\right)}t^{-2}
		+
		{p^{10} \left(1287+1260 t+210 t^2+2 t^3\right)}t^{-1}
		\\
		+
		2 p^{13} \left(858+525 t+35 t^2\right)
		+
		2 p^{16} t \left(858+300t+5 t^2\right)
		+
		9 p^{19} t^2 (143+25 t)
		\\
		+
		5 p^{22} t^3 (143+10 t)
		+
		p^{25} t^4 (286+5 t)
		+
		78 p^{28} t^5
		+
		13 p^{31} t^6
		+
		p^{34} t^7
		\\
		+
		{q^2 \left(13+24 t+9 t^2\right)}t^{-5}	
		+
		{q^5}t^{-6}	
		\end{array}
		$
		\renewcommand*{\arraystretch}{1.4}
		\\
		\hline
	\end{tabular}
    \caption{Potentials $W(p,q,-t)$ of some Whitney spheres in $\CP^2$.}
	\label{tab:pot}
\end{table}

\begin{rmk}
	Observe the $-t$ substitution to keep all signs positive. Geometrically, the change of variables $t\mapsto -t$ corresponds to a change of spin structure on $L^{\circ}$. {(Positivity of coefficients is a reflection of the local nature of the construction, disks after the modification are obtained from disks before the modification and local data of a disk.)}
\end{rmk}

\begin{rmk}
	\label{rmk:pq}
	Table \ref{tab:pot} uses the canonical choice of ordering of the sheets explained in Lemma~\ref{lem:ordering} below; this means that the $p$-terms appear in the powers $p^{1+3k}$ and the $q$-terms in the powers $q^{2+3k}$.
\end{rmk}

\begin{figure}[h]
	\includegraphics[]{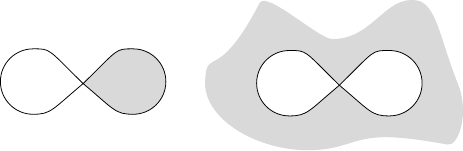}
	\caption{The two holomorphic disks with corners on the standard Whitney sphere in $\CP^2$.}
	\label{fig:eight}
\end{figure}

\begin{ex}
	\label{ex:std_wh}
	The Whitney sphere $L[(1,1,1)(1,1,2)]$ is the usual Whitney sphere in $\C^2$ compactified to $\CP^2$. In particular, it is disjoint from the complex line at infinity. It bounds two families of Maslov~index~2 holomorphic disks with corners: one family with a $p$-corner, and another with two $q$-corners. These disks contribute to the potential with the terms $p$ and $q^2t^{-1}$, respectively. They can be seen explicitly as follows. There is an $S^1$-family of complex lines which intersects $L[(1,1,1)(1,1,2)]$ in a figure eight curve. Fix such a line and take a point on $L$ within that line; then the two desired disks with corners are subsets of that line shown in Figure~\ref{fig:eight}.
\end{ex}

\subsection{Normal form of the refined potential}
Any Whitney sphere in $\CP^2$ has a canonical ordering of the sheets at its double point as follows.
\begin{lem}\label{lem:ordering}
	Let $L\subset \C P^{2}$ be a (not necessarily monotone) Whitney sphere. For the two possible orderings of the sheets at its double point, the Maslov index of a disk $u\colon (D,\partial D)\to (\C P^2,L)$ whose number of $p$-corners resp.~$q$-corners equals $a$ resp.~$b$, is given by 
	$$
	\mu(u) = \pm 2(a-b) + 6d(D),\quad d(D) \in \Z,
	$$
	where the sign $\pm$ depends on this choice. In particular, there is a distinguished ordering of the sheets of $L$ for which a term $p^aq^b$ in the refined potential can be nonzero only when
	\begin{equation}
	\label{eq:pot_cp2_div3}
	a-b \equiv 1\mod 3
	\end{equation}
	is satisfied.\qed
\end{lem}
The above ordering of the sheets of any Whitney sphere will be used throughout the rest of the section.
We next discuss a  normal form for the refined potential $W_{L}(p,q,t)$ of a monotone Whitney sphere $L\subset \CP^2$. Recall that $W_{L}(p,q,t)$ is a sum of monomials of the form $p^{a}q^{b}t^{k}$, $a,b\ge 0$.
If $a>0$ and $b>0$, the relation $pq=1-t$ gives
\[ 
p^{a}q^{b}t^{c} = p^{a-1}q^{b-1}(t^{c}-t^{c+1}).
\] 
This relation can be used repeatedly to get rid of all mixed terms: the ones which involve a non-trivial power of both $p$ and $q$.
Recalling (\ref{eq:pot_cp2_div3}), it follows that the potential can be written in the  form:
\begin{equation}
\label{eq:pot_collected_terms}
W_L(p,q,t)=\sum_{k\ge 0} p^{1+3k}\cdot P_{1+3k}(t^{\pm 1})+\sum_{k\ge 0} q^{2+3k}\cdot Q_{2+3k}(t^{\pm 1}),
\end{equation}
where $P_{1+3k}$, $Q_{2+3k}\in\C[t^{\pm 1}]$ are Laurent polynomials in $t$. We call \eqref{eq:pot_collected_terms} the \emph{normal form} of the refined potential. 

\subsection{The Hopf fibration and Legendrian lifts}
\label{sec:HopfFibration}

In this section we prove Theorem~\ref{th:cp2_p1}. 
Let $L$ be a monotone Whitney sphere in $\C P^{2}$ and consider the Hopf map
\[ 
\pr\colon S^{5}\to \C P^{2},
\] 
where $S^{5}$ is the standard contact $5$-sphere. {(See Section \ref{ssec:basicmonspheres} for background results on monotone Whitney spheres in the projective plane.)} The fibers of the Hopf map are Reeb orbits and $L$ lifts to a Legendrian sphere in $S^{5}$. Choose the contact form so that the action of a fiber equals $1$, corresponding to $\int_{\CP^{1}}\omega=1$. The lift $\Lambda$ of $L$ has the following Reeb chords:
\begin{itemize}
	\item transverse Reeb chords $a_k$, $k\ge 0$ of action $\mathfrak{a}(a_{k})=\frac13+k$ and grading $|a_{k}|=2+6k$;
	\item transverse Reeb chords $b_k$, $k\ge 0$ of action $\mathfrak{a}(b_{k})=\frac23 + k$ and grading $|b_{k}|= 4+6k$;
	\item $S^{2}$-Bott families $C_{k}$, $k\ge 1$ of Reeb chords of action $\mathfrak{a}(C_{k})=k$ and grading $|C_{k}|=-1+6k$. 
\end{itemize}
The first two computations are immediate since $L$ is monotone and the grading of the Reeb chords $|a_{k}|$ and $|b_{k}|$ equals the Maslov index of a holomorphic disk inside $(\CP^2,L)$ of area $\frac13+k$ and $\frac23+k$, respectively, {\color{red} see Lemma \ref{l:basicmonspheres}}. The degree in the last item follows from the observation that the linearized Reeb flow corresponds to full rotations in the two complex lines in the tangent plane to $\C P^{2}$ and the fact that $c_{1}(\C P^{2})=3[\C P^{1}]$. (The grading can also be identified with the dimension of the space of real rational curve of degree $2k$ inside $(\CP^3,\mathbb{R}P^3)$ satisfying a $(2k-1)$-fold tangency to the divisor at infinity $\mathbb{R}P^2 \subset \CP^2$ at a \emph{fixed} point, i.e.
$n-3+\mu-2k=4(2k)-2-(2k-1)=-1+6k.$)

Note that if we make the Reeb flow generic by a small perturbation, each of those  Bott families gives rise to two non-degenerate Reeb chords as follows:
\begin{itemize}
	\item $C_{k}$ gives $c_{k}^{m}$ of index $|c_{k}^{m}|=-1+6k$ and $c_{k}^{M}$ of index $|c_{k}^{M}|=1+6k$.
\end{itemize} 

\begin{proof}[Proof of Theorem \ref{th:cp2_p1}]
	Let $\Lambda\subset S^{5}$ be the Legendrian lift of $L$, which is embedded by Lemma \ref{l:basicmonspheres} \color{black}. Consider $CE(\Lambda)$, see Proposition \ref{p:dginv}, where we cap $\Lambda$ by a disk (this corresponds to trivial coefficients). It follows from the index calculations above that $CE(\Lambda)$ is supported in degrees $\ge 2$. In particular, $CE(\Lambda)$ admits augmentations. Fix a point $\xi$ on $\Lambda$ and consider the dg-algebra $CE(\Lambda,\xi)$ twisted by a point condition. (I.e.~we add a new generator $\xi$ of degree $1$ to the dg-algebra and count disks in the symplectization passing through $\R\times \xi$ as disks with punctures at the new generator.) It follows from Proposition \ref{p:local} that we can compute $CE(\Lambda,\xi)$ in a Darboux ball. The duality result \cite[Theorem 5.5]{EESa} then implies that the linearized contact homology of $\Lambda$ in degrees $2+k$ and $2-k$ are isomorphic for $k\ne 0$ and that the homology in degree $2$ has a fundamental class that evaluate with degree $1$ to $\Lambda$ and a non-trivial pairing of other classes. Given the degrees of the chords of $\Lambda$ we find that that the only possibility here is that the linearized homology consists only of the fundamental class. Consider now the filtration by word length in $CE(\Lambda,\xi)$, since all generators have positive degrees the corresponding spectral sequence converges. The homology of the first page in this spectral sequence, corresponding to the linearized homology complex with the generator $\xi$ added, consists of only the base ring $\Z$ of $CE(\Lambda,\xi)$. It follows that the homology of $CE(\Lambda,\xi)$ equals $\Z$.
	
	Consider the calculation of $CE(\Lambda,\xi)$ with $\Lambda$ in the original position.
	Here the differential of the degree $2$ chord counts disks in $\R \times S^5$ with boundary on $\R \times \Lambda$ that pass through $\R\times\xi$. Using the holomorphic projection of the Hopf fibration $\R \times S^5 \to \CP^2$ we can identify that count with the algebraic count of disks in $\C P^{2}$ that satisfy the following (see \cite{Di16.2}): 
	\begin{itemize}
		\item the disk have boundary on the projection of $\Lambda$ to $\CP^2$;
		\item the boundary passes through $\xi$;
		\item the disk has a single $p$-corner;
		\item its  area equals $1/3$.
	\end{itemize}
	These are precisely the disks contributing to the term $P_1(1)$ from the normal form~\eqref{eq:pot_collected_terms}. Since $\xi$ is a boundary the corresponding count of pseudoholomorphic discs now translates to $P_{1}(1)=1$.
\end{proof}

\begin{rmk}
	The proof in fact shows that the Chekanov--Eliashberg algebra $CE^{\ast}(\Lambda)$ is quasi-isomorphic to that of the standard unknot $CE^{\ast}(\Lambda_{\rm st})$. To see this, note that by the computation of linearized homology, the $E_{2}$-page of the word length spectral sequence consists exactly of the polynomial algebra $\C[c_{0}]$ and there can be no further differentials.
	
	In fact, the disks with the point constraint and one corner remember the whole linearized differential  determined by:
	\[ 
	c_{k}^{m}\mapsto b_{k-1},\quad a_{k}\mapsto c_{k}^{M},\quad k\ge 1.
	\]
\end{rmk}
\begin{rmk}	
	Theorem \ref{th:cp2_p1} can alternatively be proved using a version of Floer cohomology for monotone Lagrangian immersions with the trivial local system $p=q=0$, $t=1$ (cf.~Theorem~\ref{mainth:general_brane} in the strongly exact case). Recall that the eigenvalues of the quantum multiplication by $c_1(\CP^2) \in H^2(\CP^2)$ are the non-zero numbers $3e^{ik2\pi/3}$. Generalizing \cite[Lemma 2.7]{She13} to this case shows that Floer cohomology is trivial for any immersed sphere with the  local system $p=q=0$, $t=1$. 
	
	Indeed, the number of disks in the potential that is counted with this local system is obviously equal to zero, while it must be equal to one for the above eigenvalues, whenever Floer cohomology is nontrivial. Finally, the analysis in \cite{EESa} shows that, first, the fundamental class on the sphere is a cycle, which is exact by the established acyclicity. Second, that the Floer strips responsible for the differential hitting it are in bijection with the holomorphic disks with boundary on the sphere that contribute to the coefficient $P_1(1)$.
\end{rmk}  

\subsection{Whitney spheres  and divisors}
It is shown in \cite{AGM01}
that for every Lagrangian submanifold $L\subset X$ in a monotone or rational symplectic manifold, there exists a symplectic divisor (i.e.~a symplectic real codimension~2 submanifold) $\Sigma\subset X$ whose homology class is dual to $k\omega$ for some $k>0$, and which is disjoint from $L$. It is called a \emph{Donaldson divisor}. This theorem generalizes to immersions. In general $k$ may be large and for a given $L$, it is an interesting problem to minimize the degree  of a Donaldson divisor $\Sigma$ disjoint from it. For Lagrangian tori there is the following sharp result. 

\begin{thm}[{\cite[Theorem C]{DGI16}}]\label{t:torioffline}
	Let $H\subset \CP^2$ be a complex line.
	For every Lagrangian torus $T\subset \CP^2$, there exists a symplectomorphism $\phi\co \CP^2\to\CP^2$ such that $\phi(T)\cap H=\varnothing$.\qed 
\end{thm}

In this section we show that the analogue of Theorem \ref{t:torioffline} fails for Whitney spheres.
We begin with an elementary result.

\begin{lem}\label{l:sphereeasydisplace} 
	The following displaceability statements hold.
	\begin{itemize}
		\item 
		Any Whitney sphere from Proposition~\ref{prp:whitney_cp2} can be displaced from any smooth complex cubic curve by a symplectomorphism.
		\item 	Whitney spheres of the form $L[(1,b,c)(1,b,3b-c)]$ are displaceable from a cubic curve with a single node  by symplectomorphisms.
		\item
		The standard Whitney sphere $L[(1,1,1),(1,1,2)]$ is displaceable from the complex line by symplectomorphisms.
	\end{itemize}
\end{lem}

\begin{proof}[Proof]
Recall that any two smooth algebraic curves of the same degrees are Hamiltonian isotopic. The reason is that any two such curves obviously are isotopic through symplectic embeddings, which then can be turned into a Hamiltonian isotopy by a standard result, see \cite[Proposition 0.2]{ST05}.\color{black}

	Consider the Clifford torus $T_{\rm Cl}$ with the three Lagrangian disks $D_j$, $j=1,2,3$ as in Section~\ref{sec:wh_construct}. We claim that $T_{\rm Cl}\cup D_1\cup D_2\cup D_3$ is disjoint from a smooth cubic curve $\Sigma$, and $T_{\rm Cl}\cup D_1\cup D_2$ is disjoint from a cubic curve $\Sigma'$ with one node. This can be seen from the almost toric fibration used in \cite{Vi13} as follows. The curves $\Sigma$ (resp. $\Sigma'$) arise as the boundary divisor of the moment triangle with 3~(resp.~2) smoothed corners. By construction, mutations take place in a small neighborhood of the Lagrangian skeleton $T_{\rm Cl}\cup D_1\cup D_2\cup D_3$ (resp. $T_{\rm Cl}\cup D_1\cup D_2$), and hence $\Sigma$ (resp. $\Sigma'$) is also disjoint from the mutated tori and corresponding Whitney spheres. The last statement is obvious. 
\end{proof}

In general it is interesting to ask the following question.

\begin{que}
	Given a monotone Whitney sphere $L\subset \CP^2$, what is the lowest degree of a smooth symplectic divisor $V\subset \CP^2$ such that $V\cap L=\varnothing$?
\end{que}

Lemma \ref{l:sphereeasydisplace} implies that for any Whitney sphere from Proposition~\ref{prp:whitney_cp2} the answer is 1, 2 or~3. We show that there are infinitely many Whitney spheres in distinct Hamiltonian isotopy classes for which 2 is a lower bound.

First, let us review the terminology. A \emph{symplectic line} (resp.~\emph{symplectic conic)} in $\CP^2$ is a smooth two-dimensional symplectic submanifold representing the homology class $[\CP^1]$ (resp. $2[\CP^2]$).
We say that a set $A\subset \CP^2$ is \emph{displaceable} from $B\subset \CP^2$ if there is a symplectomorphism $\phi\co \CP^2\to\CP^2$ such that $\phi(A)\cap B=\varnothing$.
Because all symplectomorphisms of $\CP^2$ are Hamiltonian \cite{Gr85}, this is the same as to say that $A$ is Hamiltonian displaceable from $B$.
Symplectomorphisms of $\CP^2$ acts transitively on the set of symplectic lines and conics by \cite{ST05}, so when discussing displaceability of a set from a line or a conic, one may assume the latter to be a fixed complex line or conic.

\begin{rmk}
	\label{rmk:higher_degree_displace}
	Non-displaceability from the conic naturally implies non-displaceability from the line, since any neighborhood of a line contains a smooth conic obtained by smoothing the double point of the union of two nearby copies of the line. More generally,  non-displaceability from a higher-degree divisor implies non-displaceability from a lower-degree divisor. 
\end{rmk}

\subsection{Linking properties}\label{sec:linking}
We will prove the non-displaceability results by studying how the disks contributing to the refined potential can link a given  line or conic disjoint from a Whitney sphere.

Consider a symplectic divisor $\Sigma\subset X$ and a monotone Lagrangian embedding or immersion $L\subset X$ which does not intersect $\Sigma$. Then  $L$ is also monotone in $X\setminus \Sigma$.
We will use the following result on the minimal Maslov number of Lagrangian tori in uniruled (possibly open) symplectic manifolds.
\begin{prp}[\cite{CM14}]
	\label{prp:minmaslov}
	Let $X$ be one of $\CP^2$, $\R^4$, $T^*S^2$, or $T^*\mathbb{R}P^2$, with its standard symplectic structure.
	If $T\subset X$ is a Lagrangian torus then there exists a topological disk $D\subset X$ with $\partial D\subset T$, of Maslov index~2 and of positive symplectic area.
\end{prp}
\begin{proof}
	Since $\CP^2$ is uniruled by spheres, and $\R^4,$ and $T^*S^2$  are uniruled by planes, the result follows in these cases from \cite{CM14}. Consider $X=T^*\mathbb{R}P^2$. Let $p \colon T^*S^2 \to T^*\mathbb{R}P^2\color{black}$ be the two-fold symplectic  cover and let $\widetilde{L} \subset T^*S^2$ be a Lagrangian torus for which $p \colon \widetilde{L} \to L$ is c cover of degree one or two. If $D \subset (T^*S^2,\widetilde{L})$ is a Maslov index~2 disk, then  $p(D) \subset (T^*\mathbb{R}P^2,L)$ is a Maslov index~2 disk with boundary on $L$, so the case $X=T^{\ast}\R P^{2}$ follows from the case $X= T^{\ast}S^{2}$. 
\end{proof}

We next study linking. Note that if $L$ is a Whitney sphere then $H_{1}(L,\Z)\cong \Z$. If $D$ is a topological disk with boundary on $L$ having $k$ $p$-switches and $(k-1)$ $q$-switches, then $\partial D$ represents a generator of $H_{1}(L,\Z)$. If $D'$ is another disk with $a$ $p$-switches and $b$ $q$-switches then 
$$[\partial D']=(a-b)[\partial D]\in H_{1}(L,\Z).$$  

\begin{prp}\label{prp:link_0_cp2}
	Consider a monotone Whitney sphere $L\subset \CP^2$ and write its refined potential in normal form \eqref{eq:pot_collected_terms}. Let $H$ and $Q$ be a symplectic line and conic, respectively. 
	\begin{itemize}
		\item[$(i)$] If $L\cap H=\varnothing$ then the algebraic intersection number between $H$ and  any disk contributing  to the $p^{1+3k}$-terms or $q^{2+3k}$-terms equals $-k$ or $k+1$, respectively, where $k \ge 0$.
		\item[$(ii)$] If $L\cap Q=\varnothing$, then there are the following two possibilities.
		\begin{itemize}
			\item[$(a)$] If $\Z=H_1(L) \to H_1(\CP^2 \setminus Q) = \Z_2$ is the zero map, the intersection number between $Q$ and a disk contributing to  the $p^{1+3k}$-terms or $q^{2+3k}$-terms equals $-2k$ or $2(k+1)$, respectively, where $k \ge 0$.
			\item[$(b)$] If $\Z=H_1(L) \to H_1(\CP^2 \setminus Q) = \Z_2$ is surjective, the intersection number between $Q$ and a disk contributing to  the $p^{1+3k}$-terms or $q^{2+3k}$-terms equals $1+k$ or $-k$, respectively, where $k \ge 0$.
		\end{itemize}
	\end{itemize}
\end{prp}

Before proving Proposition \ref{prp:link_0_cp2} we state a corollary. Let $\Sigma\subset \C P^{2}$ be a divisor and $L\subset \C P^{2}$ a monotone Whitney sphere in the complement of $\Sigma$. Positivity of intersections implies that all terms in the potential $W_{L}$ that correspond to disks with negative intersection number with $\Sigma$ vanish, and  those corresponding to disks with zero intersection number give the refined potential in the complement of the divisor. Together with Proposition \ref{prp:link_0_cp2}, this implies to the following.

\begin{cor}\label{cor:disp_higher_term}
	Let $L\subset\C P^{2}$ be a monotone Whitney sphere.	
	If the normal form \eqref{eq:pot_collected_terms}  of the potential $W_{L}$ has $P_{1+3k}\neq 0$ for some $k\ge 1$, then $L$ is non-displaceable from the symplectic line. If also $Q_{2+3k}\neq 0$ for some $k\ge 1$, then $L$ is non-displaceable from the conic.\qed
\end{cor}

\begin{proof}[Proof of Proposition~\ref{prp:link_0_cp2}]
	Let $\Sigma=H$ or $\Sigma=Q$ and let $Y=\CP^2 \setminus \Sigma$. In the first case $Y\cong \R^4$ and in the second $Y\cong T^*\mathbb{R}P^2$. In either case, $Y$ is a rational homology ball and there exists a disk $(D,\del D)\subset (\CP^2,L)$ which is \emph{disjoint} from $\Sigma$, and which is a generator of
	$$
	\Z \cong H_2(Y,L) \xrightarrow{\lambda\cdot \id_\Z} H_1(L)\cong \Z, \quad \lambda =1 \text{ or } 2.
	$$
	Moreover, if $\Sigma=H$, we can choose $D$ to have a single $p$-corner, and if $\Sigma=Q$ we can choose $D$ to have either
	\begin{itemize}
		\item a single $p$-corner if $H_1(L) \to H_1(T^*\mathbb{R}P^2)$ is trivial (i.e.~$\lambda=1$), or
		\item two $q$-corners if $H_1(L) \to H_1(T^*\mathbb{R}P^2)\cong \Z_2$ is non-trivial (i.e.~$\lambda=2$).
	\end{itemize}
	Fix a reference disk $D$ with these properties.

	Recall the convention for the numbering the sheets at the double point of $L$ in Lemma~\ref{lem:ordering}. Using this convention we have  $\int_D \omega=1/3+l$ for the disk with a single $p$-corner, and $\int_D \omega=-2/3+l$ for the disk with two $q$-corners, for some $l \in \Z$. Consider a Lagrangian torus obtained by resolving the double point of $L$ and apply Proposition \ref{prp:minmaslov}. Monotonicity implies $\int_{D}\omega =1/3$ in both cases (i.e.~$l=0$ in the first case and $l=1$ in the second).
	
	This allows us to determine the intersection number of a disk with a divisor from the term of the potential that the disk is contributing to.
	Consider case $(i)$. Any element in $H_2(\CP^2,L)$ can be written as 
	\begin{equation}
	\label{eq:a_decomp}
	A=s[D]+l[H],\quad s,l \in \Z,
	\end{equation}
	where $[D]\cdot [H]=0$ and $[H]\cdot[H]=1$. Then
	$$
	\omega(A)=\frac 1 3 s+l ,\quad A\cdot [H]=l.
	$$
	If $A$ is a Maslov index~2 disk, then $\omega(A)=\tfrac13$ and
	$$A\cdot [H]=\frac {1-s}3.$$
	In the case under consideration, the reference disk $D$ has one $p$-corner. Therefore, a holomorphic disk in class $A$ contributes to the $p^{1+3k}$-term when $s=1+3k$, $k\ge 0$, and $A\cdot [H]=-k$. Similarly, a disk in class $A$ contributes to the $q^{2+3k}$-term when $s=-2-3k$, $k\ge 0$, and $A\cdot [H]=1+k$.
	
	Consider case $(ii)$, $(a)$. Arguing as in case $(i)$, we decompose any class $A\in H_{2}(\C P^{2},L)$ using \eqref{eq:a_decomp} and compute 
	$$
	A\cdot [Q]=\frac {2(1-s)}3.
	$$ 
	As in case $(i)$, this implies the claim.

	Consider case $(ii)$, $(b)$. Now the reference disk $D$ has two $q$-corners and does not generate $H_2(\CP^2,L)$. (It generates an index 2 subgroup.) 
	Recall that $\C P^{2}\setminus Q\cong T^{\ast}\R P^{2}$ with its standard symplectic form $\omega_{0}=d\theta$, where $\theta$ is the Liouville (or action) 1-form.
	
	Since $L \subset T^*\mathbb{R}P^2$ is exact and $\int_{D} \omega_{0}=1/3$, if $\gamma$ is any loop in $L$ with a single $p$-corner, 
	then $\int_{\gamma}\theta=-1/6$.
	Note that $\gamma$ belongs to the non-trivial homotopy class in $T^*\mathbb{R}P^2$. Therefore there exists a cylinder $C\subset T^{\ast}\R P^{2}$ (with a single $p$-corner) $\partial C=\gamma-\gamma'$ where $\gamma'=\R P^{1}$ is a real projective line in the 
	zero-section $\R P^{2}$. The curve $\gamma'$ bounds an obvious disk $D''$ in $\CP^2$ (half of the complex line $\C P^{1}$ in which it lies) such that $D''\cdot Q=1$. If $D'=C\cup D''$ then $D'\subset (\CP^2,L)$ is a disk with a single $p$-corner and since $\theta$ vanishes along the $0$ section, we have
	$$ 
	\int_{D'}\omega=\int_{C}\omega_{0}+\int_{D''}\omega=\int_{\gamma-\gamma'}\theta+\int_{D''}\omega=-1/6+1/2=1/3,
	$$
	and
	$$
	\quad [D']\cdot [Q]=1.
	$$
	An arbitrary class $A\in H_2(\CP^2,L)$ can now be written as 
	$$A=s[D]+l[H]+m[D'],\quad k,l,m \in \Z,$$
	and one computes
	$$
	\omega(A)=\frac {s+m} 3+l,\quad A\cdot [Q]=2l+m.
	$$ 
	If $A$ is a Maslov index 2 disk then $\omega(A)=1/3$ and $A\cdot[Q]=\tfrac 1 3 (2-2s+m)$. 
	Since $D$ has two $q$-corners and $D'$ has one $p$-corner, it follows that $-2s+m$ is the number of $p$-corners of $A$ minus the number of $q$-corners of $A$. This gives the result.
\end{proof}

\begin{ex} We summarize what  Proposition~\ref{prp:link_0_cp2} implies for the Whitney spheres in Table~\ref{tab:pot}. 
	\begin{enumerate}
		\item The standard  Whitney sphere $L[(1,1,1)(1,1,2)]$ is contained in the complement $\CP^2 \setminus (H \cup Q)$ of both the standard line at infinity and standard conic $Q=\{z_1z_2=1\}$. This Whitney sphere is homotopically nontrivial in the complement of this conic, that is, it realizes $(ii)$, $(b)$ in Proposition~\ref{prp:link_0_cp2}.
		\item The spheres $L[(1,1,2)(1,2,5)]$ and $L[(1,2,5)(2,5,29)]$ cannot be placed in the complement of a conic in a homotopically nontrivial way. Indeed, their potentials involve higher powers of $q$ which is not consistent with $(ii)$, $(b)$ in Proposition~\ref{prp:link_0_cp2}. On the other hand $L[(1,1,2)(1,2,5)]$ can be placed in the complement of a symplectic line, as shown in Figure \ref{fig:complement_line}.
		\item The sphere $L[(1, 2, 5)(1, 5, 13)]$ is not displaceable from a line, and cannot be placed inside the complement of a conic in a homotopically trivial way since its potential is neither consistent with $(i)$ nor with $(ii)$, $(a)$ in Proposition~\ref{prp:link_0_cp2}.
		\item The sphere $L[(1, 5, 13)(1, 13, 34)]$ is not displaceable from a conic by Proposition~\ref{prp:link_0_cp2} since its potential involves higher powers of both $p$ and $q$.
		\item We finish with an example where Proposition~\ref{prp:link_0_cp2} does not give full information. We do not know whether $L[(1,2,5)(1,5,13)]$ is displaceable from the conic (in a homotopically non-trivial way).
	\end{enumerate}
\end{ex}

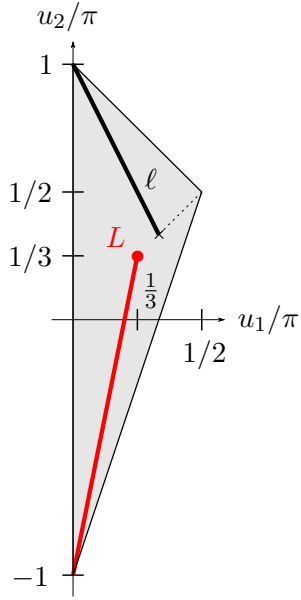
\begin{figure}[h]
	\vspace{5mm}
	\labellist
	\pinlabel $u_1/\pi$ at 81 104
	\pinlabel $u_2/\pi$ at 6 218
	\pinlabel $1/3$ at -8 127
	\pinlabel $-1$ at -8 7
	\pinlabel $1/2$ at -8 151
	\pinlabel $1$ at -2 200
	\pinlabel $\color{red}L$ at 24 135
	\pinlabel $\ell$ at 37 157
	\pinlabel $\frac{1}{3}$ at 37 115
	\pinlabel $1/2$ at 57 91
	\endlabellist
	\includegraphics[]{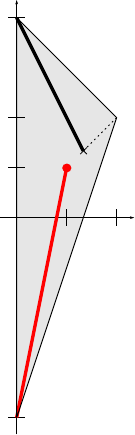}
	\caption{The `moment polytope' for an almost toric fibration on $\CP^2$ where the central fiber over $(1/3,1/3)$ is the Chekanov torus, see \cite{Vi13}. The sphere $L \coloneqq L[(1,1,2)(1,2,5)]$ lies over an arbitrarily small neighborhood of the line connecting $(1/3,1/3)$ with $(0,-1)$. A symplectic line $\ell$ in its complement lies over the straight line between the corner $(0,1)$ and the node (the image of the singular fiber). The line intersects each fiber along a vanishing cycle of the nodal singularity (a curve in the homotopy class $[\theta \mapsto (\theta,-\theta)]$ in the depicted coordinates).}
	\label{fig:complement_line}
\end{figure}

\subsection{General Whitney spheres}
Consider a Whitney sphere $$L=L[(a,b,c)(a,b,3ab-c)]\subset \CP^2$$ and write its refined potential $W_L$ in normal form \eqref{eq:pot_collected_terms}.
Let $T$ be the resolved torus whose potential $W_T$, written in some basis, is obtained from $W_L$ by 
\begin{equation}
\label{eq:markov_sm}
p\mapsto x,\quad q\mapsto(1+y)x^{-1},\quad 
t\mapsto -y,
\end{equation}
see Theorem~\ref{thm:toric_pot1}. Interchanging $c$ and $3ab-c$ if necessary, we can assume that $T=T(a,b,c)$. 

Recall that the Newton polytope of a Laurent polynomial $\sum_{(a,b)\in I} c_{a,b}x^ay^b$, where $I\subset \Z^2$ and $c_{a,b}\neq 0$, is  the convex hull of $I$ in $\R^2$.
Let $\P$ be the Newton polytope of $W_T$, written in the basis $(x,y)$ used above.
It was shown in \cite{Vi13} that $\P$ is a triangle with sides of affine lengths $a^2$, $b^2$ and $c^2$ respectively. 
On the other hand,
the way $W_T$ is obtained from $W_L$ \eqref{eq:markov_sm} leads to the following.
\begin{itemize}
	\item Let $q^rQ_r(t^{\pm 1})$ be the highest non-zero term of $W_L$ in the powers of $q$. Then $\P$ has an edge within the vertical line $\{x=-r\}$.	
	
	\item Let $p^rP_r(t^{\pm 1})$ be the highest non-zero term of $W_L$ in the powers of $p$. If $P_r=ct^s$ is a monomial, then $(r,s)\in \Z^2$ is a vertex of $\P$. If not, $\P$ would have  an edge within the vertical line $\{x=r\}$, in which case $\P$ would not be a triangle.
\end{itemize}

\begin{lem}\label{l:vertex}
	Consider a Whitney sphere $L$  as above.		
	The term $P_r$ by the highest power of $p$ is a monomial in $t$, and corresponds via (\ref{eq:markov_sm}) to a vertex of $\P$.\qed
\end{lem}

\begin{prp}
	\label{prp:wh_higher}
	Consider a Whitney sphere $L=L[(a,b,c)(a,b,3ab-c)]$  as above.
	The potential $W_L$ involves no higher powers of $p$, i.e.~$P_{1+3k}\equiv 0$ for all $k\ge 1$, only if $c=1$ or $3ab-c=1$.
\end{prp}

\begin{proof}
	Suppose that $W_L$ involves no powers of $p$ other than the term $pP_1(t^{\pm 1})$. Lemma~\ref{l:vertex} implies that $P_1=ct^s$ is a monomial. Then $(1,s)$ is a vertex of the triangle $\P$, and the opposite edge of $\P$ is vertical. We may assume that $s=0$ by applying the $SL(2,\Z)$-transformation given by $(1,s)\mapsto (1,0)$, $(0,1)\mapsto (0,1)$.
	
	As shown in \cite{Vi14}, $\P$ is the affine dual of the corresponding  moment polytope for the weighted projective space $\CP^2(a^2,b^2,c^2)$. After an appropriate $SL(2,\Z)$-transformation, the edges of the moment polytope 
	are given (up to translation) by the vectors $a^2\uu_1$, $b^2\uu_2$, $c^2\uu_3$ where 
	$$
	\uu_1=(b^2,-m_1),\ \uu_2=-(a^2,m_2),\ \uu_3=(0,1)
	$$
	and $m_1,m_2$ are some integers which are co-prime with $a$ resp.~$b$. See \cite[Section~2 and Figure~4]{Vi14}. In fact whenever $\bb u_3=(0,1)$, the vectors $\bb u_1$ and $\bb u_2$ must be of the above form. 
	
	Since $\P$ is the affine dual of the above triangle, the vertices of $\P$ are given by the primitive orthogonal vectors $\bb u_1^\perp, \bb u_2^\perp, \bb u_3^\perp$, possibly after one acts on $\P$ by an $SL(2,\Z)$ transformation. But because we have already arranged one vertex of $\P$ to be $(1,0)=\uu_3^\perp$, the remaining two vertices of $\P$ are equal to $\uu_1^\perp$ and $\uu_2^\perp$ where $\uu_1$ and $\uu_2$ are of the above form.
	
	Recall that $m_1$, $m_2$ satisfy the following equations, see \cite[Proof of Proposition~2.2]{Vi14}:
	$$
	\begin{array}{l}
	a^2m_1+b^2m_2=c^2,\\
	a^2(m_1+1)+b^2(m_2+1)=3abc.
	\end{array}
	$$
	
	Since the edge of $\P$ opposite to $(0,1)$ is vertical, $m_1=m_2$. Write $m=m_1=m_{2}$.
	The equations above then give:
	$$
	3ab=c\, \tfrac{m+1}{m}.
	$$
	If follows that $c$ is divisible by $m$, because $m+1$ and $m$ are co-prime
	Next recall that the elements $a,b,c$ of a Markov triple are pairwise co-prime by a theorem of Frobenius \cite{Frobenius}. It  follows that $c=m$. Therefore
	$$
	3ab=c+1.
	$$
	This means $3ab-c=1$. 
\end{proof}

\begin{thm}
	\label{th:wh_line}
	Consider the Whitney sphere $$L[(a,b,c)(a,b,3ab-c)]\subset \CP^2$$
	where $c\neq 1$ and $3ab-c\neq 1$. This Whitney sphere is Hamiltonian non-displaceable from the symplectic line.
\end{thm}

\begin{proof}
	By Proposition~\ref{prp:wh_higher}, the potential of this Whitney sphere in normal form involves higher powers of $p$. The theorem now follows from Corollary~\ref{cor:disp_higher_term}.
\end{proof}

\section{Fukaya category for strongly exact Lagrangian immersions}
\label{sec:fukaya}

{

In this section we sketch an extension the Fukaya category $\mathscr{F}(X)$ of exact Lagrangian submanifolds of a Liouville domain $X$ that includes also certain Lagrangian immersions equipped with what we call a refined, or generalized, local system. The method we use apply more generally and leads to further extension beyond Lagrangian immersions to Lagrangian embeddings of arbitrary skeleta of Weinstein manifolds. Here we will however consider only the simplest case, the ambient space $X$ is four dimensional and the singular Lagrangians are strongly exact Lagrangian immersions of surfaces with transverse double points. In this case, from the view point of Weinstein skeleta, the smooth locus of the Lagrangian immersion can be seen as a cap for a collection of Legendrian Hopf links (these are the links of the singularities of the Lagrangian, which in this case consist of a collection of double points). The refined local system is given by a classical local system that has been extended to a representation of the Chekanov--Eliashberg algebra of the Legendrian Hopf links with coefficients in the fundamental group of the cap. Under the isomorphism from Proposition \ref{prp:ce_and_ppa} the refined local system becomes identified with a representation of the higher genus multiplicative pre-projective algebra that is associated to a plumbing graph that describes the immersion, see Subsection \ref{sec:preprojective}.

\begin{rmk} 
	In the spirit of the more general extensions mentioned above one should use coefficients in chains on the based loop space of the cap. (In the case of surfaces this is often the same thing as a classical local system.) By Legendrian surgery results similar to \cite{EL} the resulting local system is a model for the endomorphism algebra of a collection of generating `co-core' disks in the wrapped Fukaya category of a neighborhood of the singular Lagrangian.
\end{rmk}

%\begin{rmk}	
%A general version of the Fukaya category that can handle immersed Lagrangians was constructed by Akaho--Joyce in \cite{AJ10}. One can compare this approach to the approach we take here in a neighborhood of the given Lagrangian. The standard theory of immersed Lagrangians give the endomorphisms of the immersed skeleton in the wrapped Fukaya category. It is generally possible to recover this information from the wrapped category: take derived Hom in the endomorphism algebra of the generators. In the present case one uses the surgery isomorphism to convert this to an operation on the Chekanov-Eliashberg algebras at the double points.  Going the other direction, from standard immersed Floer theory to the local systems considered here is sometimes possible but not always. The natural candidate for the endomorphism algebra of the wrapped category is the cobar construction on the linear dual of the endomorphism algebra of the immersion. This always gives the completion of the endomorphism algebra and gives the actual algebra exactly when Koszul duality holds, see \cite{EL} for similar examples.
%\end{rmk}

Floer theory for immersed Lagrangians has appeared many times in the literature. One approach to turning $L$ into an object of $\F(X)$ is by disallowing holomorphic curves to have corners on $L$, see e.g.~\cite{Ab11,She11,Al13,AlBa14}. In our language, this corresponds to equipping $L$ with the trivial rank-one local system $\rho$ that takes  all Reeb orbit generators of the Chekanov--Eliashberg algebra to zero. The theory developed here allows for deformations by refined local systems that involve also the double points.

A general version of the Fukaya category that includes immersed Lagrangians was constructed by Akaho--Joyce in \cite{AJ10}. In this theory it is possible to use bounding cochains that involve also the double points to deform the complex, provided Novikov coefficients are used to handle convergence issues, see also Remark \ref{ex:floer_immersed}. The realization of the corresponding deformations in the theory considered here have coefficients in $\C$.

In Section \ref{sec:seidel} we show that these singular Lagrangians, treated as described above, really give rise to \emph{new} objects in the Fukaya category of closed exact Lagrangians which, for certain local systems, have no closed embedded representatives.

\subsection{The objects of the refined Fukaya category}
In the following we let $(X,\omega=d\lambda)$ be a four-dimensional Liouville domain with vanishing first Chern class. Our goal is to extend the version of the Fukaya category of closed exact Lagrangians constructed by Seidel in \cite{Sei16} to include also strongly exact immersed surfaces $f \colon Q \to
X$ with transverse double points. Recall that strongly exact means that there exists a primitive of $f^*\lambda$ which assumes the same value at both sheets at each double point. (This strong assumption allows us to avoid considering weakly unobstructed Lagrangians.) As mentioned above, the general case, as well as higher dimensional cases, should also be possible to deal with in a similar manner, but we leave this to future work.

\begin{ex}
	The prime examples of $X$ to which our theory applies are plumbings and self-plumbings of cotangent bundles of surfaces. It is natural to interpret the zero section in such a plumbing as a strongly exact Lagrangian immersion with transverse double points. Note that this also is the standard model of a neighborhood of any Lagrangian with transverse double points.
\end{ex}

Let $f \colon Q \to X$, $L \coloneqq f(Q) \subset X$, be a compact strongly exact (see Section~\ref{sec:corners}) oriented, spin immersed Lagrangian surface with transverse positive double points. (The positivity assumption can be removed, but in this case we have no good understanding of the refined local systems.) As in Section \ref{sec:preprojective} we associate a PPA $\B(\Gamma)$ to a quiver that arises as the plumbing graph corresponding to $L$. By Proposition \ref{prp:ce_and_ppa} this algebra is isomorphic to the degree zero part of the Chekanov--Eliashberg algebra
$$CE_{0}(\Lambda_\Gamma;\C[\pi_1(Q_\Gamma^\oo)])$$
of $\Lambda_\Gamma$ a collection of Hopf links (the link of the singularities of $L$) with coefficients in the fundamental group of the cap $Q^\oo$ (i.e.~the smooth locus of $L$).

By a \emph{path trivialization} $\mathbf{b}$ on $Q$ we mean a choice of capping paths for the construction of the Chekanov--Eliashberg algebra $CE(\Lambda_\Gamma;\C[\pi_1(Q_\Gamma^\oo)])$.

Let $\boldsymbol{d}=(d_v)_{v\in V}$ be a collection of positive integers indexed by the vertices of $\Gamma$, i.e.~by the connected components $\pi_0(Q)$. We call it the \emph{dimension vector}. Let 
$$
\mathrm{Rep}(\B(\Gamma),\boldsymbol{d})
$$
be the space of complex representations of $\B(\Gamma)$ with the given dimension vector. (There are naturally defined quotients which are multiplicative quiver varieties, but we work only at the level of representation spaces.)

\begin{dfn}\label{d:branedef}
	A \emph{rank $\boldsymbol{d}$ local system} on $L$, or also called a \emph{brane supported on $L$}, is the choice of
\begin{itemize}
\item a path trivialization $\mathbf{b}$ on $Q$,
\item a spin structure $s$ on $Q$,
\item a choice of Maslov potential that agrees at the two different sheets at each double point, and
\item a representation
	$
	\rho\in \mathrm{Rep}(\B(\Gamma),\bar d).
	$ 
\end{itemize}
We denote a brane by $\bL=(L,\rho,\mathbf{b},s)$, or sometimes simply $\bL=(L,\rho)$. 
\end{dfn}

The objects in the refined version of the Fukaya category will be taken to consist of such branes $\bL$ with additional choices of perturbation data to be made in the subsequent subsection.

\subsection{The analytical setup of the refined Fukaya category}
\label{sec:FukayaAnalytic}

We now describe how to extend the definitions of the $A_\infty$-operations in the exact Fukaya category so that the branes $\bL$ described above become objects. We will use the SFT analysis from Section \ref{sec:monotone} for handling holomorphic disks with punctures at the double points. Therefore, we begin by deforming the Lagrangian immersions near its double points so that they satisfy the following:
\begin{itemize}
\item[$(1)$] Each singular point of $L$ has a Darboux neighborhood $(D^4,\omega_0)$ in which the punctured Lagrangian $L^{\circ}$ agrees with a cylinder $\R\times\Lambda_{\mathrm{Ho}}$ over Hopf links $\Lambda_{\OP{Ho}} \subset (S^3,\alpha_{\mathrm{st}})$ in the contact boundary which is the 3-sphere with the standard contact structure given by complex tangencies. (Recall that the Hopf link $\Lambda_{\OP{Ho}}$ lies in a small contact Darboux neighborhood where it has front projection as in Figure \ref{fig:hopffront}.)
\end{itemize}
\begin{rmk}
If completed by a point at the origin in $D^{4}$, the cylinders $\R\times\Lambda_{\mathrm{Ho}}$	are cones that are not smooth Lagrangians. Using instead of $\Lambda_{\mathrm{Ho}}$, the Hopf link $\Lambda_{\OP{st}} \subset S^3$ given as the intersection of the real and imaginary parts of $\C^2$, the cones would be smooth. The analysis of holomorphic curves from Section \ref{sec:monotone} would 
then require a Morse-Bott treatment. We choose to use $\Lambda_{\mathrm{Ho}}$ instead since it is simpler overall. From the general point of view on Floer theory for singular Lagrangians discussed above, the Floer cohomology really applies to a Weinstein neighborhood of the singular Lagrangian where the filling of the Hopf link does not matter.
\end{rmk}

It will be convenient when defining Floer cohomology below to use an almost complex structure which is both cylindrical with respect to the conical structure near the cone points of $L$ and which also extends to a smooth almost complex structure at the cone point. Therefore we work with a cylindrical almost complex structure for the \emph{round} contact spheres near the conical singularities. The resulting singularity link is then, unlike the setup in Section \ref{sec:monotone}, not generic with respect to the Reeb flow, since the round contact sphere is foliated by periodic Reeb orbits. However, as the next lemma shows, relevant SFT-limits work out also for the round contact form because of degree properties of the Reeb chords of $\Lambda_{\OP{Ho}}$. We subdivide Reeb chords of $\Lambda_{\mathrm{Ho}}\subset S^{3}$ in two sets: short and long. The short chords lie inside the Darboux ball, there are four short chords $a_{1}$ and $a_{2}$ of index $1$, $p$ and $q$ of index $0$, see Figure \ref{fig:hopffront}. The long chords do not lie inside the Darboux ball.

\begin{lem}
\label{lem:cone}
Consider the small Hopf link $\Lambda_{\OP{Ho}} \subset (S^3,\alpha_{\rm st})$ in a small Darboux neighborhood in the round contact sphere (see Figure \ref{fig:hopffront}). Its exact Lagrangian cone in the symplectization is $\R \times \Lambda_{\OP{Ho}}\subset \R \times S^3 \cong \C^2 \setminus \{0\}$. The standard integrable complex structure $J_{0}$ on $\C^2$ has the following properties.
\begin{itemize}
\item $J_{0}$ is cylindrical with respect to the standard Liouville form $\frac{1}{2}\sum(x_idy_i-y_idx_i)$ in coordinates $(x_{1}+iy_{1},x_{2}+iy_{2})$ on $\C^{2}$.
\item $J_{0}$ is regular for finite energy pseudoholomorphic disks in $\R \times S^3$ with boundary on $\R \times \Lambda_{\OP{Ho}}$, one positive and possibly several negative punctures at small Reeb chords in $\{a_1,a_2,p,q\}$.
\item All long Reeb chords of $\Lambda_{\OP{Ho}}$ and all periodic Reeb orbits in $S^3$ have Conley--Zehnder indices $\ge 2$.
\end{itemize}
\end{lem}

\begin{proof}
	This is standard, see Section \ref{sec:HopfFibration} for a similar index computations.
\end{proof}

\begin{rmk}
It follows from the last property in Lemma \ref{lem:cone} that the long chords and periodic orbits can be excluded from the breaking analysis in SFT-compactifications of pseudoholomorphic disks with boundary on $L$ of dimension $\le 1$.
\end{rmk}

The $A_\infty$-structure of the Fukaya category is defined by a count of rigid solutions of certain perturbed $\overline{\partial}$-equations for disks with boundary on the Lagrangians with (a possibly empty set of) conical singularities. We represent the domains as strips with slits $\Delta_{k+1}$ and decorate each boundary component by a possibly immersed Lagrangian submanifold $L_{j}$, see \cite{EO17}.
Setting up coherent choices of such perturbation schemes is a crucial part of the construction of the category. We follow the setup of the closed exact Fukaya category as constructed by Seidel in \cite{SeiBook08}. We use the following extra requirements. 
\begin{enumerate}
\item[(2)] For each domain $\Delta_{k+1}$ and boundary component $\partial_{j}\Delta_{k+1}$ decorated by $L_{j}$ we pick non-positive $1$-form $B$ with values in Hamiltonian vector fields, see \cite{EO17}, and domain-dependent compatible almost complex structure $J_z$ on the manifold $X$. We require that $B$ vanishes in a neighborhood of the boundary and that $J_{z}$ is the standard complex structure $J_{0}$ in the round Darboux balls from Property (1) that contain the conical singularities of the Lagrangian $L_j$, near $\partial_{j}\Delta_{k+1}$.
\item[(3)] Near the punctures of $\Delta_{k+1}$, in coordinates $s+it\in[0,\infty)\times[0,1]$, $B=X_{H_{t}}\otimes dt$, see \cite{EO17}, for a time dependent Hamiltonian $H_t \colon X \to \R$ (which vanishes in a neighborhood of the boundary). We require that the time dependent Hamiltonians near punctures in $\Delta_{k+1}$ have no time-one Hamiltonian chords that start or end at a singular point of $L_i$. 
\end{enumerate}

\begin{rmk}
Note that properties $(1) - (3)$ are possible to arrange along $1$-parameter family of Lagrangians: for small deformations the double points move inside the Darboux balls and time one Hamiltonian chords from a point hits a Lagrangian in codimension $2$.
\end{rmk}

The $A_\infty$-operations of $\mathscr{F}(X)$ will be defined by counts of rigid Floer holomorphic disks $u\colon \Delta_{k+1}\to X$ that satisfy the equation $(du-B)^{0,1}=0$ and has boundary on the conical Lagrangians associated to the Lagrangian immersions as described in Property (1). The disks that we count are allowed to have additional boundary punctures asymptotic to the small mixed Reeb chords $p$ and $q$ of degree zero on the Hopf links $\Lambda_{\OP{Ho}}$ (i.e.~the links of the singularities of the Lagrangians).

\begin{rmk}
The analytic setup is non-standard in that our disks are holomorphic for an almost complex structure defined on a symplectic manifold without concave ends, while the boundary condition of the disks is a Lagrangian with concave conical singularities as in Property (1). These properties are compatible in the standard complex structure on the small Darboux balls containing the double points is cylindrical in the complement of the cone points by Lemma \ref{lem:cone}. Together with Property (2) this enables us to use the standard SFT-analysis for the pseudoholomorphic disks near the boundary (where the only nonstandard feature of the disks appears).
\end{rmk}

\subsection{A refinement of the Fukaya category}

Using the above analytic setup we now sketch the construction of our refinement of the Fukaya category, i.e.,~a Fukaya category $\F(X)$ that contains objects that correspond to branes whose underlying Lagrangians are strongly exact Lagrangian immersions with local systems that involve Chekanov-Eliashberg algebras of their singularity links.

\begin{thm}	
	\label{th:general_brane}
	The category $\F(X)$ contains the usual exact Fukaya category as a full subcategory.
	Any brane $\bL$, see Definition \ref{d:branedef} supported on an immersed Lagrangian $L$ gives rise to a well-defined non-zero object of the category $\F(X)$.
	Up to quasi-isomorphism, any such object is obtained from a fixed $(\mathbf{b},s)$ for some   representation $\rho\in \mathrm{Rep}(\B(Q),\bar d)$, and isomorphic representations give quasi-isomorphic objects in $\F(X)$.
\end{thm}	

\begin{proof}[Sketch of proof]
The $A_\infty$-operations are defined by counts of Floer holomorphic disks with perturbations as outlined in Section \ref{sec:FukayaAnalytic}. The disks contributing to the operations are then rigid and have the following boundary conditions and punctures.
\begin{itemize}
\item The boundary lies on the conical deformations of the Lagrangian immersions with cones over small Hopf links $\Lambda_{\OP{Ho}} \subset S^3$ near the cone point (see Property (1)).
\item At one output and several input boundary punctures the disk is asymptotic to Hamiltonian chords according to the perturbation $B=X_{H_{t}}\otimes dt$ near puncture in $\Delta_{k+1}$.
\item There are possibly several negative boundary punctures where the disk is asymptotic to the small mixed Reeb chords $p$ or $q$ of the Hopf links at the conical singularities.
\end{itemize}

The inputs and outputs of the operations are the usual Hamiltonian chords $\phi^1_{H_t}(L) \cap L'$ between the Lagrangians in Floer cohomology. The count of disks that define these operations are weighted by the following additional data: the boundary punctures at the small mixed Reeb chords on the Hopf links together with the boundary components of the disks between the punctures naturally give rise to elements in the algebras
$$CE_{0}(\Lambda_{\Gamma_i};\C[\pi_1((Q_i)_{\Gamma_i}^\oo)]),\:\: i=0,1,$$
which we then reduce to finite-dimensional vectors by using the representation $\rho_i$. (See also Figure \ref{fig:hf_dif_corn}.) Here the isomorphism in Proposition \ref{prp:ce_and_ppa} is used to identify the Chekanov--Eliashberg algebra with coefficients in the cap being the singular locus of the Lagrangian and the corresponding PPA. In the case when the representation has dimension greater than one, we follow the approach taken in \cite{Kon17} where higher-dimensional local systems were considered in the ordinary Fukaya category.

The $A_\infty$-equations for the weighted counts just described follow from a bubbling analysis. The arguments are directly analogous to those used to prove Proposition \ref{th:potdga} (i.e., the proof that the refined potential is independent of choices) with one small difference. We work with the round contact form near cone points instead of a small generic perturbation. However, Lemma~\ref{lem:cone} shows the same bubbling analysis works.
\end{proof}

}

\begin{ex}
	Start with a surface $Q$ with marked points $\{p_a,q_a\}_{a\in\{1,\ldots, m\}}$. Consider the self-plumbing of $T^*Q$ along the pairs of points $p_a,q_a$ using orientation-reversing linear isomorphisms $T^*_{p_a}\to T^*_{q_a}$. The result is an exact symplectic manifold $X$ which canonically contains an immersed Lagrangian $L$, the image of the zero~section under plumbing. {These plumbings were studied earlier by Etg\"u and Lekili \cite{EL17b,EL17} who predicted results along the lines of Theorem \ref{th:general_brane}.}
\end{ex}

\begin{rmk}
	We consider an example, the Floer complex $CF^*(T,\bL)$ of $L$ and an embedded Lagrangian $T$, see Figure~\ref{fig:hf_dif_corn}. 
	\begin{figure}[h]
		\includegraphics[]{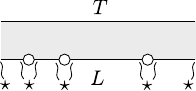}
		\caption{Floer differential where $L$ is immersed. White dots represent SFT-punctures, stars represent base-points}
		\label{fig:hf_dif_corn}
	\end{figure}
	The Floer differential counts holomorphic strips between $T$ and $L^\circ_\rho$ with additional SFT punctures asymptotic to degree~0 Reeb chords on the negative end of $L^\circ_\rho$. Using the path trivialization $\mathbf{b}$ and the reference paths to the endpoints of the Reeb chords, we cut the $L^\circ_\rho$-part of the boundary of a strip to produce a word of loops on $L^\circ_\rho$ and Reeb chords at the double points, i.e., an element of the capped Chekanov--Eliashberg algebra which is identified with $\B(\Gamma)$. Applying the given representation $\rho$ to this element gives a map $V_x\to V_y$.
\end{rmk}

\begin{rmk}
	\label{ex:floer_immersed}
{
For the standard neighborhood of the self-plumbing of a sphere with a single positive transverse double point, the capped Chekanov--Eliashberg algebra was computed in Section \ref{t:Hopf=0highdeg} to be quasi-isomorphic to
$$A=\C[p,q,t^{\pm 1}]/(1-t-pq)$$
where $$\mathrm{Spec}\, A=\{p,q\in \C: 1-pq \neq 0\}.$$
The latter space thus parameterizes branes supported on $L$. Alternative versions of Floer cohomology have been used to study $L$, see e.g.~the recent work \cite{Va17} by Varolgunes. The formality result \cite[Proposition 1]{Va17} allows for the computation of the bounding cochains in the model of the Fukaya category developed in \cite{Ab11} when coefficients are taken in a suitable formal power series ring; see \cite[Lemma 6]{Va17}. In the latter theory we can recover the bounding cochains that correspond to a formal neighborhood of the point $\{p=q=0\}$ inside the space $\mathrm{Spec}\,A$ of deformations in our theory.}
\end{rmk}

\subsection{Unobstructed and weakly unobstructed cases}

Since the holomorphic curve analysis  underlying Theorem~\ref{th:general_brane} applies generally, the result can be modified to apply to different flavors of the Fukaya category. For example, one can drop the assumption that $X$ and $L$ are exact and assume instead that $L$ is tautologically unobstructed, i.e.~there is an admissible $J$ such that $L$ bounds no $J$-holomorphic disks with corners.

The weakly unobstructed case is also important and should similarly follow without too much additional effort. The simplest case should be when $f \colon Q \to L\subset X$ is a monotone immersion with a well-defined potential
$$W_L \coloneqq \sum_{\st \in \pi_0(Q)} W_L(\st) \in \mathscr{B}(\Gamma_Q)$$
and $\rho$ is a representation with dimension vector $\boldsymbol{d}=(1,\ldots,1)$ of length $l=|\pi_0(Q)|$.
\begin{rmk}
There is always a representation with the dimension vector $(1,\ldots,1)$ since it is possible to send all generators $p,q$ to zero and all generators $t_i$ to appropriate diagonal matrices.
\end{rmk}
We call such branes $\bL$ weakly unobstructed, the curvature is given by the value $\rho(W_L)$ on the refined potential. Note that $\rho(W_L)$ is a diagonal matrix in $\C^{l^2}$ by construction and thus, in particular, lies in the center of $\Aut(\C^l)$. We next show that $\rho(W_L)$ is given by a constant multiple of the identity matrix when $L$ is connected.

\begin{lem}
	\label{lem:onedimrepr}
	Assume that $L=f(Q)$ is a monotone connected immersion (but $Q$ is not necessarily connected) inside a simply connected monotone symplectic manifold $X$,  and let $\rho$ be a representation with dimension vector $\boldsymbol{d}=(1,\ldots,1)$. Then the 
	eigenvalue of $\rho(W_L(\st_i))\in \C$ does not depend on the base-point $\st_i \in Q$, i.e., it is independent of the connected component of $Q$.
\end{lem}
\begin{proof}
{
Let $\OP{pr}_v \colon \C^l \to \C$ be the projection to the factor corresponding to the vertex $v \in \Gamma$. Thus $\OP{pr}_v \circ \rho$ is a representation of the sub-algebra $e_v \mathscr{B}(\Gamma_Q)e_v \subset \mathscr{B}(\Gamma_Q)$. We show that the value $\OP{pr}_v(W_L) \in \C$ does not depend on $v \in \pi_0(Q)$.

Assume that $\rho(p_a)\neq 0$ or $\rho(q_a)\neq 0$ for some edge $e$. Proposition~\ref{prp:surg} implies that the representations $\OP{pr}_v \circ \rho$ is a pull-back of a representation of the algebra after surgery at double point corresponding to $e$, and that the value $\OP{pr}_v \circ \rho(W_L)$ coincides with the evaluation of the potential after surgery.

We call a representation $\rho$ \emph{connected} if any two vertices of $\Gamma_Q$ (corresponding to connected components of $Q$) can be connected by a sequence of edges (orientations not necessarily matching), such that $\rho$ is nonzero on each edge in the sequence. 

The surgery argument above, together with the fact that the potential is independent of the choice of basepoint on each connected component of $Q$, proves the lemma for connected representations. Indeed, in this case $\OP{pr}_v \circ \rho(W_L)$ for different $v \in \pi_0(Q)$ correspond to the value of a potential for different choices of basepoints on a Lagrangian immersion of a connected surface.

For arbitrary representations $\rho$ we use a continuity argument and the observation that the subspace of connected representations is dense in the space of all representations. To see the dense property, note that if there exists a decomposition of the vertices of $\Gamma$ into two disjoint subsets $\Gamma_{Q}'\cup\Gamma_{Q}''$ such that $\rho$ vanishes on all arrows $p_a$, $q_a$, that connect a vertex in $\Gamma_{Q}'$ to a vertex in $\Gamma_{Q}''$ then we can continuously deform the value $\rho(p_a)$ of the representation to a non-zero morphism while keeping $\rho$ constant on all other generators. Repeating this argument we find a connected representation arbitrarily close to the original representation.}
\end{proof}

\subsection{Surgery formula for Floer theory}
\label{sec: surgery floer}

In this section, we give a sketch of the surgery formula from Proposition~\ref{prp:surg} in the context of Floer theory and the Fukaya category. The proof is based on the same stretching argument as in Section~\ref{sec:CurvesOnFilling}.

\begin{thm}
	In the setting of Theorem~\ref{th:general_brane}, let $a$ be a positive double point of a strongly exact Lagrangian immersion $L \subset (X,d\lambda)$ of a connected compact manifold $Q$. Consider a representation $\rho$ of $\B(Q)$ with any dimension vector. The following quasi-isomorphism holds in the Fukaya category $\F(X)$:
	$$
	(L,\rho\circ \SS^p_a,\mathbf{b},s)\simeq ( L_{p_a},\rho,\mathbf{b},s^*).
	$$
	Here
	$ L_{p_a}$ is the $p$-surgery of $L$ at the double point $a$, and $\SS^p_a\co \B(Q)\to\B(Q_{p,a})$ is the surgery map. The spin structure $s^*$ on $Q_{p,a}$ \emph{extends} over the small surgery disk, and is otherwise pulled back from $s$. The path trivializations on $Q$ and $Q_{p,a}$ are taken to be the same.
	The analogue holds if we replace $p_a$ with $q_a$ everywhere.
\end{thm}

\begin{rmk}
	\label{rmk:rep_pull}
	Any representation of $\B(Q)$ of form $\rho\circ \SS^p_a$ has the property that the image of
	$p_a\in \B(Q)$ (defined as in Subsection~\ref{sec:preprojective}) is invertible in $\mathrm{End}(\C^{d_{s(a)}})$. In particular not all representations  $\B(Q)$ can be obtained this way. In other words, the surgered Lagrangians support less objects in $\F(X)$ than the more singular one.
\end{rmk}

\begin{proof}
{
The goal is to construct a quasi-isomorphism
$$e_{p_a} \in CF((L,\rho\circ \SS^p_a,\mathbf{b},s),( L_{p_a},\rho,\mathbf{b},s^*))$$
for suitable data. We will pass to the category of twisted complexes, and translate the problem into finding a bounding cochain on the direct sum $L[1] \oplus L_{p_a}$ of objects that represents the cone $C(e_{p_a})$, where this cone moreover becomes acyclic.

We start by considering the acyclic twisted complex which represents the cone $C(e)$ of the identity morphism
$$ e \colon (L,\rho\circ \SS^p_a,\mathbf{b},s) \to (L,\rho\circ \SS^p_a,\mathbf{b},s).$$
For convenience we can perform a Hamiltonian perturbation $L'$ of the target Lagrangian $L$ to yield an acyclic cone $C(e')$ given as a bounding cochain on $L[1] \oplus L'$.

Recall that the Floer complex of $L[1] \oplus L'$ can be tautologically identified with the Floer complex of the immersed union $L \cup L'$ (although this is not a strongly exact immersion). Since $L'$ is strongly exact we can perform a Lagrange surgery, and the analysis from Section~\ref{sec:CurvesOnFilling}. This readily gives an acyclic cone $C(e_{p_a})$ given as a bounding cochain on the complex $L[1] \oplus (L')_{p_a}$, i.e.~where the surgery has been applied to one of the two involved Lagrangians in the direct sum. Since $L_{p_a}$ is Hamiltonian isotopic to $(L')_{p_a}$ the result follows.}
\end{proof}

\section{Floer-essential double point}
\label{sec:seidel}
Consider the exact symplectic manifold
$$
X=\C^2\setminus\{xy=1\}
$$
where $x,y\in \C^2$. It has the standard Liouville structure induced by the Fubini-Study symplectic form on $\CP^2$ and contains the standard Lagrangian sphere with one double point, 
$$L\subset X,$$
which is the Lagrangian skeleton of $X$; see e.g.~\cite{SeiBook13,Riz17}. It is known that the manifold $X$ is self-mirror over $\C$. Considered as an object of the compact exact $\C$-linear Fukaya category $\F(X)$, the immersed sphere $L$ corresponds to some point $z_0$ in the mirror. It was observed by Seidel \cite{SeiBook13} that all other points of the mirror are realized by the exact Clifford or Chekanov torus in $X$ equipped with a $\C^*$-local system, whereas $z_0$ is not realized by a local system on either of those two tori. This led to the conjecture that $z_0$ cannot be realized by any exact Lagrangian torus in $M$ with a local system. 

This conjecture follows from the recent classification result \cite{Riz17}. In this  section we give a different proof of Seidel's conjecture.

\begin{thm}\label{th:seidel_q}
	Let $\bT=(T,\rho)$ be an exact Lagrangian torus $T\subset M$ equipped with a $\C^*$-local system $\rho$. Then $HF^*(L,\bT)=0$.
\end{thm}

The Floer cohomology in question is computed in the classical sense: the Floer differential counts holomorphic strips between $T$ and $L$ without corners on $L$. (In the language of the previous section, this corresponds to equipping $L$ with the `trivial' local system $p\mapsto 0$, $q\mapsto 0$, $t\mapsto 1$.) The torus $T$ is not assumed to be of vanishing Maslov class, so the cohomology $HF^*(L,\bT)$ is generally $\Z/2$-graded.

\subsection{Closed-open maps}
Let $X$ be a Liouville domain and $\bT\subset X$ be an exact Lagrangian submanifold which is oriented, spin, and equipped with a $\C^*$-local system $\rho$. The closed-open map  is a linear map
$$
\CO_\bT\co SH^*(X)\to HF^*(\bT,\bT)\cong H^*(T).
$$
Recall that $\CO$ counts appropriately defined pseudo-holomorphic maps from the punctured unit disk
\begin{equation}
\label{eq:CO_domain}
u\co (D\setminus\{0\},\del D)\to (X,T)
\end{equation}
where the puncture $0$ is asymptotic to the input orbit, and $1\in\del D$ is the output marked point. The counts are weighted using the local system $\rho$.
Restricting to degree~0, we get
\begin{equation}
\label{eq:CO_deg0}
\CO_\bT\co SH^0(X)\to HF^0(\bT,\bT)\cong \C\cdot 1_{\bT}.
\end{equation}
Here $1_\bT$ is the unit.
This map counts rigid pseudo-holomorphic maps as in \eqref{eq:CO_domain} with the additional point constraint, 
$u(1)=\zeta$, where $\zeta\in T$ is a fixed point.

It is convenient to define the universal closed-open map:
\begin{equation}
\label{eq:CO_Zpi1}
\CO^{\rm string}_T\co SH^0(X)\to\C[\pi_1(T)],
\end{equation}
counting the same curves as in \eqref{eq:CO_deg0}, but remembering the boundary homotopy class 
$[u(\del D)]\in \C[\pi_1(T)]$. Note that this version of the closed-open map is a map of algebras.

The closed-open map for a fixed representation $\rho$ is then naturally expressed as
$$
\CO_\bT=\rho(\CO^{\rm string}_T)\cdot 1_\bT.
$$
Suppose that $T$ is an $n$-torus, then for every $u\in SH^0(X)$ one can write $\CO^{\rm string}_T(u)$ as a Laurent polynomial
$$
\CO^{\rm string}_T(u)\in\C[x_1^{\pm 1},\ldots x_n^{\pm 1}]
$$
which we call the \emph{symplectic potential} of $u$.
By viewing a local system $\rho$ on $T$ as a point $\rho\in (\C^*)^n$, represented by its monodromies in a fixed basis for $H_1(T;\Z)$, the previous relation translates into
$$
\CO^{\rm string}_T(u)(\rho)=\CO_\bT(u)\cdot 1_\bT
$$
where $\CO^{\rm string}_T(u)(\rho)$ should be interpreted as evaluating the Laurent polynomial at the point $\rho$.

\subsection{The symplectic potential of the Whitney sphere}
Let $X$ and $L$ be as in Theorem~\ref{th:seidel_q}. In direct analogy to the refined disk potential, define the refined closed-open map 
$$
\CO^{\rm string}_{L}\co SH^0(X)\to \C[p,q,t^{\pm 1}]/(pq=1-t).
$$
It counts curves as in \eqref{eq:CO_domain} with any number of $p$- and $q$-corners on $L$, weighted by their boundary classes $[\del u]$ in the capped Chekanov--Eliashberg algebra.
Considering $L$ as an object of $\F(X)$ `classically' amounts to setting $p=q=0$, $t=1$ in the symplectic potential:
$$
\CO_L=\CO^{\rm string}_L|_{p=q=0,\, t=1}\cdot 1_L,\quad \CO_L\co SH^0(X)\to HF^*(L,L)\cong \C\cdot 1_L. 
$$
The isomorphism
$$SH^0(X)\cong \C[u,v,s^{\pm 1}]/(uv=1+s)$$
was established in \cite{Pa13}. In the following proposition we extend this result by showing that the refined closed-open map to the skeleton $L$ of $X$ is an isomorphism.
\begin{prp}
	\label{prp:co_wh}
	Using the algebra isomorphism
	$$SH^0(X)\cong \C[u,v,s^{\pm 1}]/(uv=1+s),$$
	the map $\CO_{ L}^{string}$ is given by
	$$
	u\mapsto p,\quad v\mapsto q,\quad s\mapsto -t.
	$$ 
	In particular, for the 	`classical' closed-open map one has 
	$$
	\CO_L(u)=\CO_L(v)=0.
	$$
\end{prp}

\begin{proof}
	The surgery formula  for the potential in Proposition~\ref{prp:surg} can be adapted to the present case. Example~\ref{ex:co} below explains the computation of the closed-open map to the Clifford torus, and the surgery formula then recovers the closed-open map to the Whitney sphere. 
\end{proof}

\begin{rmk}
	Alternatively, one can compute $SH^0(X)$ together with $\CO^{\rm string}_L$ using an extension of the Legendrian surgery formula of \cite{BEE}, compare with the discussion in Section~\ref{sec:dga}.
\end{rmk}

\begin{ex}
	\label{ex:co}
	If $T$ is the exact Clifford torus then
	$$
	\begin{array}{l}
	\CO^{string}_T(u)=(x+1)y,\\ \CO^{string}_T(v)=y^{-1}\in \C[x^{\pm 1},y^{\pm 1}].
	\end{array}	
	$$
	This closed-open map can be computed by considering partial compactifications of $X$ to $\C^2$ and $\CP^2\setminus\{xy=1\}$, and interpreting $u,v\in SH^0(X)$ as  Borman-Sheridan classes for the respective compactifications, see \cite{To17}. The formula for the Chekanov torus  differs by interchanging $x$ and $y$. 
\end{ex}

\subsection{Proof of Theorem~\ref{th:seidel_q}}

Let $\bT=(T,\rho)$ an exact torus with a local system from the statement. Let $u,v\in SH^0(X)$ be the elements from Proposition~\ref{prp:co_wh}. Consider the symplectic potentials
$$
\begin{array}{l}
U(x,y)=\CO^{\rm string}_T(u),\\ V(x,y)=\CO^{\rm string}_T(v)\in \C[x^{\pm 1}, y^{\pm 1}].
\end{array}
$$
Viewing the local system on $T$ as a point $\rho\in (\C^*)^2$,
one has
$$
\begin{array}{l}
\CO_\bT(u)=U(\rho)\cdot 1_\bT,\\ \CO_\bT(v)=V(\rho)\cdot 1_\bT.
\end{array}
$$
Assume that $HF^*(L,\bT)\neq 0$ and pick any non-zero element
$$
0\neq z\in HF^*(L,\bT). 
$$
Recall that
$$
\begin{array}{l}
\CO_L(u)=0,\\ \CO_L(v)=0.
\end{array}
$$
Using the module map property of $\CO$ with respect to the Floer products
$$
\begin{array}{l}
\mu^2\co HF^*(L,L)\otimes HF^*(L,\bT )\to HF^*(L,\bT),\\ \mu^2\co HF^*(L,\bT)\otimes HF^*(\bT,\bT )\to HF^*(L,\bT),
\end{array} 
$$
one argues that
$$
0=\mu^2(\CO_L(u),z)=\mu^2(z,\CO_\bT(u))=\mu^2(z,U(\rho)\cdot 1_\bT)=U(\rho)\cdot z.
$$
Analogously, $0=V(\rho)\cdot z$. Since $z\neq 0$, this leads to
$$
U(\rho)=V(\rho)=0.
$$

Now, as $\CO^{\rm string}_T$ is an algebra map and $uv-1$ is invertible in $SH^0(X)$, the function
$$
U(x,y)V(x,y)-1
$$
must be invertible in the Laurent polynomial ring $\C[x^{\pm 1},y^{\pm 1}]$. Since $U,V$ are themselves Laurent polynomials, this is only possible when $UV-1$ is a monomial, that is:
$$
U(x,y)V(x,y)-1=cx^ay^b,\quad a,b\in\Z.
$$
It includes the possibility of $UV=0$.

\begin{rmk}
	This property holds true for the exact Clifford torus, where $U=(x+1)y$ and $V=y^{-1}$.
\end{rmk}

Now consider the following two cases separately. For the first case, assume that neither $U$ nor $V$ vanish identically. Then the zero-locus 
$$
Z=\{x,y\co U(x,y)V(x,y)=0\}\subset (\C^*)^2
$$
is a curve which is singular at the point $\rho$ since $U(\rho)=V(\rho)=0$. This includes the case when $U\equiv V$, when the zero locus is still singular scheme-theoretically.
On the other hand,  
$$
Z=\{x^ay^b-1=0\}
$$
is evidently everywhere smooth, which is a contradiction.

The second case is that either $U(x,y)\equiv 0$ or $V(x,y)\equiv 0$. There is a non-compactly supported symplectomorphism of the Liouville completion of $X$ whose action on $SH^0(X)$ swaps $u$ with $v$, so one may assume that $U(x,y)\equiv 0$. But $u\in SH^0(X)$ is the Borman-Sheridan class corresponding to the partial compactification $X\subset \C^2$, therefore $U(x,y)$ is in fact the Landau-Ginzburg potential of the monotone Lagrangian torus $T\subset \CP^2$ by \cite{To17}. 

Since $U(x,y)\equiv 0$, the Floer cohomology of $T$ in $\C^2$ has to be non-zero with any local system. On the other hand, every Lagrangian submanifold of $\C^2$ is displaceable by Hamiltonian isotopies, which implies that $T$ must have vanishing Floer cohomology therein. This is once again a contradiction; it completes the proof of Theorem~\ref{th:seidel_q}.\qed

\section{Outlook}
\label{sec:outlook}
The study of refined disk potentials and Floer theory for monotone immersed Lagrangians  can be seen as a basic example of Floer theory for singular Lagrangians. Ongoing work is developing similar Floer theories for symplectic embeddings of Weinstein manifolds, which can be thought of as regular neighborhoods of their skeleta. As here,  refined potentials and Floer theory have coefficients in endomorphisms of the wrapped Fukaya category of the domain. 

\begin{ex}
	The Weinstein manifold $X'=\C^3\setminus\{x_1x_2x_3=1\}$ is important for mirror symmetry: it is a building block for SYZ fibrations. Its skeleton is the following singular Lagrangian $L$: a 3-torus with a 2-sub-torus collapsed to a point. The link $\Lambda$ of this singularity is a ``Hopf link'' of two Legendrian 2-tori in $Y=S^5$, and the cap is $L^\circ=T^2\times [0,1]$. A computation of the capped Chekanov--Eliashberg algebra should correctly recover the ring of functions on the mirror $\{uv=1+z+w\}\subset \C^2\times (\C^*)^2$, and provide enough Lagrangian branes supported on $L$ to fill in all points of the mirror.
\end{ex}

We next discuss the origin of the capped Chekanov--Eliashberg algebra by considering cotangent bundles. Let $M$ be a smooth manifold and fix a Morse function $f\colon M\to \R$. Denote by $F$ a cotangent fiber of $T^*M$. 

One can compute $CW^{\ast}(F)$ by cutting $M$ near the maximum $c_{\rm top}$ value of $f$. More precisely, if $c$ is a regular value just below $c_{\rm top}$, then $M_{c}=f^{-1}((-\infty,c])$ is compact manifold with boundary. Its cotangent bundle $T^{\ast}M_{c}$ with rounded corners is a subcritical Weinstein subdomain of $T^*M$. The (ideal) boundary of $T^*M_c$ contains the  Legendrian sphere $\Lambda_{c}$, and Legendrian surgery states that
\[ 
CW^{\ast}(F)\cong CE(\Lambda_{c}).
\]    
The same wrapped Floer complex is also quasi-isomorphic to the chain complex on the based loop space of $M$ with the Pontryagin product:
\[ 
CW^{\ast}(F)\cong C_{\ast}(\Omega M).
\]
One can think of this quasi-isomorphism as provided by not cutting $M$ at all or, paraphrased in a strange way, cutting below the minimum of $f$. It turns out that one can compute $CW^*(F)$ in a similar way by cutting $M$ at any regular value of $f$. If $b$ is a regular value, denote $M_{b}=f^{-1}((-\infty,b])$.  The subcritical domain $T^*M_b$ with rounded corners contains the Legendrian $\Lambda_{b}=\partial M_{b}$. Denoting $M^{b}=f^{-1}([b,\infty))$, one can show that
\[ 
CW^{\ast}(F)\cong CE_{C_{\ast}(\Omega M^{b})}(\Lambda_{b}),
\]  
where the right hand side is the Chekanov--Eliashberg algebra of $\Lambda_b$ with coefficients in chains of the based loop space of the cap $M^{b}$. Similar results hold for singular Lagrangians, and we have used this idea in our definition of the refined potential. 

{There is also another perspective on our refined potentials that can be explained in analogy with the manifold case. Consider the case when we cut below the minimum and compute $CW^{\ast}(F)$ as $C_{\ast}(\Omega M)$. One can also try to compute this 'perturbatively': use the co-algebra which is the dual of the Floer cohomology $A_{\infty}$-algebra of the zero section $M$, $CF^{\ast}(M)'$ and apply the reduced cobar construction to get $\Omega CF^{\ast}(M)'$. By Adams famous result $\Omega CF^{\ast}(M)'$ is quasi-isomorphic to $C_{\ast}(\Omega M)$ if $M$ is simply connected, if not we may view $\Omega CF^{\ast}(M)'$ as a perturbative approximation of $C_{\ast}(\Omega M)$. 
	
Similarly, if we consider the Weinstein manifold that is a regular neighborhood of a self plumbed sphere (a Whitney sphere) and let $F$ be a fiber disk, then our $CE$-algebra $\C[p,q,t^{\pm 1}]/(1-t-pq)$, by the surgery result, is quasi-isomorphic to  $CW^{\ast}(F)$, whereas the cobar on the dual of the usual immersed Floer cohomology algebra of the self-intersecting sphere is a perturbative approximation.}

\bibliography{Symp_bib}
\bibliographystyle{plain}

\end{document}